\RequirePackage{etex}
\documentclass[12pt]{amsart}
\usepackage{amsmath, amsthm, amsfonts, amssymb, amscd}
\allowdisplaybreaks[4]
\usepackage[mathscr]{eucal}
\usepackage[dvips, usenames, dvipsnames]{xcolor}
\definecolor{orange2}{HTML}{FF7F2A}
\usepackage[all, knot, arc]{xy}
\usepackage[dvips, final]{graphicx}
\usepackage{pinlabel}
\usepackage{textcomp}
\usepackage{float}
\usepackage{enumitem}
\usepackage{tikz}
\usetikzlibrary{arrows,automata}
\usetikzlibrary{matrix,arrows}
\tikzset{cdlabel/.style={above,sloped,%
    execute at begin node=$\scriptstyle,execute at end node=$}}
\tikzset{algarrow/.style={->, thick}}   
\tikzset{alb/.style={->, bend right=25, thick}}
\tikzset{arb/.style={->, bend left=25, thick}}
\tikzset{al/.style={->, bend right=20, thick}}
\tikzset{ar/.style={->, bend left=20, thick}}
\tikzset{als/.style={->, bend right=15, thick}}
\tikzset{ars/.style={->, bend left=15, thick}}
\tikzset{blgarrow/.style={->, thick}}
\tikzset{clgarrow/.style={->, thick}}
\tikzset{tensoralgarrow/.style={double, double equal sign distance, -implies}}
\tikzset{tensorblgarrow/.style={double, double equal sign distance, -implies}}
\tikzset{tensorclgarrow/.style={double, double equal sign distance, -implies}}
\tikzset{tensorelgarrow/.style={double, double equal sign distance, -implies}}
\tikzset{modarrow/.style={->, dashed}}
\tikzset{Amodar/.style={->, dashed}}
\tikzset{Dmodar/.style={->, dashed}}
\tikzset{DAmodar/.style={->, dashed}}
\def\mathcenter#1{%
  \vcenter{\hbox{$#1$}}%
}
\usepackage{cite}
\usepackage{xifthen}
\usepackage{slashbox}
\usepackage{multirow}
\usepackage{makecell}
\usepackage{longtable}
\usepackage[para]{threeparttable}
\usepackage[referable]{threeparttablex}
\usepackage{afterpage}
\usepackage{adjustbox}
\usepackage{caption}
\usepackage{subcaption}
\usepackage{footnote}
\makesavenoteenv{tabular}
\makesavenoteenv{table}
\usepackage[final, colorlinks=true,%
  linkcolor=blue, citecolor=blue, urlcolor=blue]{hyperref}

\usepackage[lmargin=1in,rmargin=1in,tmargin=1in,bmargin=1in]{geometry}

\theoremstyle{definition}

\newtheorem{defn}[equation]{Definition}

\newtheorem*{rmk}{Remark}
\theoremstyle{plain}

\newtheorem{lem}[equation]{Lemma}
\newtheorem{prop}[equation]{Proposition}

\numberwithin{equation}{section}
\numberwithin{figure}{section}
\numberwithin{table}{section}
\newtheorem{ourtheorem}{Theorem}
\newtheorem{ourcorollary}[ourtheorem]{Corollary}
\newtheorem{ourconjecture}[ourtheorem]{Conjecture}

\newcommand{\set}[1]{\left\lbrace #1 \right\rbrace}
\newcommand{\setc}[2]{
  \left\lbrace #1 \, \middle\vert \, #2 \right\rbrace
}
\newcommand{\brac}[1]{\left( #1 \right)}
\newcommand{\sqbrac}[1]{\left[ #1 \right]}
\newcommand{\cbrac}[1]{\left\lbrace #1 \right\rbrace}
\newcommand{\comp}{\DOTSB\circ}

\newcommand{\vphi}{\varphi}
\renewcommand{\phi}{\vphi}
\newcommand{\eset}{\emptyset} % Or possibly \varnothing

\DeclareMathOperator{\id}{id}
\DeclareMathOperator{\Id}{Id}

\DeclareMathOperator{\Cone}{Cone}
\newcommand{\numset}[1]{\mathbb{#1}}

\newcommand{\Z}{\numset{Z}}
\newcommand{\F}[1]{\numset{F}_{#1}}
\newcommand{\field}[1]{\mathbf{#1}}
\renewcommand{\j}{\field{j}}
\renewcommand{\k}{\field{k}}
\newcommand{\cycgrp}[1]{\Z / #1 \Z}
\newcommand{\T}{T}
\newcommand{\eT}{T^{\mathit{el}}}

\newcommand{\ee}{e^{\mathit{el}}}

\DeclareMathOperator{\Neg}{neg}
\newcommand*{\blank}{\mathord{-}} % package textcomp
\newcommand{\opminus}[1]{#1^-}
\newcommand{\ophat}[1]{\widehat{#1}}
\newcommand{\optilde}[1]{\widetilde{#1}}
\newcommand{\opdelta}[1]{#1^\delta}
\newcommand{\opdeltaa}[1]{#1 {}^\delta}
\newcommand{\opunor}[1]{#1^u}
\newcommand{\opunorr}[1]{#1 {}^u}

\DeclareMathOperator{\HF}{HF}
\DeclareMathOperator{\CFK}{CFK}
\DeclareMathOperator{\HFK}{HFK}
\DeclareMathOperator{\Kh}{Kh}
\DeclareMathOperator{\CT}{CT}

\DeclareMathOperator{\CDTD}{CDTD}
\DeclareMathOperator{\CDTA}{CDTA}
\DeclareMathOperator{\CATD}{CATD}
\DeclareMathOperator{\CATA}{CATA}

\newcommand{\HFh}{\ophat{\HF}}

\newcommand{\CFKm}{\opminus{\CFK}}
\newcommand{\CFKh}{\ophat{\CFK}}

\newcommand{\HFKm}{\opminus{\HFK}}
\newcommand{\HFKh}{\ophat{\HFK}}

\newcommand{\Kht}{\optilde{\Kh}}

\newcommand{\CTt}{\optilde{\CT}}
\newcommand{\CTd}{\opdelta{\CTt}}
\newcommand{\CTu}{\opunorr{\CTu}}
\newcommand{\CDTDm}{\opminus{\CDTD}}
\newcommand{\CDTAm}{\opminus{\CDTA}}
\newcommand{\CATDm}{\opminus{\CATD}}
\newcommand{\CATAm}{\opminus{\CATA}}
\newcommand{\CDTDt}{\optilde{\CDTD}}
\newcommand{\CDTAt}{\optilde{\CDTA}}
\newcommand{\CATDt}{\optilde{\CATD}}

\newcommand{\CDTDd}{\opdeltaa{\CDTDt}}
\newcommand{\CDTAd}{\opdeltaa{\CDTAt}}
\newcommand{\CATDd}{\opdeltaa{\CATDt}}

\newcommand{\CDTDu}{\opunorr{\CDTDt}}
\newcommand{\CDTAu}{\opunorr{\CDTAt}}
\newcommand{\CATDu}{\opunorr{\CATDt}}

\newcommand{\Ainf}{\alg{A}_\infty}
\newcommand{\alg}[1]{\mathcal{#1}}
\newcommand{\module}[1]{\mathcal{#1}}
\newcommand{\moduletype}[1]{\ensuremath{\mathit{#1}}}
\newcommand{\DD}{\moduletype{DD}}
\newcommand{\DA}{\moduletype{DA}}
\newcommand{\AD}{\moduletype{AD}}
\renewcommand{\AA}{\moduletype{AA}}
\newcommand{\DDm}[1]{\module{#1}}

\newcommand{\Idm}[1]{\Id_{\DDm{#1}}}
\newcommand{\am}[1]{\opminus{\alg{A}} (#1)}
\newcommand{\im}[1]{\opminus{\alg{I}} (#1)}
\newcommand{\ah}[1]{\ophat{\alg{A}} (#1)}
\newcommand{\asub}[1]{\alg{A}_{#1}}
\newcommand{\ad}[1]{\opdelta{\ophat{\alg{A}}} (#1)}
\newcommand{\au}[1]{\opunor{\ophat{\alg{A}}} (#1)}
\newcommand{\ads}[1]{\opdelta{\asub{#1}}}
\newcommand{\aus}[1]{\opunor{\asub{#1}}}

\newcommand{\aun}{\opunor{\asub{n}}}
\newcommand{\as}[1]{\ads{#1}}
\newcommand{\an}{\aun}
\newcommand{\subsetIdem}[1]{\mathbf{#1}}
\newcommand{\sss}{\subsetIdem{s}}
\newcommand{\ttt}{\subsetIdem{t}}
\DeclareMathOperator{\inv}{inv}

\newcommand{\bdy}{\partial}
\DeclareMathOperator{\Int}{Int}
\newcommand{\HD}{\mathcal H}
\newcommand{\curves}[1]{\boldsymbol{#1}}
\newcommand{\alphas}[1][]{%
  \ifthenelse{\equal{#1}{}}{\curves{\alpha}}{\curves{\alpha^{#1}}}
}
\newcommand{\betas}[1][]{%
  \ifthenelse{\equal{#1}{}}{\curves{\beta}}{\curves{\beta_{#1}}}
}
\newcommand{\position}[2]{#1^#2}
\newcommand{\FL}{\mathit{FL}}
\newcommand{\FR}{\mathit{FR}}
\newcommand{\BL}{\mathit{BL}}
\newcommand{\BR}{\mathit{BR}}

\newcommand{\leftbdy}[1]{\position{#1}{L}}
\newcommand{\rightbdy}[1]{\position{#1}{R}}
\newcommand{\frontleftbdy}[1]{\position{#1}{{\FL}}}
\newcommand{\frontrightbdy}[1]{\position{#1}{{\FR}}}
\newcommand{\backleftbdy}[1]{\position{#1}{\BL}}
\newcommand{\backrightbdy}[1]{\position{#1}{\BR}}
\newcommand{\onalphas}[1]{\position{#1}{\alpha}}
\newcommand{\onbetas}[1]{\position{#1}{\beta}}

\newcommand{\bdyL}{\leftbdy{\bdy}}
\newcommand{\bdyR}{\rightbdy{\bdy}}
\newcommand{\bdya}{\onalphas{\bdy}}
\newcommand{\bdyb}{\onbetas{\bdy}}
\newcommand{\bdyFL}{\frontleftbdy{\bdy}}
\newcommand{\bdyFR}{\frontrightbdy{\bdy}}
\newcommand{\bdyBL}{\backleftbdy{\bdy}}
\newcommand{\bdyBR}{\backrightbdy{\bdy}}
\newcommand{\markers}[1]{\mathbb{#1}}
\newcommand{\OO}{\markers{O}}
\newcommand{\XX}{\markers{X}}
\newcommand{\YY}{\markers{Y}}

\newcommand{\ocl}[1]{\leftbdy{o} (#1)}
\newcommand{\ocr}[1]{\rightbdy{o} (#1)}
\newcommand{\unocl}[1]{\overline{\leftbdy{o}} (#1)}
\newcommand{\unocr}[1]{\overline{\rightbdy{o}} (#1)}
\newcommand{\ideml}[1]{\leftbdy{e}_D (#1)}
\newcommand{\idemr}[1]{\rightbdy{e}_D (#1)}
\newcommand{\algl}[1]{\leftbdy{a} (#1)}
\newcommand{\algr}[1]{\rightbdy{a} (#1)}
\newcommand{\gen}[1]{\mathbf{#1}}
\newcommand{\genset}[1]{\mathfrak{#1}}
\newcommand{\subgenset}[1]{\mathbf{#1}}
\newcommand{\x}{\gen{x}}
\newcommand{\y}{\gen{y}}
\newcommand{\z}{\gen{z}}
\newcommand{\gt}{\gen{t}}
\renewcommand{\SS}{\genset{S}}
\newcommand{\II}{\subgenset{I}}
\newcommand{\NN}{\subgenset{N}}
\DeclareMathOperator{\Rect}{Rect}
\DeclareMathOperator{\Tri}{Tri}
\DeclareMathOperator{\Quad}{Quad}
\DeclareMathOperator{\Pent}{Pent}
\DeclareMathOperator{\Hex}{Hex}
\DeclareMathOperator{\Hept}{Hept}
\newcommand{\emptypoly}[2][]{%
  #2^{\ifthenelse{\equal{#1}{}}{\circ}{\circ, #1}}
}
\newcommand{\eRect}{\emptypoly{\Rect}}
\newcommand{\eTri}{\emptypoly{\Tri}}
\newcommand{\ePent}{\emptypoly{\Pent}}
\newcommand{\eQuad}{\emptypoly{\Quad}}
\newcommand{\eHex}{\emptypoly{\Hex}}
\newcommand{\eHept}{\emptypoly{\Hept}}
\newcommand{\eRectI}{\emptypoly[\Int]{\Rect}}
\newcommand{\eRectL}{\emptypoly[L]{\Rect}}
\newcommand{\eRectR}{\emptypoly[R]{\Rect}}

\newcommand{\ePentL}{\emptypoly[L]{\Pent}}

\newcommand{\eHexI}{\emptypoly[\Int]{\Hex}}

\newcommand{\eHexR}{\emptypoly[R]{\Hex}}

\newcommand{\poly}[1]{\mathscr{#1}}

\newcommand{\PP}{\poly{P}}
\newcommand{\TT}{\poly{T}}
\newcommand{\HH}{\poly{H}}
\newcommand{\Q}{\poly{Q}}
\newcommand{\K}{\poly{K}}
\newcommand{\nan}{\mathtt{o}}
\newcommand{\lan}{\mathtt{s}}
\newcommand{\ran}{\mathtt{d}}
\newcommand{\lran}{\mathtt{sd}}
\newcommand{\htd}[2]{\xrightarrow{(#1, #2)}}
\DeclareMathOperator{\NM}{NM}
\newcommand{\htdnm}{\xrightarrow{\NM}}
\newcommand{\htdid}{\xrightarrow{\Id}}
\newcommand{\mr}[2]{\multirow{#1}{*}{#2}}
\newcommand{\vs}[1]{\rule{0pt}{#1 pt}}
\newcommand{\mra}[3]{{\mr{#1}{\vs{#2} #3 \hspace*{0.0pt}}}}
\newcommand{\str}[1]{r_{#1}}
\DeclareMathOperator{\hgt}{ht}
\newcommand{\numOO}[1]{n_{\OO} (#1)}
\newcommand{\numXX}[1]{n_{\XX} (#1)}
\newcommand{\eiota}{\iota^{\mathit{el}}}
\newcommand{\orf}{f^{\mathit{or}}}
\newcommand{\degd}{\deg_\delta}
\newcommand{\Ppm}{P_{+,-}}
\newcommand{\Ppmt}{\widetilde{P}_{+,-}}
\DeclareMathOperator{\rk}{rk}

\begin{document}
\title{Skein relations for tangle Floer homology}
\author{Ina Petkova}
\address {Department of Mathematics, Dartmouth College, Hanover, NH 03755, 
  USA}
\email{\href{mailto:ina.petkova@dartmouth.edu}{ina.petkova@dartmouth.edu}}
\urladdr{\url{http://www.math.dartmouth.edu/~ina/}}
\author{C.-M.~Michael Wong}
\address {Department of Mathematics, Louisiana State University, Baton Rouge, LA 70803, USA}
\email{\href{mailto:cmmwong@lsu.edu}{cmmwong@lsu.edu}}
\urladdr{\url{http://www.math.lsu.edu/~cmmwong/}}
\begin{abstract}
In a previous paper, V\'ertesi and the first author used grid-like Heegaard diagrams to define tangle Floer homology, which associates to a tangle $\T$ a differential graded bimodule $\CTt (\T)$. If $L$ is obtained by gluing together $\T_1, \dotsc, \T_m$, then the knot Floer homology $\HFKh (L)$ of $L$ can be recovered from $\CTt (\T_1), \dotsc, \CTt (\T_m)$. In the present paper, we prove combinatorially that tangle Floer homology satisfies unoriented and oriented skein relations, which are analogues of the skein exact triangles for knot Floer homology.
\end{abstract}
\maketitle

%%%%%%%%%%%%%%%%%%%%%%%%%%%%%%%%%%%%%%%%%%%%%%%%

%%%%%%%%%%%%%%%%%%%%%%%%%%%%%%%%%%%%%%%%%%%%%%%%%%%%%%%
% !TEX root = ../skein.tex
%%%%%%%%%%%%%%%%%%%%%%%%%%%%%%%%%%%%%%%%%%%%%%%%%%%%%%%

\section{Introduction} % (fold)
\label{sec:introduction}

%%%%%%%%%%%%%%%%%%%%%%%%%%%%%%%%%%%%%%%%%%%%%%%%%%%%%%%

%%%%%%%%%%%%%%%%%%%%%%%%%%%%%%%%%%%%%%%%%%%%%%%%%%%%%%%
% section introduction 
%%%%%%%%%%%%%%%%%%%%%%%%%%%%%%%%%%%%%%%%%%%%%%%%%%%%%%%

\emph{Heegaard Floer homology} is an invariant of closed, oriented $3$-manifolds introduced in \cite{osz8} that has found numerous applications in recent years, and is known to be equivalent \cite{klt1, klt2, klt3, klt4, klt5} to monopole Floer homology \cite{km}, and also equivalent \cite{cgh} to embedded contact homology \cite{hut, ht1, ht2}. In \cite{hfk, jrth}, it is extended to give an invariant, \emph{knot Floer homology}, of null-homologous knots in closed, oriented $3$-manifolds, which is further generalized to oriented links in \cite{oszlink}. Knot Floer homology comes in many flavors; its simplest form, $\HFKh (L)$ for an oriented link $L$, is a bigraded module over $\F{2}$ or $\Z$. There is a combinatorial description of the knot Floer homology of links $L \subset S^3$ called \emph{grid homology} \cite{mos, most, OSSbook}, defined using \emph{grid diagrams}. Because knot Floer homology, defined analytically, is known to categorify the Alexander polynomial, it is often compared with Khovanov homology $\Kht (L)$ \cite{kh1, kh3}, a link invariant from representation theory that categorifies the Jones polynomial.

Ozsv\'ath and Szab\'o \cite{bdc} show that the Heegaard Floer homology $\HFh (- \Sigma (L))$ of the branched double cover of a link $L$ satisfies an \emph{unoriented skein exact triangle}, from which they derive a spectral sequence from $\Kht (L)$ to $\HFh (- \Sigma (L))$, thus relating the two theories. Following this, Manolescu \cite{mskein}, by counting holomorphic polygons, shows that knot Floer homology with $\F{2}$ coefficients also satisfies an unoriented skein exact triangle, and uses it to show that $\operatorname{rk} \HFKh (L; \F{2}) = 2^{\ell-1} \det (L)$ for a quasi-alternating link $L$ with $l$ components. Manolescu and Ozsv\'ath \cite{quasi} then use the skein exact sequence to show that quasi-alternating links are \emph{Floer-homologically $\sigma$-thin} over $\F{2}$.

While the discussion above seems to suggest that there may be a spectral sequence relating $\Kht (L)$ and $\HFKh (L)$ that comes from iterating Manolescu's skein relation, Baldwin and Levine \cite{bl} discover that the $E_2$ page of the spectral sequence they so construct is not even a link invariant. However, one may be able to relate the two theories with some modifications: Baldwin, Levine, and Sarkar \cite{bls} construct another spectral sequence that converges to $\HFKh (L) \otimes V^n$ for some module $V$ of rank $2$ and some integer $n$, where the differential $D^0$ counts some of the holomorphic polygons in Manolescu's unoriented skein sequence. They conjecture that the $E_1$ page of this spectral sequence coincides with a variant of Khovanov homology for pointed links, the proof of which would imply a version of the following conjecture, first formulated by Rasmussen \cite{rasconj} for knots:

\begin{ourconjecture}
  \label{conj:rasconj}
  For any $\ell$-component link $L \subset S^3$, we have
  \[
    2^{\ell - 1} \rk \Kht (L) \geq \rk \HFKh (L).
  \]
\end{ourconjecture}

\begin{rmk}
  \label{rmk:dowlin}
  During the peer review process of this article, Dowlin \cite{dowlin} has announced the existence of a spectral sequence relating $\Kht (L)$ and $\HFKh (L)$, and hence a proof of Conjecture~\ref{conj:rasconj}. How this spectral sequence compares with the candidate constructed by Baldwin, Levine, and Sarkar \cite{bls} remains unknown.
\end{rmk}

To better understand Manolescu's skein relation and the related conjectures, there are several approaches. One idea involves computing the maps in the skein relation combinatorially: The second author \cite{wong} gives a version of the skein sequence for grid homology, generalizing the results on quasi-alternating links \cite{mskein, quasi} to $\Z$-coefficients and giving a spectral sequence from a cube-of-resolutions complex with no diagonal maps. Lambert-Cole \cite{lc} exploits the computability in \cite{wong} to show that $\delta$-graded knot Floer homology is invariant under Conway mutation by a large class of tangles.

Another idea, suggested to the authors by Levine \cite{adampc}, is to understand the maps in the skein relation on a local level, by slicing the links involved into tangles and studying a tangle version of knot Floer homology. One such theory is \emph{tangle Floer homology}, defined by V\'ertesi and the first author \cite{pv}. In this theory, to a sequence of points one associates a differential graded algebra, and to a tangle $\T \subset I \times \mathbb{R}^2$ one associates an $\Ainf$-module $\CTt (\T)$ over the differential graded algebra(s) associated to its boundary. If a link $L$ is obtained by gluing together tangles $\T_1, \dotsc, \T_m$, then $\HFKh (L)$ can be recovered by taking a suitable notion of tensor product, called the \emph{box tensor product}, of $\CTt (\T_1), \dotsc, \CTt (\T_m)$. The $\Ainf$-modules $\CTt (\T)$ are defined combinatorially using \emph{nice diagrams} (in the sense of Sarkar and Wang \cite{sw}) that are similar to grid diagrams. The tangle Floer package is inspired by \emph{bordered Floer homology}, an invariant of $3$-manifolds with parametrized boundary that can be used to recover the Heegaard Floer homology of a manifold obtained by gluing, defined by Lipshitz, Ozsv\'ath, and Thurston \cite{bfh2}.

Similar to knot Floer homology, tangle Floer homology also comes in multiple flavors. For example, $\CDTDu (\T, n)$ is an ungraded \emph{type~$\DD$ structure}, $\CDTDd (\T, n)$ is a $\delta$-graded type~$\DD$ structure, and $\CDTDt (\T, n)$ is an $(M, A)$-bigraded type~$\DD$ structure, where $M$ and $A$ are the Maslov and Alexander gradings respectively. As the notation suggests, these structures do not depend on the choice of Heegaard diagram $\HD$ for $\T$, but only on the number of markers in $\HD$, which we denote by $n$. There is also a richer bigraded version, $\CDTDm (\HD)$, which recovers the richer version of knot Floer homology $\HFKm (L)$, and which is believed but not yet proven to be an invariant of $\T$. We postpone the precise definitions of these, as well as other \emph{type~$\DA$}, \emph{type~$\AD$}, and \emph{type~$\AA$} structures, to Section~\ref{sec:background}.

The first part of this paper addresses the idea above; namely, we prove an unoriented skein relation for tangle Floer homology. Suppose $\T_\infty$, $\T_0$, and $\T_1$ are three unoriented tangles in $I \times \mathbb{R}^2$ identical except near a point, as in Figure~\ref{fig:T_inf01}.
\begin{figure}[h]
  \centering
  \labellist
  \pinlabel  $\T_\infty$ at 25 55
  \pinlabel $\T_0$ at 152 55
  \pinlabel $\T_1$ at 280 55
  \endlabellist
  \includegraphics[scale=0.97]{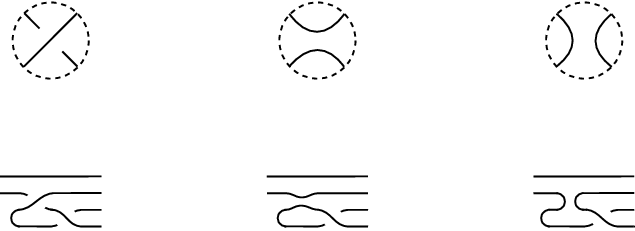}
  \caption{Top: Three tangles $\T_\infty$, $\T_0$, $\T_1$ form an unoriented skein triple if they are identical except near a point, as displayed. Bottom: A specific example of an unoriented skein triple.}
  \label{fig:T_inf01}
\end{figure}

\begin{ourtheorem}
  \label{thm:ourtheorem}
  There exists a type~$\DD$ homomorphism $F_0 \colon \CDTDu (\T_0, n) \to \CDTDu (\T_1, n)$ such that
  \[
    \CDTDu (\T_\infty, n) \simeq \Cone (F_0 \colon \CDTDu (\T_0, n) \to \CDTDu (\T_1, n))
  \]
  as ungraded type~$\DD$ structures. Analogous statements hold for type~$\DA$, $\AD$, and $\AA$ structures.
\end{ourtheorem}

In fact, we prove a strengthened version of Theorem~\ref{thm:ourtheorem} for oriented tangles, taking into account the $\delta$-grading. Suppose $\T_\infty$, $\T_0$, and $\T_1$ are three tangles as above, but oriented, and choose corresponding oriented planar diagrams that are identical (after forgetting the orientations) except near a point. Let $\Neg (\T_k)$ denote the number of negative crossings in the diagram for $\T_k$, and let $e_0 = \Neg (\T_1) - \Neg (\T_0)$ and $e_1 = \Neg (\T_\infty) - \Neg (\T_1)$.

\begin{ourtheorem}
  \label{thm:gradings}
  There exists a type~$\DD$ homomorphism $F_0 \colon \CDTDd (\T_0, n) \to \CDTDd (\T_1, n)$ of $\delta$-degree $(e_0 - 1)/2$ such that
  \[
    \CDTDd (\T_\infty, n) \simeq \Cone (F_0 \colon \CDTDd (\T_0, n) \to \CDTDd (\T_1, n)) \sqbrac{\frac{e_1-1}{2}}
  \]
  as $\delta$-graded type~$\DD$ structures. Analogous statements hold for type~$\DA$, $\AD$, and $\AA$ structures.
\end{ourtheorem}

\begin{rmk}
  Following \cite{OSSbook, pv}, our $\delta$-gradings differ from those in \cite{quasi, wong} by a factor of $-1$.
\end{rmk}

By taking the box tensor product, we immediately obtain a combinatorially computable unoriented skein exact triangle for knot Floer homology analogous to those in \cite{mskein, wong}. Suppose $L_\infty$, $L_0$, and $L_1$ are three oriented links that are identical (after forgetting the orientations) except near a point, so that they form an unoriented skein triple. Let $\ell_\infty$, $\ell_0$, and $\ell_1$ be the number of components of $L_\infty$, $L_0$, and $L_1$ respectively, and define $\Neg (L_k)$, $e_0$, and $e_1$ in a fashion analogous to $\Neg (\T_k)$, $e_0$, and $e_1$ above.

\begin{ourcorollary}
  \label{cor:hfk}
  For sufficiently large $m$, there exists a $\delta$-graded exact triangle
  \begin{align*}
    \dotsb & \to \HFKh_* (L_1; \F{2}) \otimes V^{m-\ell_1} \otimes W \to \HFKh_{*+\frac{e_1-1}{2}} (L_\infty; \F{2}) \otimes V^{m-\ell_\infty} \otimes W\\
    & \to \HFKh_{*-\frac{e_0+1}{2}} (L_0; \F{2}) \otimes V^{m-\ell_0} \otimes W \to \HFKh_{*-1} (L_1; \F{2}) \otimes V^{m-\ell_1} \otimes W \to \dotsb,
  \end{align*}
  where $V$ is a vector space of dimension $2$ with grading $0$, and $W$ is a vector space of dimension $2$ with grading $-1$.
\end{ourcorollary}

\begin{rmk}
  Due to a difference in the orientation convention, the arrows in the exact triangle point in the opposite direction from those in \cite{mskein, quasi}. We follow the convention in \cite{hfk, mos, most, wong}, where the Heegaard surface is the oriented boundary of the $\alpha$-handlebody.
\end{rmk}

\begin{rmk}
  Technically, we do not show that the exact triangle in Corollary~\ref{cor:hfk} agrees with the ones in \cite{mskein} and \cite{wong}, which themselves are not known to coincide. However, we do expect all three to agree.
\end{rmk}

Parallel to the above, we also prove an oriented skein relation for tangle Floer homology in the second part of this paper, which can be viewed as a local analogue of the oriented skein relation for knot Floer homology proven by Ozs\'vath and Szab\'o \cite{hfk, oszskein}. While formally similar to the unoriented skein relation, we pursue this direction for a slightly different reason---we do so with a view towards the further development of knot Floer homology in the framework of categorification.

Precisely, tangle Floer homology has been shown by Ellis, V\'ertesi, and the first author \cite{epv} to categorify the Reshetikhin--Turaev invariant for the quantum group $U_q (\mathfrak{gl}_{1|1})$. This puts tangle Floer homology on a similar footing as the tangle formulation of Khovanov homology \cite{kh3, chkh, bs}, which categorifies the Reshetikhin--Turaev invariant for $U_q (\mathfrak{sl}_2)$. What is missing in the work of Ellis, V\'ertesi, and the first author is a construction of $2$-morphisms, corresponding to tangle cobordisms. For knot Floer homology, cobordism maps are defined by Juh\'asz \cite{hfkcob} using contact geometry, and independently by Zemke \cite{zemke} using elementary cobordisms, and together they \cite{jz} show that their definitions coincide. Juh\'asz and Marengon \cite{jm} prove that the cobordism maps in \cite{hfkcob} fit into a skein exact triangle, providing evidence that these cobordism maps are actually the maps in skein sequences. Thus, one approach to constructing the $2$-morphisms mentioned above is to study the skein relations of tangle Floer homology further.

To state our results, suppose $\eT_+$, $\eT_-$, and $\eT_0$ are three oriented elementary tangles identical except near a point, with the strands at which the tangles differ oriented from right to left, as in Figure~\ref{fig:Tel_ori_intro}.
\begin{figure}[h]
  \centering
  \includegraphics[scale=1.05]{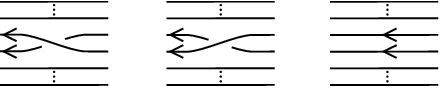}
  \caption{From left to right, the elementary tangles $\eT_+$, $\eT_-$, $\eT_0$.}
  \label{fig:Tel_ori_intro}
\end{figure}
There are corresponding Heegaard diagrams $\HD_+$, $\HD_-$, and $\HD_0$, which we describe explicitly in Section~\ref{sec:oriented}. Below, $U_1$ and $U_2$ are variables corresponding to the strands at which the tangles differ. 

\begin{ourtheorem}
  \label{thm:oriented}
  There exists a type~$\DD$ homomorphism $\Ppm \colon \CDTDm (\HD_+) \to \CDTDm (\HD_-)$ of $(M, A)$-degree $(0,0)$ such that
  \[
    \Cone (\Ppm) \simeq \Cone (\Id_{\CDTDm (\HD_0)} \otimes (U_2 - U_1) \colon \CDTDm (\HD_0) \to \CDTDm (\HD_0)) \sqbrac{1} \cbrac{\frac{1}{2}}
  \]
  as $(M, A)$-bigraded type~$\DD$ structures. Analogous statements hold for type~$\DA$, $\AD$, and $\AA$ structures.
\end{ourtheorem}

\begin{rmk}
  Since $\CDTDm$ is not yet known to be a tangle invariant \cite{pv}, Theorem~\ref{thm:oriented} is stated for the type~$\DD$ bimodules of Heegaard diagrams rather than for bimodules associated to a tangle.
\end{rmk}

\begin{rmk}
  Tangle Floer homology is currently only defined over $\F{2}$, and so the negative signs in Theorem~\ref{thm:oriented} could be replaced by positive signs. However, the stated signs are what one would expect for a theory defined over $\Z$. This remark applies also to Lemma~\ref{lem:comm_diag} and Lemma~\ref{lem:h_is_u}.
\end{rmk}

Restricting to $\CDTDt$, we also obtain a local oriented skein relation for that version. In this case, we have a proven tangle invariant, so the relation holds for general tangles. Suppose $\T_+$, $\T_-$, and $\T_0$ are three tangles that form an oriented skein triple, as in Figure~\ref{fig:T_+-0}.
\begin{figure}[h]
  \centering
  \labellist
  \pinlabel  $\T_+$ at 25 55
  \pinlabel $\T_-$ at 152 55
  \pinlabel $\T_0$ at 280 55
  \endlabellist
  \includegraphics[scale=0.97]{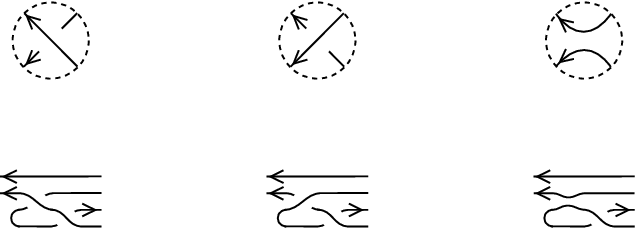}
  \caption{Top: Three tangles $\T_+$, $\T_-$, $\T_0$ form an oriented skein triple if they are identical except near a point, as displayed. Bottom: A specific example of an oriented skein triple.}
  \label{fig:T_+-0}
\end{figure}

\begin{ourtheorem}
  \label{thm:oriented-tilde}
  There exists a type~$\DD$ homomorphism $\Ppmt \colon \CDTDt (\T_+, n) \to \CDTDt (\T_-, n)$ of $(M, A)$-degree $(0,0)$ such that
  \[
    \Cone ( \Ppmt) \simeq  \CDTDt (\T_0, n) \sqbrac{0} \cbrac{-\frac{1}{2}} 
    \oplus \CDTDt (\T_0, n)\sqbrac{1} \cbrac{\frac{1}{2}}
  \]
  as $(M, A)$-bigraded type~$\DD$ structures. Analogous statements hold for type~$\DA$, $\AD$, and $\AA$ structures.
\end{ourtheorem}

Again by taking the box tensor product, we obtain an oriented skein exact triangle for knot Floer homology analogous to those in \cite{hfk, oszskein}. Suppose $L_+$, $L_-$, and $L_0$ are three oriented links that are identical except near a point, so that they form an oriented skein triple.

\begin{ourcorollary}
  \label{cor:hfk-oriented}
  If the two strands of $L_0$ belong to the same component, then there exists an $(M, A)$-bigraded exact triangle
  \begin{align*}
    \dotsb & \to \HFKm_m (L_+, s) \otimes W \to \HFKm_m (L_-, s) \otimes W\\
    & \to (\HFKm (L_0) \otimes V \otimes W)_{m-1, s} \to \HFKm_{m-1} (L_+, s) \otimes W \to \dotsb,
  \end{align*}
  and if the two strands of $L_0$ belong to different components, then there exists an $(M, A)$-bigraded exact triangle
  \begin{align*}
    \dotsb & \to \HFKm_m (L_+, s) \otimes W \to \HFKm_m (L_-, s) \otimes W\\
    & \to H_{m-1} ((\CFKm (L_0) / (U_2 - U_1)) \otimes W, s) \to \HFKm_{m-1} (L_+, s) \otimes W \to \dotsb,
  \end{align*}
  where $m$ and $s$ are the Maslov and Alexander gradings, respectively,  $V$ is a module of rank $2$ with bigradings $(0,0)$ and $(1,1)$, and $W$ is a module of rank $2$ with bigradings $(0,0)$ and $(-1,0)$. Analogous statements hold for $\HFKh$.
\end{ourcorollary}

\begin{rmk}
  Similar to the unoriented skein exact sequence, we expect but do not prove that the sequence in Corollary~\ref{cor:hfk-oriented} coincides with the ones in \cite{hfk, oszskein}.
\end{rmk}

To put Theorem~\ref{thm:oriented-tilde} in the context of categorification discussed above, we briefly outline the Reshetikhin--Turaev construction, specializing to the case yielding the Alexander polynomial.  To the boundaries of an oriented tangle  one associates a tensor product of copies of the standard $U_q (\mathfrak{gl}_{1|1})$-representation $V$ and its dual $V^*$, and to the tangle a map between these representations. The construction is combinatorial, decomposing a diagram $D$ for a tangle $T$ into elementary pieces (cups, caps, and crossings), assigning morphisms to the elementary pieces and defining $Q(D)$ as the composition of these morphisms. The map $Q(D)$ is an isotopy invariant of the oriented tangle.  Further, for the triple of oriented elementary tangles $\T_+$, $\T_-$, and $\T_0$ from Figure~\ref{fig:Tel_ori_intro}, $Q$ satisfies the skein relation
\begin{equation}
  \label{eqn:gl11-skein}
  Q (\T_+) - Q (\T_-) = (q - q^{-1}) Q (\T_0).
\end{equation}
The ground ring here is $\mathbb C(q)$. If we present a link $L$ as a $(1,1)$-tangle $T_L$ and set $q^2=t$, then it turns out that $Q (T_L) = \Delta(L) \id_V$, where $\Delta(L)$ is the Alexander polynomial of $L$.

In \cite{epv}, it is shown that tangle Floer homology categorifies the construction described above---tensor products of copies $V$ and $V^*$ lift to categories of left type $D$ modules over the dg algebras associated to the boundaries of $T$, and the map $Q(T)$ lifts to the functor $\CDTAt (T)\boxtimes \blank$. In particular, $Q(\eT_+)$, $Q(\eT_-)$, and $Q(\eT_0)$  lift to  $\CDTAt(\eT_+)$, $\CDTAt(\eT_-)$, and $\CDTAt(\eT_0)$, respectively. With this set-up, we have:

\begin{ourcorollary}
  \label{cor:gl11}
  The homotopy equivalence in Theorem~\ref{thm:oriented-tilde} categorifies the skein relation in Equation~\ref{eqn:gl11-skein}.
\end{ourcorollary}

Since tangle Floer homology shares some similarities with grid homology, our approach to proving Theorems~\ref{thm:ourtheorem} and \ref{thm:gradings} is similar to that in \cite{wong}, and our approach to proving Theorem~\ref{thm:oriented} is similar to that in \cite[Chapter 9]{OSSbook}. In particular, all maps involved are combinatorially computable.

\subsection*{Organization}

We review the necessary algebraic background and the definition of tangle Floer homology in Section~\ref{sec:background}. We prove the ungraded unoriented skein relation, Theorem~\ref{thm:ourtheorem}, in Section~\ref{sec:proof}. We then determine the $\delta$-gradings in Section~\ref{sec:grading} to prove the graded skein relation, Theorem~\ref{thm:gradings}. Theorems~\ref{thm:oriented} and \ref{thm:oriented-tilde} are proven in Section~\ref{sec:oriented}.

\subsection*{Acknowledgments}

The authors thank Robert Lipshitz and Vera V\'ertesi for useful conversations. The authors are also grateful to Robert Lipshitz and the anonymous referee for many insightful comments and corrections on previous drafts. IP received support from an AMS-Simons travel grant and NSF Grant DMS-1711100. Part of the research was conducted while IP and MW were affiliated with Columbia University. IP thanks Louisiana State University, and MW thanks Rice University and Dartmouth College for their hospitality while this research was undertaken.

%%%%%%%%%%%%%%%%%%%%%%%%%%%%%%%%%%%%%%%%%%%%%%%%%%%%%%%

%%%%%%%%%%%%%%%%%%%%%%%%%%%%%%%%%%%%%%%%%%%%%%%%%%%%%%%
% !TEX root = ../skein.tex
%%%%%%%%%%%%%%%%%%%%%%%%%%%%%%%%%%%%%%%%%%%%%%%%%%%%%%%

\section{Background} % (fold)
\label{sec:background}

%%%%%%%%%%%%%%%%%%%%%%%%%%%%%%%%%%%%%%%%%%%%%%%%%%%%%%%

%%%%%%%%%%%%%%%%%%%%%%%%%%%%%%%%%%%%%%%%%%%%%%%%%%%%%%%
% section preliminaries
%%%%%%%%%%%%%%%%%%%%%%%%%%%%%%%%%%%%%%%%%%%%%%%%%%%%%%%

\subsection{Algebraic structures}\label{ssec:alg}

We first review the underlying algebraic structures of tangle Floer homology. We will only define the immediately relevant structures here, and refer the interested reader to \cite[Section~2]{bimod}.

Let $A$ be a unital differential graded algebra (DGA) with differential $d$ and multiplication $\mu$ over a base ring $\k$ of characteristic $2$. In this paper, $\k$ will always be the ring of idempotents, which is a direct sum of copies of $\F{2} = \Z/2\Z$. We will also write $a \cdot b$ to denote $\mu (a, b)$ for algebra elements $a, b \in A$, whenever no confusion can arise.

A \emph{(left) type $D$ structure over $A$} is a graded $\k$-module $M$ equipped with a homogeneous map
\[
  \delta^1 \colon M \to (A \otimes M) [1]
\]
satisfying the compatibility condition
\[
  (d \otimes \id_M) \circ \delta^1 + (\mu \otimes \id_M) \circ (\id_A \otimes \delta^1) \circ \delta^1 = 0.
\]

It may be advantageous to represent this graphically:
\begin{equation*}
  \adjustimage{scale = 0.753, valign = m}{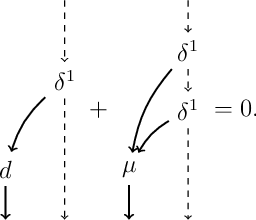}
\end{equation*}
The map $\delta^1$ is called the \emph{structure map} of $M$. Defining
\[
  \delta^i \colon M \to (A^{\otimes i} \otimes M) [i]
\]
inductively by
\[
  \delta^i =
  \begin{cases}
    \id_M & \text{for $i = 0$},\\
    (\id_A \otimes \delta^{i-1}) \circ \delta^1 & \text{for $i \geq 1$}
  \end{cases},
\]
we say that $M$ is \emph{bounded} if for all $x \in M$, there exists an integer $n = n (x)$ such that $\delta^i (x) = 0$ for all $i > n$. 

Let $A$ and $B$ be two unital differential graded algebras, with differentials $d_A$ and $d_B$, and multiplications $\mu_A$ and $\mu_B$, over the base rings $\k$ and $\j$ respectively. (Recall that the base rings have characteristic $2$.) A
\emph{(left-right) type $\DD$ structure over $(A, B)$} is a graded $(\k,
\j)$-bimodule $M$ equipped with a homogeneous structure map
\[
  \delta^1 \colon M \to (A \otimes M \otimes B) [1]
\]
satisfying the compatibility condition
\[
  (d_A \otimes \id_M \otimes \id_B) \circ \delta^1 + (\id_A \otimes \id_M
  \otimes d_B) \circ \delta^1 + (\mu_A \otimes \id_M \otimes \mu_B) \circ
  (\id_A \otimes \delta^1 \otimes \id_B) \circ \delta^1 = 0.
\]
Graphically, this can be represented as:
\begin{equation*}
  \adjustimage{scale = 0.753, valign = m}{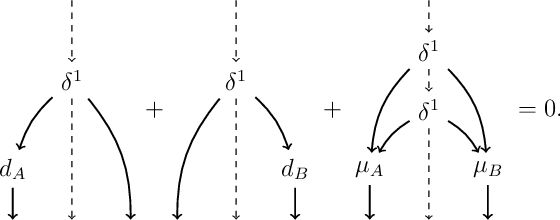}
\end{equation*}
Like for type~$D$ structures, we can define $\delta^i$ and the notion of \emph{boundedness} analogously. Type~$\DD$ structures are the main objects of study in this paper.  We will denote type~$\DD$ structures by calligraphic letters (e.g.\ $\DDm{M, N}$), and reserve $M, N$ for the underlying $(\k, \j)$-bimodules.

A \emph{morphism $f \colon \DDm{M} \to \DDm{N}$ of degree $\ell$} is simply a module homomorphism
\[
  f \colon M \to (A \otimes N \otimes B) [-\ell].
\]
(By abuse of notation, we use $f$ to denote both maps above.) Given a morphism, we can define its \emph{boundary} $\partial f \colon M \to (A \otimes N \otimes B) [- \ell + 1]$ by
\begin{align*}
  \partial f = & (\mu_A \otimes \id_N \otimes \mu_B) \circ (\id_A \otimes \delta_N^1 \otimes \id_B) \circ f + (\mu_A \otimes \id_N \otimes \mu_B) \circ (\id_A \otimes f \otimes \id_B) \circ \delta_M^1\\*
  & + (d_A \otimes \id_N \otimes \id_B) \circ f + (\id_A \otimes \id_N \otimes d_B) \circ f,
\end{align*}
or graphically,
\begin{equation*}
  \adjustimage{scale = 0.753, valign = m}{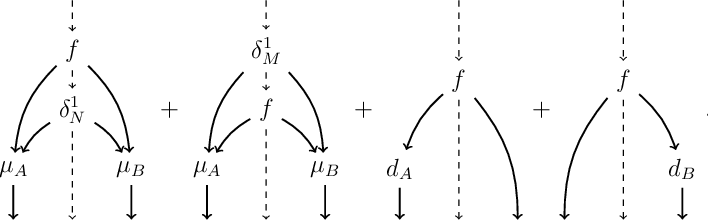}
\end{equation*}
For convenience, although this is not found in the literature, we will use the notation $d f \colon M \to (A \otimes N \otimes B) [-\ell + 1]$ to represent the last two terms above:
\[
  d f = (d_A \otimes \id_N \otimes \id_B) \circ f + (\id_A \otimes \id_N \otimes d_B) \circ f.
\]

Given two morphisms $f \colon \DDm{M} \to \DDm{N}$ of degree $\ell_1$ and $g \colon \DDm{N} \to \DDm{P}$ of degree $\ell_2$, where $\DDm{M}$, $\DDm{N}$, $\DDm{P}$ are type~$\DD$ structures over $(A, B)$, their composition $g \circ f \colon \DDm{M} \to \DDm{P}$, of degree $\ell_1 + \ell_2$, is defined as the map
\[
  g \circ f \colon M \to A \otimes P \otimes B
\]
given by
\[
  g \circ f = (\mu_A \otimes \id_P \otimes \mu_B) \circ (\id_A \otimes g \otimes \id_B) \circ f,
\]
or graphically,
\begin{equation*}
  \adjustimage{scale = 0.753, valign = m}{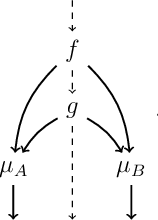}
\end{equation*}
Note that the structure map $\delta_M^1 \colon M \to (A \otimes M \otimes B)[1]$ can be thought of as a morphism $\delta_M^1 \colon \DDm{M} \to \DDm{M}$, and so we can consider the morphisms $f \circ \delta_M^1 \colon \DDm{M} \to \DDm{N}$ and $\delta_N^1 \circ f \colon \DDm{M} \to \DDm{N}$ also. In this notation, we can write
\[
  \partial f = \delta_N^1 \circ f + f \circ \delta_M^1 + d f.
\]

The above operations make type~$\DD$ structures over $(A, B)$ a differential graded category.

A \emph{type $\DD$ homomorphism} (or simply a \emph{homomorphism}) \emph{from $\DDm{M}$ to $\DDm{N}$ of degree $\ell$} is a morphism satisfying $\partial f = 0$. Graphically, this can be represented as
\begin{equation*}
  \adjustimage{scale = 0.753, valign = m}{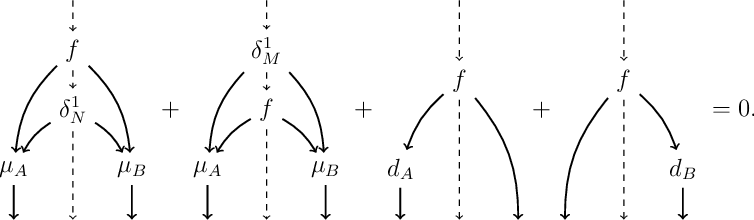}
\end{equation*}
For example, the \emph{identity morphism} $\Idm{N} \colon \DDm{N} \to \DDm{N}$ of a type~$\DD$ structure $\DDm{N}$ is the map that sends $x \in N$ to $I_A \otimes x \otimes I_B$, where $I_A$ (resp.\ $I_B$) is the unit of the algebra $A$ (resp.\ $B$). In the context of tangle Floer homology, $I_A$ and $I_B$ will be the sum of all primitive idempotents.

Given a homomorphism $f \colon \DDm{M} \to \DDm{N}$ of degree $\ell$ between two type~$\DD$ structures over $(A, B)$,  the \emph{mapping cone $\Cone (f)$ of $f$} is the type $\DD$ structure with underlying $(\k, \j)$-bimodule $M[\ell+1] \oplus N$ and structure map $\delta_f^1$ given by
\[
  \delta_f^1 (m,n) = (\delta_M^1 (m), f (m) + \delta_N^1 (n)).
\]

Let $f,g \colon \DDm{M} \to \DDm{N}$ be two homomorphisms of degree $l$. A \emph{homotopy
  between $f$ and $g$} is a morphism $h \colon \DDm{M} \to \DDm{N}$ of degree $l+1$ such that
\[
  \partial h = f + g,
\]
or graphically,
\begin{equation*}
  \adjustimage{scale = 0.753, valign = m}{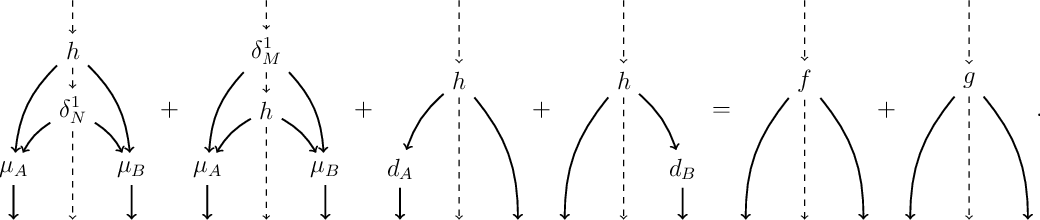}
\end{equation*}
Note that $h$ is a morphism, but not a homomorphism unless $f = g$. We write $f \simeq g$ if $f$ and $g$ are homotopic. We say that two type~$\DD$ structures $\DDm{M, N}$ are \emph{homotopy equivalent}, and write $\DDm{M} \simeq \DDm{N}$, if there exist grading-preserving type~$\DD$ homomorphisms $f \colon \DDm{M} \to \DDm{N}$ and $g \colon \DDm{N} \to \DDm{M}$ such that $g \circ f$ is homotopic to $\Idm{M}$ and $f \circ g$ is homotopic to $\Idm{N}$ via homotopies of degree $1$.  In the full subcategory of type~$\DD$ structures that are homotopy equivalent to bounded ones, the notion of homotopy equivalence coincides with an appropriate notion of \emph{quasi-isomorphism} \cite[Corollary~2.4.4]{bimod}. All algebraic structures in bordered Heegaard Floer homology and tangle Floer homology are homotopy equivalent to bounded ones; this can be seen by choosing an \emph{admissible} Heegaard diagram that defines the same bordered $3$-manifold or tangle \cite[Lemma~6.6]{bimod}.

Although we will only focus on type~$D$---in fact, type~$\DD$ structures---we should mention that there are also \emph{type~$A$ structures} over a differential graded algebra $A$, which (in the present context) are just differential graded modules over $A$. By extension, there are also \emph{type~$\DA$, $\AD$, and $\AA$ structures}.  There is a \emph{box product} (or \emph{box tensor}) operation $\boxtimes$ between a right (resp.\ left) type~$A$ structure $\DDm{M}$ and a left (resp.\ right) type~$D$ structure $\DDm{N}$ (at least one of which is bounded), resulting in a chain complex $\DDm{M} \boxtimes \DDm{N}$ (resp.\ $\DDm{N} \boxtimes \DDm{M}$) over $\F{2}$.  The box tensor is defined also for bimodules; for example, box-tensoring a type~$\AD$ structure $\DDm{M}$ and a type~$\AA$ structure $\DDm{N}$ yields a type~$\AA$ structure $\DDm{M} \boxtimes \DDm{N}$. We refer the interested reader to \cite[\S 2.3.2]{bimod}.

We may treat $A$ itself as a type~$\AA$ structure; box-tensoring with $A$ then turns a type~$\DD$ structure $\DDm{M}$ into a type~$\DA$ structure $\DDm{M} \boxtimes A$. In fact, this defines a differential graded functor from the full subcategory of type~$\DD$ structures that are homotopy equivalent to bounded ones to the full subcategory of type~$\DA$ structures that are homotopy equivalent to bounded ones. This functor is actually a \emph{quasi-equivalence} \cite[Proposition~2.3.18]{bimod}, implying that it preserves quasi-isomorphisms. Corresponding statements hold for type~$\AD$ and type~$\AA$ structures. Since the notions of quasi-isomorphism and homotopy equivalence coincide for structures of any type given that they are homotopy equivalent to bounded ones \cite[Corollary~2.4.4]{bimod}, to prove Theorem~\ref{thm:ourtheorem}, we need only prove it for type~$\DD$ structures.

In our proof of Theorem~\ref{thm:ourtheorem}, we will need to adapt to the setting of type~$\DD$ structures a lemma in homological algebra, whose version for chain complexes first appeared in \cite{bdc}.

\begin{lem}
  \label{lem:hom_alg}
  Let $\DDm{M}_k = \{(M_k, \delta^1_k)\}_{k\in \cycgrp{3}}$ be a collection of type~$\DD$ structures over $(A,B)$, where $A$ and $B$ are both unital differential graded algebras over a base ring $\k$ of characteristic $2$, and let $f_k \colon \DDm{M}_k \to \DDm{M}_{k+1}$, $\phi_k \colon \DDm{M}_k \to \DDm{M}_{k+2}$, and $\psi_k \colon \DDm{M}_k \to \DDm{M}_k$ be morphisms satisfying the following conditions for each $k$:
  \begin{enumerate}
    \item The morphism $f_k \colon \DDm{M}_k\to \DDm{M}_{k+1}$ is a type~$\DD$ homomorphism, i.e.
      \[
        \partial f_k = 0;
      \]
    \item The morphism $f_{k+1} \circ f_k$ is homotopic to zero via the homotopy $\phi_k$, i.e.
      \[
        f_{k+1} \circ f_k + \partial \phi_k = 0;
      \]
    \item The morphism $f_{k+2} \circ \phi_k + \phi_{k+1} \circ f_k$ is homotopic to the identity $\Id_k$ via the homotopy $\psi_k$, i.e.
      \[
        f_{k+2} \circ \phi_k + \phi_{k+1} \circ f_k + \partial \psi_k = \Id_k.
      \]
  \end{enumerate}
  (A graphical representation of the conditions above is given in Figure~\ref{fig:hom_alg}.) Then for each $k$, the type~$\DD$ structure $\DDm{M}_k$ is homotopy equivalent to the mapping cone $\Cone (f_{k+1})$.
\end{lem}

\vspace{-0.2cm}

\begin{figure}[h]
  \centering
  \begin{equation*}
    \adjustimage{scale = 0.75, valign = m}{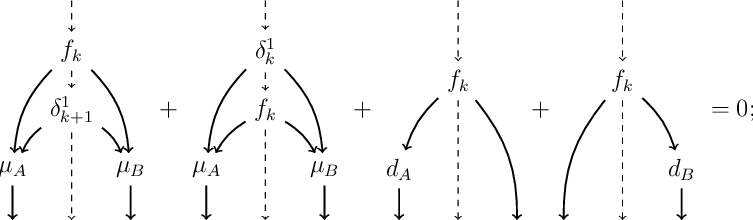}\tag{1}
  \end{equation*}
  \begin{equation*}
    \adjustimage{scale = 0.75, valign = m}{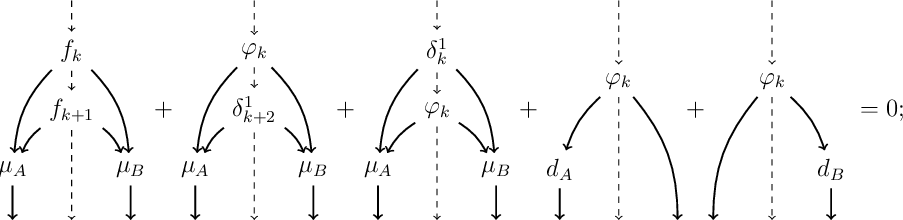}\tag{2}
  \end{equation*}
  \begin{equation*}
    \adjustimage{scale = 0.75, valign = m}{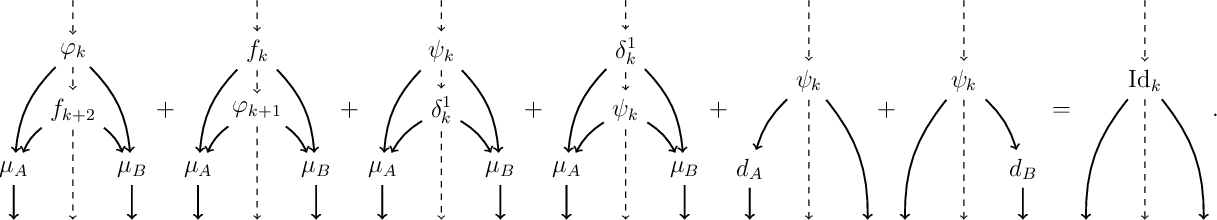}\tag{3}
  \end{equation*}
  \vspace{-0.2cm}
  \caption{Graphical representations of the conditions in Lemma~\ref{lem:hom_alg}.}
  \label{fig:hom_alg}
\end{figure}

\begin{proof}
  Observe that the mapping cone $\Cone (f_{k+1})$ is defined because $f_{k+1}$ is a type~$\DD$ homomorphism, using Condition~(1). It has underlying module $M_{k+1} \oplus M_{k+2}$, and structure map
  \[
    \delta_{f_{k+1}}^1 (m_{k+1}, m_{k+2}) = (\delta_{k+1}^1 (m_{k+1}), f_{k+1} (m_{k+1}) + \delta_{k+2}^1 (m_{k+2})).
  \]
  
  \begin{figure}[h]
    \centering
    \begin{tikzpicture}
      \node at (0,0) (m1) {$\DDm{M}_k$};
      \node at (3,0) (m2) {\phantom{\fbox{$\DDm{M}_{k+1}$}}};
      \node at (6,0) (m3) {\phantom{\fbox{$\DDm{M}_{k+2}$}}};
      \node at (4.5,.15) (c3) {\fbox{$\DDm{M}_{k+1}$\raisebox{.5cm}{\phantom{a}}\vphantom{$a=\frac{\frac{a^a}{a^a}}{a}$} \hspace{1.3cm} $\DDm{M}_{k+2}$}};
      \node at (13.5,.15) (c2) {\fbox{$\DDm{M}_{k+1}$\raisebox{.5cm}{\phantom{a}}\vphantom{$a=\frac{\frac{a^a}{a^a}}{a}$} \hspace{1.3cm} $\DDm{M}_{k+2}$}};
      \node at (9,0) (m4) {$\DDm{M}_k$};
      \node at (12,0) (m5) {\phantom{\fbox{$\DDm{M}_{k+1}$}}};
      \node at (15,0) (m6) {\phantom{\fbox{$\DDm{M}_{k+2}$}}};
      \draw[al, red] (m1) to node[below] {$\phi_k$} (m3);
      \draw[al, red] (m4) to node[below] {$\phi_k$} (m6);
      \draw[al, blue] (m2) to node[below] {$\phi_{k+1}$} (m4);
      \draw[al, YellowOrange] (m3) to node[below] {$\phi_{k+2}$} (m5);
      \draw[arb, OliveGreen] (m1) to node[above] {$\psi_k$} (m4);
      \draw[arb, YellowOrange] (m2) to node[above] {$\psi_{k+1}$} (m5);
      \draw[arb, YellowOrange] (m3) to node[above] {$\psi_{k+2}$} (m6);
      \draw[algarrow, red] (m1) to node[above] {$f_k$} (m2);
      \draw[algarrow, red] (m4) to node[above] {$f_k$} (m5);
      \draw[algarrow] (m2) to  node[above] {$f_{k+1}$} (m3);
      \draw[algarrow, blue] (m3) to node[above] {$f_{k+2}$} (m4);
      \draw[algarrow] (m5) to (m6);
    \end{tikzpicture}		
    \caption{A diagram for the maps discussed in the proof of Lemma~\ref{lem:hom_alg}. The mapping cone $\Cone (f_{k+1})$ is boxed. The maps $G_k$, $G_k'$, $H_k$, and $H_k'$ are shown in red, blue, green, and orange, respectively.}
    \label{fig:cone}
  \end{figure}

  To show that $\DDm{M}_k \simeq \Cone (f_{k+1})$, we will define homomorphisms $G_k \colon \DDm{M}_k \to \Cone (f_{k+1})$ and $G_k' \colon \Cone (f_{k+1}) \to \DDm{M}_k$, and homotopies $H_k \colon \DDm{M}_k \to \DDm{M}_k$ and $H_k' \colon \Cone (f_{k+1}) \to \Cone (f_{k+1})$. For ease of reading, we provide a schematic diagram in Figure~\ref{fig:cone}.

  We define $G_k \colon \DDm{M}_k \to \Cone (f_{k+1})$ and $G_k' \colon \Cone (f_{k+1}) \to \DDm{M}_k$ as follows:
  \begin{align*}
    G_k (m_k) & = (f_k (m_k), \phi_k (m_k)),\\*
    G_k' (m_{k+1}, m_{k+2}) & = \phi_{k+1} (m_{k+1}) + f_{k+2} (m_{k+2}).
  \end{align*}
  We first claim that $G_k$ and $G_k'$ are type~$\DD$ homomorphisms. Indeed,
  \begin{align*}
    \partial G_k (m_k) & = \delta_{f_{k+1}}^1 \circ G_k (m_k) + G_k \circ \delta_k^1 (m_k) + d G_k (m_k)\\
    & = (\delta_{k+1}^1 \circ f_k (m_k), f_{k+1} \circ f_k (m_k) + \delta_{k+2}^1 \circ \phi_k (m_k))\\*
    & \quad + (f_k \circ \delta_k^1 (m_k), \phi_k \circ \delta_k^1(m_k)) + (d f_k (m_k), d \phi_k (m_k))\\
    & = (\partial f_k (m_k), f_{k+1} \circ f_k (m_k) + \partial \phi_k (m_k))\\
    & = (0, 0),
  \end{align*}
  where the last equality follows from Conditions~(1) and (2). Similarly,
  \begin{align*}
    \partial G_k' (m_{k+1}, m_{k+2}) & = \delta_k^1 \circ G_k' (m_{k+1}, m_{k+2}) + G_k' \circ \delta_{f_{k+1}}^1 (m_{k+1}, m_{k+2}) + d G_k' (m_{k+1}, m_{k+2})\\
    & = \delta_k^1 \circ \phi_{k+1} (m_{k+1}) + \delta_k^1 \circ f_{k+2} (m_{k+2})\\*
    & \quad + \phi_{k+1} \circ \delta_{k+1}^1 (m_{k+1}) + f_{k+2} \circ f_{k+1} (m_{k+1}) + f_{k+2} \circ \delta_{k+2}^1 (m_{k+2})\\*
    & \quad + d \phi_{k+1} (m_{k+1}) + d f_{k+2} (m_{k+2})\\
    & = (f_{k+2} \circ f_{k+1} (m_{k+1}) + \partial \phi_{k+1} (m_{k+1})) + \partial f_{k+2} (m_{k+2})\\
    & = 0,
  \end{align*}
  where again the last equality follows from Conditions~(1) and (2).

  We next claim that $G_k' \circ G_k \simeq \Id_k$ and $G_k \circ G_k' \simeq \Id_{\Cone (f_{k+1})}$. To show this, we define the homotopy morphisms $H_k \colon \DDm{M}_k \to \DDm{M}_k$ and $H_k' \colon \Cone (f_{k+1}) \to \Cone (f_{k+1})$ as follows:
  \begin{align*}
    H_k (m_k) & = \psi_k (m_k),\\*
    H_k' (m_{k+1}, m_{k+2}) & = (\psi_{k+1} (m_{k+1}) + \phi_{k+2} (m_{k+2}), \psi_{k+2} (m_{k+2})).
  \end{align*}
  Then
  \begin{align*}
    \partial H_k (m_k) + G_k' \circ G_k (m_k) & = \partial \psi_k (m_k) + \phi_{k+1} \circ f_k (m_k) + f_{k+2} \circ \phi_k (m_k) \\
    & = \Id_k (m_k),
  \end{align*}
  where the last equality follows from Condition~(3). The homotopy $G_k \circ G_k' \simeq \Id_{\Cone (f_{k+1})}$ is a little more tedious.
  \begin{align*}
    & \quad \, \, \partial H_k' (m_{k+1}, m_{k+2}) + G_k \circ G_k' (m_{k+1}, m_{k+2})\\
    & = \delta_{f_{k+1}}^1 \circ H_k' (m_{k+1}, m_{k+2}) + H_k' \circ \delta_{f_{k+1}}^1 (m_{k+1}, m_{k+2}) + d H_k' (m_{k+1}, m_{k+2})\\*
    & \quad + G_k \circ G_k' (m_{k+1}, m_{k+2})\\
    & = (\delta_{k+1}^1 \circ \psi_{k+1} (m_{k+1}) + \delta_{k+1}^1 \circ \phi_{k+2} (m_{k+2}),\\*
    & \qquad \qquad f_{k+1} \circ \psi_{k+1} (m_{k+1}) + f_{k+1} \circ \phi_{k+2} (m_{k+2}) + \delta_{k+2}^1 \circ \psi_{k+2} (m_{k+2}))\\*
    & \quad + (\psi_{k+1} \circ \delta_{k+1}^1 (m_{k+1}) + \phi_{k+2} \circ f_{k+1} (m_{k+1}) + \phi_{k+2} \circ \delta_{k+2}^1 (m_{k+2}),\\*
    & \qquad \qquad \psi_{k+2} \circ f_{k+1} (m_{k+1}) + \psi_{k+2} \circ \delta_{k+2}^1 (m_{k+2}))\\*
    & \quad + (d \psi_{k+1} (m_{k+1}) + d \phi_{k+2} (m_{k+2}), d \psi_{k+2} (m_{k+2}))\\*
    & \quad + (f_k \circ \phi_{k+1} (m_{k+1}) + f_k \circ f_{k+2} (m_{k+2}), \phi_k \circ \phi_{k+1} (m_{k+1}) + \phi_k \circ f_{k+2} (m_{k+2}))\\
    & = ((f_k \circ \phi_{k+1} (m_{k+1}) + \phi_{k+2} \circ f_{k+1} (m_{k+1}) + \partial \psi_{k+1} (m_{k+1}))\\*
    & \qquad \qquad \quad \; \; + (f_k \circ f_{k+2} (m_{k+2}) + \partial \phi_{k+2} (m_{k+2})),\\*
    & \qquad \qquad (f_{k+1} \circ \phi_{k+2} (m_{k+2}) + \phi_k \circ f_{k+2} (m_{k+2}) + \partial \psi_{k+2} (m_{k+2}))\\*
    & \qquad \qquad \quad \; \; + (\phi_k \circ \phi_{k+1} (m_{k+1}) + f_{k+1} \circ \psi_{k+1} (m_{k+1}) + \psi_{k+2} \circ f_{k+1} (m_{k+1})))\\*
    & = (\Id_{k+1} (m_{k+1}), \Id_{k+2} (m_{k+2}) + \eta_{k+1} (m_{k+1})),
  \end{align*}
  where the last equality uses Conditions~(2) and (3), and $\eta_{k+1} \colon \DDm{M}_{k+1} \to \DDm{M}_{k+2}$ is the morphism
  \[
    \eta_{k+1} = \phi_k \circ \phi_{k+1} + f_{k+1} \circ \psi_{k+1} + \psi_{k+2} \circ f_{k+1}.
  \]
  Letting $J_k \colon \Cone (f_{k+1}) \to \Cone (f_{k+1})$ be the morphism
  \[
    J_k (m_{k+1}, m_{k+2}) = (\Id_{k+1} (m_{k+1}), \Id_{k+2} (m_{k+2}) + \eta_{k+1} (m_{k+1})),
  \]
  we see that $G_k \circ G_k' \simeq J_k$. Observe also that $J_k \circ J_k = \Id_{\Cone (f_{k+1})}$. But then
  \[
    J_k \simeq G_k \circ G_k' = G_k \circ \Id_k \circ G_k' \simeq G_k \circ (G_k' \circ G_k) \circ G_k' = (G_k \circ G_k') \circ (G_k \circ G_k') \simeq J_k \circ J_k = \Id_{\Cone (f_{k+1})}.
  \]
  This shows that $G_k \comp G_k' \simeq \Id_{\Cone (f_{k+1})}$, as desired.
\end{proof}

\begin{rmk}
  In \cite{bdc}, a proof is given for the chain-complex version of this lemma. That proof does not translate to the case of type~$\DD$ structures, since it involves taking the homology of the chain complexes. Instead, the proof we have presented here is the type~$\DD$ version of an alternative proof for the lemma in \cite{bdc, kmos, mskein, wong}, which is known in the community but not found in the literature.
\end{rmk}

In our proof of Theorem~\ref{thm:oriented}, we will also use another lemma in homological algebra, which is the analogue of a well-known lemma for chain complexes.

\begin{lem}
  \label{lem:hom_alg_2}
  Let $\DDm{M}, \DDm{N}_1, \DDm{N}_2$ be type~$\DD$ structures over $(A, B)$, where $A$ and $B$ are both unital differential graded algebras over a base ring $\k$ of characteristic $2$. Let $\PP \colon \DDm{N}_1 \to \DDm{N}_2$ be a homotopy equivalence of type~$\DD$ structures, and for $i = 1, 2$, let $f_i \colon \DDm{M} \to \DDm{N}_i$ be type~$\DD$ homomorphisms such that
  \[
    \PP \circ f_1 = f_2.
  \]
  Then $\Cone (f_1)$ is homotopy equivalent to $\Cone (f_2)$.
\end{lem}

\begin{proof}
  By definition, there is a type~$\DD$ homomorphism $\PP' \colon \DDm{N}_2 \to \DDm{N}_1$ such that $\PP' \circ \PP = \Id_{\DDm{N}_1} + \bdy \HH_1$ and $\PP \circ \PP' = \Id_{\DDm{N}_2} + \bdy \HH_2$ for some morphisms $\HH_i \colon \DDm{N}_i \to \DDm{N}_i$. The mapping cone $\Cone (f_i)$ has underlying module $M \oplus N_i$, and structure map
  \[
    \delta_{f_i}^1 (m, n_i) = (\delta_{\DDm{M}}^1 (m), f_i (m) + \delta_{\DDm{N}_i}^1 (n_i)).
  \]
  We will define homomorphisms $F_1 \colon \Cone (f_1) \to \Cone (f_2)$ and $F_2 \colon \Cone (f_2) \to \Cone (f_1)$ and homotopies $\Phi_i \colon \Cone (f_i) \to \Cone (f_i)$; see Figure~\ref{fig:cone-2} for a schematic diagram.

  \begin{figure}[h]
    \centering
    \begin{tikzpicture}
      \node at (0,3) (m1) {$\DDm{M}$};
      \node at (3,3) (m2) {$\DDm{M}$};
      \node at (6,3) (m3) {$\DDm{M}$};
      \node at (9,3) (m4) {$\DDm{M}$};
      \node at (0,0) (n1) {$\DDm{N}_1$};
      \node at (3,0) (n2) {$\DDm{N}_2$};
      \node at (6,0) (n3) {$\DDm{N}_1$};
      \node at (9,0) (n4) {$\DDm{N}_2$};
      \draw[al, OliveGreen] (n1) to node[below] {$\HH_1$} (n3);
      \draw[al, YellowOrange] (n2) to node[below] {$\HH_2$} (n4);
      \draw[algarrow] (m1) to node[right] {$f_1$} (n1);
      \draw[algarrow] (m2) to node[right] {$f_2$} (n2);
      \draw[algarrow] (m3) to node[right] {$f_1$} (n3);
      \draw[algarrow] (m4) to node[right] {$f_2$} (n4);
      \draw[algarrow, red] (m1) to node[above] {$\Id$} (m2);
      \draw[algarrow, blue] (m2) to  node[above] {$\Id$} (m3);
      \draw[algarrow, red] (m3) to node[above] {$\Id$} (m4);
      \draw[algarrow, red] (n1) to node[above] {$\PP$} (n2);
      \draw[algarrow, blue] (n2) to  node[above] {$\PP'$} (n3);
      \draw[algarrow, red] (n3) to node[above] {$\PP$} (n4);
      \draw[algarrow, blue] (m2) to node[midway, sloped, above] {$\HH_1\circ f_1$} (n3);
    \end{tikzpicture}		
    \caption{A diagram for the maps discussed in the proof of Lemma~\ref{lem:hom_alg_2}. The vertical pieces form the mapping cones. The maps $F_1$, $F_2$, $\Phi_1$, and $\Phi_2$ are shown in red, blue, green, and orange, respectively.}
    \label{fig:cone-2}
  \end{figure}

  We define the morphisms $F_1 \colon \Cone (f_1) \to \Cone (f_2)$ and $F_2 \colon \Cone (f_2) \to \Cone (f_1)$ by
  \begin{align*}
    F_1 (m, n_1) & = (\Id_{\DDm{M}} (m), \PP (n_1)),\\
    F_2 (m, n_2) & = (\Id_{\DDm{M}} (m), \HH_1 \circ f_1 (m) + \PP' (n_2)).
  \end{align*}
  We first claim that $F_1$ and $F_2$ are type~$\DD$ homomorphisms. Indeed,
  \begin{align*}
    \partial F_1 (m, n_1) & = \delta_{f_2}^1 \circ F_1 (m, n_1) + F_1 \circ \delta_{f_1}^1 (m, n_1) + d F_1 (m, n_1)\\
    & = (\delta_{\DDm{M}}^1 (m), f_2 (m) + \delta_{\DDm{N}_2}^1 \circ \PP (n_1)) + (\delta_{\DDm{M}}^1 (m), \PP \circ f_1 (m) + \PP \circ \delta_{\DDm{N}_1}^1 (n_1))\\
    & \quad + (d \Id_{\DDm{M}} (m), d \PP (n_1))\\
    & = (0, 0),
  \end{align*}
  where the last equality follows from $\PP \circ f_1 = f_2$, the fact that $\PP$ is a homomorphism, and $d_A I_A = d_B I_B = 0$. Similarly,
  \begin{align*}
    \partial F_2 (m, n_2) & = \delta_{f_1}^1 \circ F_2 (m, n_2) + F_2 \circ \delta_{f_2}^1 (m, n_2) + d F_2 (m, n_2)\\
    & = (\delta_{\DDm{M}}^1 (m), f_1 (m) + \delta_{\DDm{N}_1}^1 \circ \HH_1 \circ f_1 (m) + \delta_{\DDm{N}_1}^1 \circ \PP' (n_2))\\
    & \quad + (\delta_{\DDm{M}}^1 (m), \HH_1 \circ f_1 \circ \delta_{\DDm{M}}^1 (m) + \PP' \circ f_2 (m) + \PP' \circ \delta_{\DDm{N}_2}^1 (n_2))\\
    & \quad + (d \Id_{\DDm{M}} (m), d (\HH_1 \circ f_1) (m) + d \PP' (n_2))\\
    & = (0, f_1 (m) + \delta_{\DDm{N}_1}^1 \circ \HH_1 \circ f_1 (m) + \HH_1 \circ f_1 \circ \delta_{\DDm{M}}^1 (m) + \PP' \circ f_2 (m)\\
    & \qquad \qquad \quad + d \HH_1 \circ f_1 (m) + \HH_1 \circ d f_1 (m)).\\
    & = (0, \PP' \circ f_2 (m) + f_1 (m) + \delta_{\DDm{N}_1}^1 \circ \HH_1 \circ f_1 (m) + d \HH_1 \circ f_1 (m) + \HH_1 \circ \delta_{\DDm{N}_1}^1 \circ f_1 (m))\\
    & = (0, \PP' \circ f_2 (m) + \PP' \circ \PP \circ f_1 (m))\\
    & = (0, 0),
  \end{align*}
  where the last four equalities follow, respectively, from the fact $\PP'$ is a homomorphism, that $f_1$ is a homomorphism, that $\HH_1$ is a homotopy between $\Id_{\DDm{N}_1}$ and $\PP' \circ \PP$, and that $\PP \circ f_1 = f_2$.

  We next claim that $F_2 \circ F_1 \simeq \Id_{\Cone (f_1)}$ and $F_1 \circ F_2 \simeq \Id_{\Cone (f_2)}$. To show this, we define the homotopy morphisms $\Phi_i \colon \Cone (f_i) \to \Cone (f_i)$ for $i = 1, 2$, as follows:
  \begin{align*}
    \Phi_1 (m, n_1) & = (0, \HH_1 (n_1)),\\
    \Phi_2 (m, n_2) & = (0, \HH_2 (n_2)).
  \end{align*}
  Then
  \begin{align*}
    & \quad \, \, \bdy \Phi_1 (m, n_1) + F_2 \circ F_1 (m, n_1)\\
    & = \delta_{f_1}^1 \circ \Phi_1 (m, n_1) + \Phi_1 \circ \delta_{f_1}^1 (m, n_1) + d \Phi_1 (m, n_1) + F_2 \circ F_1 (m, n_1)\\
    & = (0, \delta_{\DDm{N}_1}^1 \circ \HH_1 (n_1)) + (0, \HH_1 \circ f_1 (m) + \HH_1 \circ \delta_{\DDm{N}_1}^1 (n_1)) + (0, d \HH_1 (n_1))\\
    & \quad + (\Id_{\DDm{M}} (m), \HH_1 \circ f_1 (m) + \PP' \circ \PP (n_1))\\
    & = (\Id_{\DDm{M}} (m), \Id_{\DDm{N}_1} (n_1))\\
    & = \Id_{\Cone (f_1)} (m, n_1),
  \end{align*}
  where the third equality follows from the fact that $\PP' \circ \PP$ is homotopic to $\Id_{\DDm{N}_1}$ via $\HH_1$. Similarly,
  \begin{align*}
    & \quad \, \, \bdy \Phi_2 (m, n_2) + F_1 \circ F_2 (m, n_2)\\
    & = \delta_{f_2}^1 \circ \Phi_2 (m, n_2) + \Phi_2 \circ \delta_{f_2}^1 (m, n_2) + d \Phi_2 (m, n_2) + F_1 \circ F_2 (m, n_2)\\
    & = (0, \delta_{\DDm{N}_2}^1 \circ \HH_2 (n_1)) + (0, \HH_2 \circ f_2 (m) + \HH_2 \circ \delta_{\DDm{N}_2}^1 (n_2)) + (0, d \HH_2 (n_2))\\
    & \quad + (\Id_{\DDm{M}} (m), \PP_1 \circ \HH_1 \circ f_1 (m) + \PP \circ \PP' (n_2))\\
    & = (\Id_{\DDm{M}} (m), \Id_{\DDm{N}_2} (n_2) + \HH_2 \circ f_2 (m) + \PP_1 \circ \HH_1 \circ f_1 (m)),
  \end{align*}
  where the last equality follows from the fact that $\PP \circ \PP'$ is homotopic to $\Id_{\DDm{N}_2}$ via $\HH_2$. An argument similar to the last part of the proof of Lemma~\ref{lem:hom_alg} finishes the proof.
\end{proof}

\subsection{Tangle Floer homology}\label{ssec:tf}

Tangle Floer homology is an invariant of tangles, which takes the form of a (bi)module such as the ones discussed in Section \ref{ssec:alg}; see \cite{pv}. In this section, we review the combinatorial construction of tangle Floer homology for tangles in  $I\times \mathbb R^2$, with special focus on those tangles that are relevant to our proof. 

An $(m,n)$-\emph{tangle} (or simply a \emph{tangle}) $\T$ is a properly, smoothly embedded, oriented $1$-manifold in  $I\times \mathbb R^2$, with boundary $\bdy \T = \bdy^L\T\sqcup \bdy^R\T$, where $\bdy^L\T = \{0\}\times\{1, \ldots, m\}\times \{0\}$ and $\bdy^R\T = \{1\}\times\{1, \ldots, n\}\times \{0\}$, treated as oriented sequences of points.  A planar diagram of a tangle is a projection to the $I\times\mathbb R$ subset of the $(x,y)$-plane, with no triple intersections, self-tangencies, or cusps, and with over- and under-crossing data preserved (as viewed from the positive $z$ direction).  The boundaries of $\T$ can be thought of as sign sequences
\begin{equation*}
  -\bdy^L\T\in\{+,-\}^m, \quad \bdy^R\T\in\{+,-\}^n,
\end{equation*}
according to the orientation of each point ($+$ if the tangle is oriented left-to-right, $-$ if the tangle is oriented right-to-left at that point).
Given two tangles $\T$ and $\T'$ with  $\bdy^R\T=-\bdy^L\T'$, we can concatenate them to obtain a new tangle $\T\circ\T'$, by placing $\T'$ to the right of $\T$ and scaling in the $x$ direction by $1/2$. We also consider unoriented tangles, and think of their boundaries as sequences of (unoriented) points. 

In \cite{pv}, to a sign sequence one associates a DGA, and to a tangle a left-right bimodule over the DGAs for the respective boundaries. First, we recall the definition of the algebra. For more details, see \cite[Section 3]{pv}.

Let $P = (p_1, \ldots, p_n) \in \{+, -\}^n$ be a sign sequence and let $[n]=\{0,1,\ldots, n\}$. One associates to $P$ a differential graded algebra $\am{P}$ over $\F{2}[U_1, \ldots, U_t]$, where the variables $U_1, \ldots, U_t$ correspond to the positively oriented points in $P$. The algebra is generated over $\F{2}[U_1, \ldots, U_t]$ by partial bijections $[n] \to [n]$ (i.e. bijections $\sss\to \ttt$ for $\sss, \ttt\subset [n]$), which can be drawn as strand diagrams (up to isotopy and Reidemeister III moves), as follows. 

Represent each $p_i$ by a horizontal orange strand  $[0,1]\times \{i-\frac{1}{2}\}$ oriented left-to-right if $p_i=+$  and right-to-left if $p_i = -$ (in \cite{pv}, those are dashed green strands and double orange strands, respectively). Represent a bijection $\phi:\sss\to \ttt$ by black strands connecting $(0,i)$ to $(1, \phi i)$ for $i\in \sss$. We further require that there are no triple intersection points and there are a minimal number of intersection points between strands. 

Let $a:\sss_1\to \ttt_1, b:\sss_2\to \ttt_2$ be generators. If $\ttt_1 \neq\sss_2$, define the product $ab$ to be $0$. If $\ttt_1 = \sss_2$, consider the concatenation of a diagram for $a$ to the left and a diagram for $b$ to the right.  If  there is a black strand that crosses a left-oriented orange strand or another black strand twice, define $ab=0$.
Otherwise define $ab = (\prod_i U_i^{n_i}) b\circ a$ where $n_i$ is the number of black strands that double cross the $i^{\text{th}}$ right-oriented orange strand. See Figure \ref{fig:alg}.

For a generator $a$, define its differential $d a$ as the sum of all ways of smoothing one black-black crossing in a diagram for $a$ locally, subject to the following rules. Any resulting diagram with a black-black double crossing, or a double crossing between a black strand and a left-oriented orange strand, is discarded. If a resulting diagram has a double crossing between the $i^{\text{th}}$ right-oriented orange strand and a black strand, it represents $U_i b$, where the diagram for $b$ is obtained from this diagram by performing a Reidemeister~II move to remove the aforementioned double crossing. (This process may have to be iterated a number of times before we obtain a diagram without double crossings.) See Figure \ref{fig:alg}.

\begin{figure}[h]
  \centering
  \labellist
  \pinlabel $d$ at 292 50
  \pinlabel $U_1$ at 522 40
  \endlabellist
  \includegraphics[scale=.72]{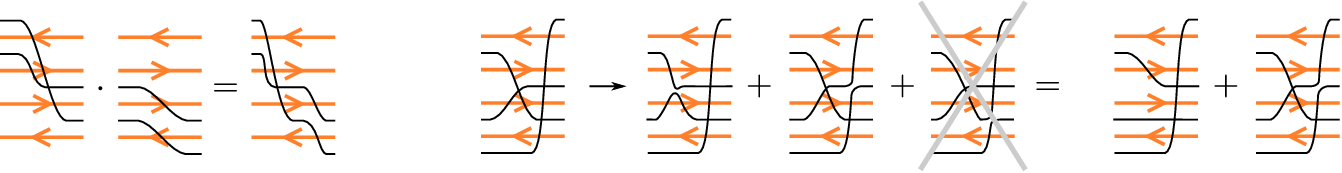}
  \caption{The algebra $\am{P}$ for $P=(-,+,+,-)$. Left: an example of the multiplication. Right: an example of the differential.}
  \label{fig:alg}
\end{figure}

The subalgebra of idempotents $\im{P}$ is generated by the identity bijections $e_{\sss}:\sss\to \sss$.

The algebra has a differential grading $M$ called the \emph{Maslov} grading, and an internal grading $A$ called the \emph{Alexander} grading. They are defined on generators by counting crossings, as follows:
\begin{eqnarray*}
  2A(a) &=& \diagup \hspace{-.4cm}{\color{orange2}{\nwarrow}}(a) + \diagdown \hspace{-.4cm}{\color{orange2}{\swarrow}}(a)
  -\diagup \hspace{-.4cm}{\color{orange2}{\searrow}}(a) - \diagdown \hspace{-.4cm}{\color{orange2}{\nearrow}}(a),\\
  M(a) & =& \diagup \hspace{-.37cm}\diagdown(a) - \diagup \hspace{-.4cm}{\color{orange2}{\searrow}}(a) -\diagdown \hspace{-.4cm}{\color{orange2}{\nearrow}}(a).
\end{eqnarray*}
Further, \begin{eqnarray*}
  A(U_i a) & =& A(a) -1,\\
  M(U_i a) &=& M(a)-2.
\end{eqnarray*}

Setting all $U_i$ to zero, we get a bigraded quotient algebra $\ah{P}= \am{P}/(U_i=0)$ over $\F{2}$. 

Further collapsing the bigrading on $\ah{P}$ to a single grading $\delta = M-A$, we obtain the $\delta$-graded algebra $\ad{P}$. Observe that the orientation of the orange strands is not seen by this algebra, since all types of double crossings are set to zero, and the $\delta$-grading on generators is given by
\begin{eqnarray*}
  \delta (a) & =&  \diagup \hspace{-.37cm}\diagdown(a) - \frac{\diagup \hspace{-.4cm}{\color{orange2}{\searrow}}(a) +\diagdown \hspace{-.4cm}{\color{orange2}{\nearrow}}(a) + \diagup \hspace{-.4cm}{\color{orange2}{\nwarrow}}(a) + \diagdown \hspace{-.4cm}{\color{orange2}{\swarrow}}(a)}{2}\\
  &=& \diagup \hspace{-.35cm}\diagdown(a) - \frac{\diagup \hspace{-.37cm}{\color{orange2}{\diagdown}}(a) +\diagdown \hspace{-.37cm}{\color{orange2}{\diagup}}(a)}{2},
\end{eqnarray*}
i.e. it counts the number of black-black crossings minus one half the number of black-orange crossings. So if $n=|P|$,  we use the shorter notation $\as{n}$ for $\ad{P}$. See Figure \ref{fig:algd}. As we just pointed out,  this algebra already does not detect the orientation on the sequence of points. However, when studying unoriented tangles, we will work with the ungraded version of this algebra, $\au{P}$, also denoted $\aus{n}$.

\begin{figure}[h]
  \centering
  \labellist
  \pinlabel $d$ at 292 50
  \endlabellist
  \includegraphics[scale=.72]{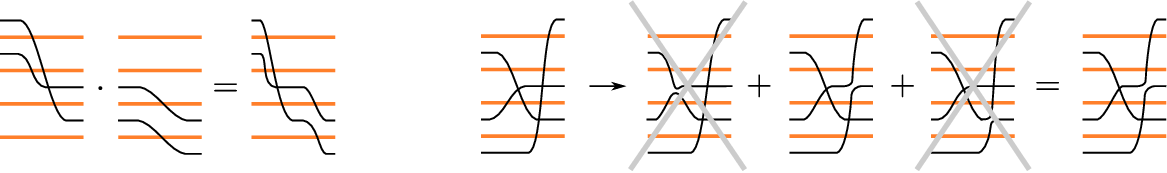}
  \caption{The algebra $\as{4}$. Left: an example of the multiplication. Right: an example of the differential.}
  \label{fig:algd}
\end{figure}

Given a tangle $\T$, we can define various bigraded bimodules over $(\am{\pm\bdy^L\T}, \am{\pm\bdy^R\T})$, where again the homological and internal gradings are denoted $M$ and $A$, respectively. These descend to $\delta$-graded bimodules over $(\as{|\bdy^L\T|}, \as{|\bdy^R\T|})$, for $\delta = A - M$, as well as to ungraded bimodules over  $(\aus{|\bdy^L\T|}, \aus{|\bdy^R\T|})$. The latter are also invariants of the underlying unoriented tangles.
In \cite{pv}, an explicit description was given only of a type~$\DA$ bimodule associated to a multipointed bordered Heegaard diagram. However, one could similarly define a type~$\AA$, $\AD$, or $\DD$ bimodule. Here we explicitly define the type~$\DD$ bimodule in the special cases of interest; see \cite[Section 4]{pv} for more details. 

Let $\T$ be an $(n,n)$-tangle consisting of straight strands, or of one crossing, or of a cap followed by a cup at the same height, possibly with straight strands on either side. Then $\T$ can be represented by a genus-one multipointed Heegaard diagram $\HD=(\Sigma, \alphas, \betas, \XX, \OO)$ such as the diagrams in Figure~\ref{fig:hd_elem_ori}; see \cite{pv} for a complete definition. If we cut $\HD$ through the middle along a vertical plane, all relevant data is contained in the two resulting \emph{bordered grids}. We may occasionally refer to the two grids as the \emph{left grid} and the \emph{right grid}, based on where they stand relative to each other when the diagram is drawn as in  Figure~\ref{fig:hd_elem_ori}. Here, $\Sigma$ is a genus-one surface with two boundary components, $\betas$ is a set of $n+1$ circles in $\Sigma$, $\alphas$ is a set of $2n+2$ arcs, and $\XX$ and $\OO$ are sets of $n$ points labeled $X_1, \ldots, X_n$ and $O_1, \ldots, O_n$, respectively (we often omit the indices, both in figures and in writing). One can see the tangle by connecting $X$ to $O$ markings in the complement of the $\beta$ curves and pushing the interior of the resulting arcs below the Heegaard surface, and connecting $X$ and $O$ markings to the boundary of the diagram in the complement of the $\alpha$ curves so that the $X$'s are endpoints and the $O$'s are starting points. The tangle is oriented so that the arcs in the complement of the $\beta$ curves flow into the $O$'s, and the arcs in the complement of the $\alpha$ curves flow away from the $O$'s.  See the first diagram in Figure~\ref{fig:hd_elem_ori}.

\begin{figure}[h]
  \centering
  \includegraphics[scale=0.89]{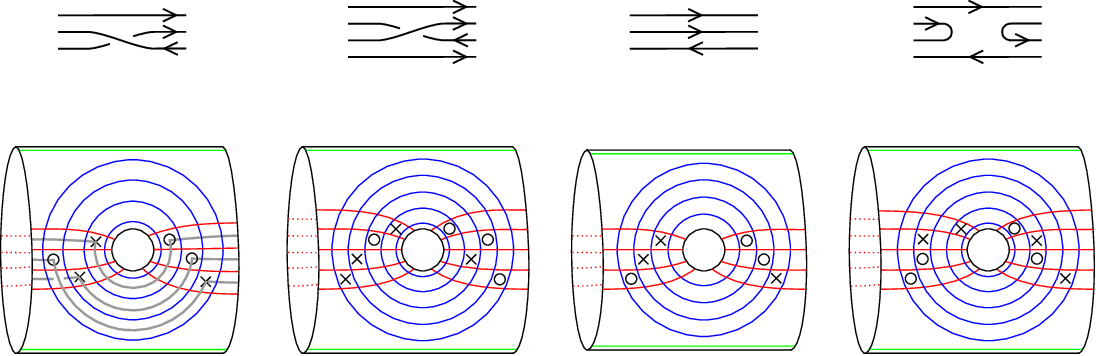}
  \caption{Examples of $(n,n)$-tangles that can be represented by genus-one Heegaard diagrams, and their respective Heegaard diagrams.}
  \label{fig:hd_elem_ori}
\end{figure}

Similarly, one can represent unoriented tangles by multipointed Heegaard diagrams with only $X$ markings (to get from a diagram for an oriented tangle to a diagram for the underlying unoriented tangle, simply replace all $O$'s  with $X$'s). See Figure~\ref{fig:hd_012} for example.  

As seen in Figure~\ref{fig:hd_elem_ori}, label the curves on $\HD$ as follows.  Label the $\alpha$ arcs touching the left boundary by $\alpha_0^L, \ldots, \alpha_n^L$, and those touching the right boundary by $\alpha_0^R, \ldots, \alpha_n^R$, indexed by their relative height, starting from the bottom. Label the $\beta$ circles by $\beta_0, \ldots, \beta_n$, indexed from the outermost to the innermost.

One defines a left-right type~$\DD$ structure $\CDTDm (\HD)$ over $(\am{-\bdy^L\T}, \am{-\bdy^R\T})$ as follows. As a module,  $\CDTDm (\HD)$  is freely generated over $\F{2}[U_1, \ldots, U_{2n}]$ (where each $U_i$ either corresponds to an $O_i\in \OO$ or is a variable in the ground ring for $\am{-\bdy^L\T}$ or $\am{-\bdy^R\T}$) by the set $\SS(\HD)$ consisting of tuples of intersection points $\x$ in $\alphas\cap \betas$ such that there is exactly one point on each $\beta$ and at most one point on each $\alpha$. For $\x \in \SS(\HD)$, let $\ocl{\x}=\setc{i}{\alpha_i^L \textrm{ is occupied by } \x}$,  $\ocr{\x}=\setc{i}{\alpha_i^R \textrm{ is occupied by } \x}$, $\unocl{\x} = [n]\setminus \ocl{\x}$, and $\unocr{\x} = [n]\setminus \ocr{\x}$.  Define an $(\im{-\bdy^L\T}, \im{-\bdy^R\T})$-bimodule structure on $\CDTDm (\HD)$ by \begin{equation*}
  e_{\sss} \x e_{\ttt} = \begin{cases}
    \x & \text{if } \sss =  \unocl{\x} \text{ and } \ttt = \unocr{\x}, \\
    0 & \text{otherwise}.
  \end{cases}
\end{equation*}
Denote  $e_{\unocl{\x}}$ and $e_{\unocr{\x}}$ by $\ideml{\x}$ and $\idemr{\x}$, respectively. 

We next describe a structure map
\[
  \delta^1 \colon \CDTDm (\HD)\to \am{-\bdy^L\T}\otimes_{\im{-\bdy^L\T}}\CDTDm (\HD) \otimes_{\im{-\bdy^R\T}}\am{-\bdy^R\T}
\]
by counting the following types of embedded $2$-chains in $\HD$:

\begin{enumerate}
  \item \label{case:delta_1} A rectangle $r$ with boundary on $\alphas\cup \betas$. Given generators $\x$ and $\y$, $r$ \emph{connects $\x$ to $\y$} if the two corners where $\bdy r$ jumps from an arc in $\betas$ to an arc in $\alphas$ are points in $\x$, the other two corners are points in $\y$, and $\x$ and $\y$ coincide elsewhere.  Define $\algl{\x,r}= \ideml{\x}$ and $\algr{\x,r} = \idemr{\x}$.
  \item \label{case:delta_2} A rectangle $r$ such that along its oriented boundary we see $\bdy^L\Sigma$, followed by $\alpha_i^L$, followed by $\beta_m$, followed by $\alpha_j^L$. We say that $r$ \emph{connects $\x$ to $\y$} if $\alpha_j^L\cap \beta_m = \x\setminus \y$ and $\alpha_i^L\cap \beta_m = \y\setminus\x$. Define $\algl{\x,r}$ as the bijection from $\unocl{\x}$ to $\unocl{\y}$ that sends $i$ to $j$ and is the identity elsewhere, and define $\algr{\x,r}=\idemr{\x}$. 
  \item \label{case:delta_3} A rectangle $r$ such that along its oriented boundary we see $\bdy^R\Sigma$, followed by $\alpha_i^R$, followed by $\beta_m$, followed by $\alpha_j^R$. We say that $r$ \emph{connects $\x$ to $\y$} if $\alpha_j^R\cap \beta_m = \x\setminus \y$ and $\alpha_i^R\cap \beta_m = \y\setminus\x$. Define $\algr{\x,r}$ as the bijection from $\unocr{\y}$ to $\unocr{\x}$ that sends $j$ to $i$ and is the identity elsewhere, and define $\algl{\x,r}=\ideml{\x}$. 
  \item \label{case:delta_4} A rectangle $r$ with boundary  two entire arcs $\alpha_i^L$ and $\alpha_j^L$ and two arcs in $\bdy^L\Sigma$. We say that $r$ \emph{connects $\x$ to $\y$} if $\x=\y$ and $i,j\notin \ocl{\x}$. Define $\algl{\x,r}$ as the bijection with domain $\unocl{\x}$ that exchanges $i$ and $j$ and is the identity elsewhere, and $\algr{\x,r}=\idemr{\x}$. 
  \item \label{case:delta_5} Defined analogously to \ref{case:delta_4}, but interchanging the superscripts $L$ and $R$ throughout.
    \end{enumerate}
For types \ref{case:delta_1}--\ref{case:delta_5}, define $U^r$  as the product of all  $U_s$ with corresponding $O_s$ in the interior of $r$.
\begin{enumerate}[resume]
  \item \label{case:delta_6} For $i<j$ and $m<n$, the union $r$ of two rectangles of the second type, such that one has boundary on $\alpha_i^L$, $\beta_m$, $\alpha_j^L$, $\bdy^L\Sigma$, and the other has boundary on  $\alpha_j^L$, $\beta_n$, $\alpha_i^L$, $\bdy^L\Sigma$. We say that $r$ \emph{connects $\x$ to $\y$} if $\{\alpha_j^L\cap \beta_m, \alpha_i^L\cap \beta_n\} = \x\setminus \y$ and $\{\alpha_i^L\cap \beta_m, \alpha_j^L, \beta_n\} = \y\setminus\x$.  \item \label{case:delta_7} For $i<j$ and $m<n$, the union $r$ of two rectangles of the third type, such that one has boundary on $\alpha_j^R$, $\beta_m$, $\alpha_i^R$, $\bdy^L\Sigma$, and the other has boundary on  $\alpha_i^R$, $\beta_n$, $\alpha_j^R$, $\bdy^R\Sigma$. We say that $r$ \emph{connects $\x$ to $\y$} if $\{\alpha_i^R\cap \beta_m, \alpha_j^R\cap \beta_n\} = \x\setminus \y$ and $\{\alpha_j^R\cap \beta_m, \alpha_i^R, \beta_n\} = \y\setminus\x$.
  \end{enumerate}
For types \ref{case:delta_6} and \ref{case:delta_7},  define $\algl{\x,r}= \ideml{\x}$ and $\algr{\x,r} = \idemr{\x}$. Also define $U^r$ as the product of all $U_s$ with corresponding $O_s$ in the interior of $r$, and all $U_t$ corresponding to positively oriented points above the $i^{\textrm th}$ and below the $j^{\mathrm th}$ point in $-\bdy^L\T$ (if type \ref{case:delta_6}) or in $-\bdy^R\T$ (if type \ref{case:delta_7}).

A $2$-chain of one of the first five types is \emph{empty} if $\Int r \cap \x = \emptyset$ and $\Int r \cap \XX = \emptyset$. A $2$-chain of the sixth or seventh type is \emph{empty} if, in addition to the requirement that  $\Int r \cap \x = \emptyset$ and $\Int r \cap \XX = \emptyset$, the interior of its complement in the strip bounded by $\alpha_i^L$ and $\alpha_j^L$, or  $\alpha_i^R$ and $\alpha_j^R$ respectively, contains $j-i-1$ points in $\XX$ and $j-i-1$ points in $\x$. From now on, we abuse notation and call each of the seven types of $2$-chains ``rectangles'', even though the latter two are, strictly speaking, unions of such.
Define
\[
  \delta^1(\x) = \sum_{\y\in \SS(\HD)} \sum_{\substack{r \textrm{ empty}\\ \textrm{$r$ connects $\x$ to $\y$}}} \algl{\x,r} \otimes U^r \y \otimes \algr{\x,r}.
\]

\begin{figure}[h]
  \centering
  \includegraphics[scale=0.89]{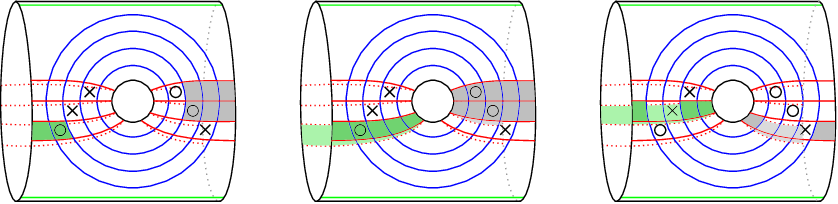}
  \caption{Left: Rectangles of types~\ref{case:delta_2} (green) and \ref{case:delta_3} (grey). Center: Rectangles of types~\ref{case:delta_4} (green) and \ref{case:delta_5} (grey). Right: Rectangles of types~\ref{case:delta_6} (green) and \ref{case:delta_7} (grey). All generators are omitted. Note the shading at the back of the center and right diagrams.}
  \label{fig:hd-dd-rects}
\end{figure}

The above types of $2$-chains are the projections onto $\Sigma$ of certain embedded curves in $\Sigma \times I\times \mathbb R$  that would appear in a holomorphic interpretation of (minus) tangle Floer homology. 
In this perspective, the algebra elements $\algl{\x,r}$ and $\algr{\x,r}$ correspond to Reeb chords that arise as the intersection of the $2$-chains with the left and right boundaries of $\HD$ respectively. Note that, as in bordered Floer homology \cite{bfh2}, the Reeb chords on the left boundary get the reverse orientation while those on the right boundary inherit the usual orientation; the definitions of $\algl{\x,r}$ and $\algr{\x,r}$  for types~\ref{case:delta_2}--\ref{case:delta_7} above take into account the resulting subtle asymmetry. For further details, we refer the reader to \cite{pv}.

Next, we restate the definition of the bigrading on generators from \cite[Section 3.4]{pv} in terms of the Heegaard diagrams described above. Gradings will not be used until Section~\ref{sec:grading}, so the reader only interested in the ungraded version of the skein relation should feel free to skip to the end of the proof of Lemma~\ref{lem:AAandADgr}.

Let  $\HD=(\Sigma, \alphas, \betas, \XX, \OO)$ be a Heegaard diagram as above. Let $\XX^R$ and $\XX^L$ be the subsets of $\XX$  that lie in the right or left grid, respectively. Define $\OO^R$ and $\OO^L$ similarly, and for $\x\in \SS(\HD)$ define $\x^R$ and $\x^L$ similarly.  We say that $\alpha_j^R\cap \beta_i$ has coordinates $(i,j)$. For a point $p\in \XX^R\cup\OO^R$, we say that $p$ has coordinates $(i+\frac 1 2, j+ \frac 1 2)$ if it lies between $\beta_i$ and $\beta_{i+1}$ and between $\alpha_j^R$ and $\alpha_{j+1}^R$.  Given two finite sets $S, T\subset \mathbb R^2$, let  $\inv(S, T)$ be the number of pairs $(s_1, s_2)\in S$ and $(t_1, t_2)\in T$ with $s_1 < t_1$ and $s_2 > t_2$, or $s_1 > t_1$ and $s_2 < t_2$. For a set $S\subset \mathbb R^2$, define $\inv(S) = \frac 1 2 \inv(S, S)$. Thinking of points in the right grid in terms of their coordinates, define
\begin{eqnarray*}
  M(\x^R) &=& \inv(\x^R)- \inv(\x^R, \OO^R)+ \inv(\OO^R),\\
  2 A(\x^R) &=& \inv(\x^R, \XX^R)- \inv(\x^R, \OO^R)+ \inv(\OO^R) - \inv(\XX^R) - |\XX^R|.
\end{eqnarray*}
Define coordinates for points in the left grid analogously, and define
\begin{eqnarray*}
  M(\x^L) &=& -\inv(\x^L)+\inv(\x^L, \OO^L)- \inv(\OO^L) - |\OO^L|,\\
  2 A(\x^L) &=& -\inv(\x^L, \XX^L)+ \inv(\x^L, \OO^L)- \inv(\OO^L) + \inv(\XX^L) - |\OO^L|.
\end{eqnarray*}
The \emph{Maslov} grading of $\x$ is given by $M(\x)=M(\x^R)+M(\x^L)$, and the \emph{Alexander} grading of $\x$ is given by $A(\x) = A(\x^R) + A(\x^L)$.
By further defining
\begin{eqnarray*}
  A(U_i \x) & =& A(\x) -1,\\
  M(U_i \x) &=& M(\x)-2.
\end{eqnarray*}
we get a bigrading compatible with the structure map on  $\CDTDm (\HD)$. 

\begin{lem}\label{lem:DDgr}
  For the left-right type $\DD$ structure $\CDTDm(\HD)$ defined above, $\delta^1$ lowers the Maslov grading by one, and preserves the Alexander grading.
\end{lem}
\begin{proof}
  Suppose $\x$ and $\y$ are connected by an empty rectangle $r$ of the third type above (such that along its oriented boundary we see $\bdy^R\Sigma$, followed by $\alpha_i^R$, followed by $\beta_m$, followed by $\alpha_j^R$), let $a$ be the corresponding algebra element $\algr{\x,r}\in \am{-\bdy^R\T}$, and let $U^r$ be the corresponding power of $U_s$ variables. The rectangle $r$ contributes to the map $\delta^1(\x) =  \ideml{\y}\otimes U^r \y \otimes a$. 
We compute the bigradings of $a$, $\x$, and $\y$ below. 
  
  Assume $i<j$; the proof when $j>i$ is analogous. 
  Let $t$ be the number of $\alpha$-arcs between $\alpha_i^R$ and $\alpha_j^R$ unoccupied by $\y$, and let $s$ be the number of $O$'s between $\alpha_i^R$ and $\alpha_j^R$ (so the number of $X$'s between $\alpha_i^R$ and $\alpha_j^R$ is $j-i-s$). Then $s$ and $j-i-s$ are the number of the coordinates $i+1$ through $j$ in $-\bdy^R\T$ that are positive and negative, respectively, so 
  \begin{eqnarray*}
    \diagup \hspace{-.37cm}\diagdown(a)  &=& t,\\
    \diagup \hspace{-.4cm}{\color{orange2}{\searrow}}(a) +\diagdown \hspace{-.4cm}{\color{orange2}{\nearrow}}(a) &=& j-i-s,\\
    \diagup \hspace{-.4cm}{\color{orange2}{\nwarrow}}(a) + \diagdown \hspace{-.4cm}{\color{orange2}{\swarrow}}(a) &=& s,
  \end{eqnarray*}
 and we have
  \begin{eqnarray*}
    M(a) &=& t-j+i+s,\\
    A(a) &=& \frac{i-j}{2}+s.
  \end{eqnarray*}

Next, we compare the inversions used in the definition of the bigrading for $\x$ and for $\y$.
For example, $\inv(\y^R)-\inv(\x^R)$ is given by counting the points in $\x\cap \y$ that are in the interior of the strip bounded by $\alpha_i^R$ and $\alpha_j^R$ with negative sign if they are in $r$ (there are no such points, as $r$ is empty) and with positive sign if they are not in $r$. So 
\[\inv(\y^R)-\inv(\x^R) = j-i-t-1.\]
Similarly, letting $p$ be  the number of $O$'s in $r$, we obtain
  \begin{eqnarray*}
    \inv(\y^R, \OO^R) - \inv(\x^R, \OO^R) &=& (s-p)-p = s-2p,\\
    \inv(\y^R, \XX^R) - \inv(\x^R, \XX^R) &=& j-i-s.
  \end{eqnarray*}
All other counts in the definition of the bigrading are the same for $\x$ and for $\y$, therefore
  \begin{eqnarray*}
    	M(\x) - M(U^r\y) &=& M(\x) - M(\y)+2p = -j+i+t+1+s = M(a)+1,\\
 	A(\x) - A(U^r\y) &=& A(\x) - A(\y)+ p  = \frac{i-j+2s-2p}{2} + p = A(a).
  \end{eqnarray*}
  This completes the proof of the lemma.
\end{proof}
As an immediate consequence, it follows that $\delta^1$ lowers the $\delta$-grading by one. 

Similarly, one could define a type~$\AA$ structure $\CATAm(\HD)$ over $(\am{\bdy^L\T}, \am{\bdy^R\T})$, or a  type~$\AD$ structure $\CATDm(\HD)$ over $(\am{\bdy^L\T}, \am{-\bdy^R\T})$, with the same underlying bigraded module as for $\CDTDm(\HD)$ and $\mathrm{CDTA}^-(\HD)$. 

\begin{lem}
  \label{lem:AAandADgr}
  For the left-right type $\AA$ structure $\CATAm(\HD)$, the multiplication maps are compatible with the Maslov grading, and they preserve the Alexander grading. 

  For the left-right type $\AD$ structure $\CATDm(\HD)$, the structure map $\delta_0^1$ lowers the Maslov grading by one and preserves the Alexander grading, and the structure map $\delta_1^1$ preserves the bigrading.
\end{lem}
\begin{proof}
  The proof is analogous to that of Lemma~\ref{lem:DDgr}.
\end{proof}

One can also define $\CDTDm(\HD)$, as well as $\CDTAm(\HD)$, $\CATAm(\HD)$, and $\CATDm(\HD)$,  for a more general combinatorial bordered Heegaard diagram $\HD$ for a tangle, see \cite[Section 4]{pv}. 

Setting all $U_i$ variables to zero yields an  $(M,A)$-bigraded  type~$\DD$ structure $\CDTDt(\HD)$  over $(\ah{-\bdy^L\T}, \ah{-\bdy^R\T})$. This corresponds to only counting rectangles which do not contain any $O$'s and are of the first three types above.   Further collapsing the bigrading to a single grading $\delta = A - M$ yields a $\delta$-graded type~$\DD$ structure $\CDTDd(\HD)$ over $(\as{|\bdy^L\T|}, \as{|\bdy^R\T|})$.

\begin{prop}
  The structure $\CDTDd(\HD)$ is a type $\DD$ structure. Moreover, the structure map lowers the $\delta$-grading of homogeneous generators by one.
\end{prop}

\begin{proof}
  By Lemma~\ref{lem:DDgr}, $\CDTDt(\HD)$ is a type~$\DD$ structure for which $\delta^1$ lowers the Maslov grading by one, and preserves the Alexander grading. The claim now follows directly from the definition of $\CDTDd(\HD)$.
\end{proof}

Similarly, one can represent an unoriented tangle $\T$ by a Heegaard diagram $\HD$ that only has $X$ markings and no $O$ markings. Given a Heegaard diagram $\HD$ for an oriented or an unoriented tangle,  define an ungraded  type $\DD$ structure $\CDTDu (\HD)$ over $(\aus{|\bdy^L\T|}, \aus{|\bdy^R\T|})$ by counting rectangles  in the same way as  in the definition of $\CDTDt(\HD)$. 

\begin{prop}
  The structure $\CDTDu(\HD)$ is a type $\DD$ structure.
\end{prop}

\begin{proof}

  For any Heegaard diagram $\HD$ for an oriented or an unoriented tangle $\T$, there is some choice of replacing some $X$'s with $O$'s to obtain a Heegaard diagram $\HD'$ for an oriented tangle $\T'$ that is the same as $\T$ as an unoriented manifold.  By \cite{pv}, $\CDTDt(\HD')$ is a type~$\DD$ structure. But $\CDTDu(\HD)$ and $\CDTDt(\HD')$ are clearly (ungraded)  isomorphic, so $\CDTDu(\HD')$ is a type~$\DD$ structure too.
\end{proof}

\begin{prop}\label{prop:eq}
  Let $\HD_1 =(\Sigma_1, \alphas_1, \betas_1, \XX_1, \OO_1)$ and $\HD_2 = (\Sigma_2, \alphas_2, \betas_2, \XX_2, \OO_2)$ be Heegaard diagrams for the same oriented tangle, with $|\XX_1\cup \OO_1|-|\XX_2\cup \OO_2| = 2k\geq 0$.  Then
  \[
    \CDTDd(\HD_1)\simeq \CDTDd(\HD_2)\otimes_{\F{2}} V^{\otimes k}
  \]
  where $V = \F{2}\oplus\F{2}$ is supported in $\delta$-grading $0$.

  Similarly, let  $\HD_1 =(\Sigma_1, \alphas_1, \betas_1, \XX_1, \OO_1)$ and $\HD_2 = (\Sigma_2, \alphas_2, \betas_2, \XX_2, \OO_2)$ be Heegaard diagrams for tangles $\T_1$ and $\T_2$, where the tangles may be oriented or unoriented. Suppose that $\T_1$ and $\T_2$ are the same as unoriented tangles and $|\XX_1\cup \OO_1|-|\XX_2\cup \OO_2| = 2k\geq 0$.  Then
  \[
    \CDTDu(\HD_1)\simeq \CDTDu(\HD_2)\otimes_{\F{2}} (\F{2}\oplus\F{2})^{\otimes k}.
  \]
\end{prop}

\begin{proof}
  The first case follows directly from the fact that by \cite{pv}, $\CDTDt(\HD_1)\simeq \CDTDt(\HD_2) \otimes (\F{(0,0)}\oplus \F{(-1,-1)})^{\otimes k}$, where $\F{(i,j)}$ is the vector space $\F{2}$ in $(M,A)$-grading $(i,j)$.

  For the second case, replace some $X$'s with $O$'s, or $O$'s  with $X$'s, if necessary, to obtain diagrams $\HD_i'$ for tangles $\T_i'$ from $\HD_i$, so that $\T_1'$ and $\T_2'$ are the same as oriented tangles.  Then $\CDTDt(\HD_1')\simeq \CDTDt(\HD_2') \otimes (\F{(0,0)}\oplus \F{(-1,-1)})^{\otimes k}$, so $\CDTDu(\HD_1')\simeq \CDTDu(\HD_2')\otimes V^{\otimes k}$, since $\CDTDt(\HD_i')$ and $\CDTDu(\HD_i')$ are ungraded isomorphic. Since  $\CDTDd(\HD_i)\cong  \CDTDd(\HD_i')$, the statement follows.
\end{proof}

Proposition~\ref{prop:eq} implies that if $\HD_1$ and $\HD_2$ are two diagrams for $\T$ with the same number of markers $n$, then $ \CDTDd(\HD_1)\simeq \CDTDd(\HD_2)$ and $\CDTDu(\HD_1)\simeq \CDTDu(\HD_2)$. In view of this, here and afterwards we use $\CDTDd(\T, n)$ and $\CDTDu(\T, n)$ to denote the homotopy types of the structures arising from a diagram with $n$ markers associated to a tangle $\T$.

We end this section by stating a version of the gluing theorem for tangle Floer homology:

\begin{prop}
  If $\T = \T' \circ \T''$, then
  \begin{align*}
    \CDTDd (T, n' + n'') & \simeq \CDTAd (\T', n') \boxtimes \CDTDd (\T'', n''),\\
    \CDTDu (T, n' + n'') & \simeq \CDTAu (\T', n') \boxtimes \CDTDu (\T'', n'').
  \end{align*}
  An analogous equivalence holds for any other pair of bimodules for which the box tensor product is defined (e.g.\ a type~$\AA$ and a type~$\DA$ bimodule).
\end{prop}

\begin{proof}
  This is essentially \cite[Corollary~12.5]{pv}.
\end{proof}

Table~\ref{tab:all_flavors} summarizes the notation from this section.  

\begin{table}[h]
  \captionsetup{belowskip=10pt}
  \centering
  {
    \setlength{\extrarowheight}{3pt}%
    \renewcommand{\arraystretch}{2}%
    \begin{tabular}{|p{7.2cm}|c||c|c|}
      \hline
      \multicolumn{1}{|c|}{Description} & \makecell{Type~$\DD$\\ Bimodule} & \makecell{Associated\\ Algebras} & \makecell{Gradings\\ Endowed}\\[10pt]
      \hline
      \hline
      \makecell{Unblocked\footnote{The unblocked version of tangle Floer homology is not yet proven to be an invariant, as remarked in \cite[Section 1]{pv}.}, bigraded tangle Floer\\ bimodule of an oriented tangle $\T$} & $\CDTDm (\T)$ & \makecell{$\am{-\bdyL \T}$\\ $\am{-\bdyR \T}$} & $M, A$\\[10pt]
      \hline
      \makecell{Fully blocked, bigraded tangle Floer\\ bimodule of an oriented tangle $\T$} & $\CDTDt (\T, n)$ & \makecell{$\ah{-\bdyL \T}$\\ $\ah{-\bdyR \T}$} & $M, A$\\[10pt]
      \hline
      \makecell{Fully blocked, $\delta$-graded tangle Floer\\ bimodule of an oriented tangle $\T$} & $\CDTDd (\T, n)$ & \makecell{$\ad{-\bdyL \T} = \ads{|\bdyL \T|}$\\ $\ad{-\bdyR \T} = \ads{|\bdyR \T|}$} & $\delta$\\[10pt]
      \hline
      \makecell{Fully blocked, ungraded tangle Floer\\ bimodule of an unoriented tangle $\T$} & $\CDTDu (\T, n)$ & \makecell{$\au{\bdyL \T} = \aus{|\bdyL \T|}$\\ $\au{\bdyR \T} = \aus{|\bdyR \T|}$} & None\\[10pt]
      \hline
    \end{tabular}
  }
  \caption{A summary of the notation relevant to the four flavors of tangle Floer homology discussed in this section.}
  \label{tab:all_flavors}
\end{table}

%%%%%%%%%%%%%%%%%%%%%%%%%%%%%%%%%%%%%%%%%%%%%%%%%%%%%%%

%%%%%%%%%%%%%%%%%%%%%%%%%%%%%%%%%%%%%%%%%%%%%%%%%%%%%%%
% !TEX root = ../skein.tex
%%%%%%%%%%%%%%%%%%%%%%%%%%%%%%%%%%%%%%%%%%%%%%%%%%%%%%%

\section{The unoriented skein relation} % (fold)
\label{sec:proof}

%%%%%%%%%%%%%%%%%%%%%%%%%%%%%%%%%%%%%%%%%%%%%%%%%%%%%%%

%%%%%%%%%%%%%%%%%%%%%%%%%%%%%%%%%%%%%%%%%%%%%%%%%%%%%%%
% section proof
%%%%%%%%%%%%%%%%%%%%%%%%%%%%%%%%%%%%%%%%%%%%%%%%%%%%%%%

Our strategy for proving Theorem~\ref{thm:ourtheorem} is to prove it first for the simplest case, where the skein triple has one crossing, and then to apply a gluing theorem. 

More precisely, fix integers $n$ and $i$ with $n \geq 2$ and $1\leq i\leq n$. Let $\eT_\infty$ be the (unoriented) elementary $(n, n)$-tangle that consists of one crossing where the strand with the higher slope crosses over the strand with the lower slope, and there are $i-1$ horizontal strands running below the crossing and $n-i-1$ horizontal strands running above the crossing;  let $\eT_0$  be the  resolution of $\eT_\infty$ that results in only horizontal strands, and let $\eT_1$ be the resolution of $\eT_\infty$ that results in a cup and a cap, as in Figure~\ref{fig:hd_012}. Up until the very end of this section, we will  be working with the type~$\DD$ structures associated to these three tangles.

\begin{figure}[h]
  \centering
  \includegraphics[scale=1.05]{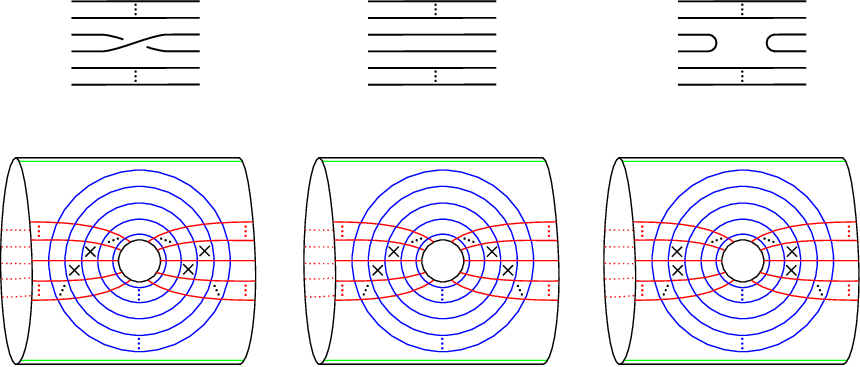}
  \caption{Top: From left to right, the three elementary tangles $\eT_\infty$, $\eT_0$, and  $\eT_1$. Bottom: The corresponding Heegaard diagrams $\HD_\infty$, $\HD_0$, and  $\HD_1$.}
  \label{fig:hd_012}
\end{figure}

We draw three Heegaard diagrams $\HD_\infty, \HD_0, \HD_1$ associated to $\eT_\infty, \eT_0, \eT_1$ respectively, with all marked points being $X$'s (since these are unoriented tangles); see Figure~\ref{fig:hd_012}. In Section ~\ref{sec:grading}, when we endow these tangles with orientations, we will be working with the same kind of diagrams, but with both $O$'s and $X$'s.  We label the $\alpha$ and $\beta$ curves for each diagram as in Section~\ref{ssec:tf}. The number of $\alpha$ arcs is $2 n + 2$ and the number of $\beta$ circles is $n + 1$ in each diagram.  Next, we combine all three diagrams into one diagram to obtain Figure~\ref{fig:hd_comb}.  Note that $\HD_\infty, \HD_0, \HD_1$ share the same $\alpha$ arcs ($2 n + 2$ in total) and marked points (i.e.~the $X$'s), and also all $\beta$ circles but one.  We label by $\beta_{i, \infty}$ (dark blue), $\beta_{i, 0}$ (green), $\beta_{i, 1}$ (purple) the three different circles corresponding to $\HD_\infty, \HD_0, \HD_1$ respectively.

\begin{figure}[h]
  \centering
  \labellist
  \pinlabel  \textcolor{red}{$c^{\mathit{FR}}_0$} at 148 35
  \pinlabel  \textcolor{red}{$c^{\mathit{FR}}_{n}$} at 147 87
  \pinlabel  \textcolor{red}{$c^{\mathit{FL}}_0$} at 10 35
  \pinlabel  \textcolor{red}{$c^{\mathit{FL}}_{n}$} at 9 85
  \pinlabel  \textcolor{red}{$c^{\mathit{BL}}_0$} at -5 37
  \pinlabel  \textcolor{red}{$c^{\mathit{BL}}_{n}$} at -5 85
  \pinlabel \rotatebox{90}{\textcolor{red}{$\dots$}} at -5 61
  \pinlabel \rotatebox{90}{\textcolor{red}{$\dots$}} at 12 61
  \pinlabel \rotatebox{90}{\textcolor{red}{$\dots$}} at 148 62
  \endlabellist
  \includegraphics[scale=1.7]{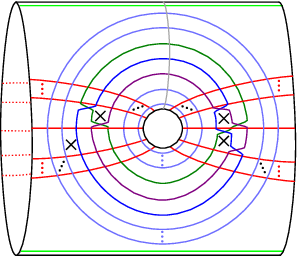}
  \caption{The diagram obtained by ``combining" $\HD_\infty$, $\HD_0$, and  $\HD_1$ so that they share the same $\alpha$ arcs, marked points, and all $\beta$ circles but one. Choosing the dark blue, green, or purple circle, and forgetting the other two, gives $\HD_\infty$, $\HD_0$, or  $\HD_1$, respectively.}
  \label{fig:hd_comb}
\end{figure}
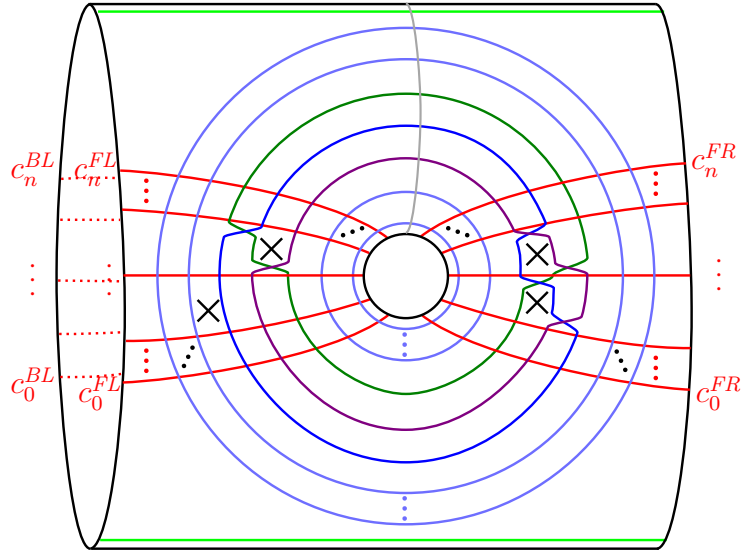

For ease of visualization, we cut open the Heegaard diagram along the indicated grey circle in Figure~\ref{fig:hd_comb}, and also delete the non-combinatorial regions (``the forbidden regions'' with the light green arcs) to obtain Figure~\ref{fig:hd_cut}. What we call the ``right grids" in Section~\ref{ssec:tf} combine to give the right half of the diagram as drawn in Figure~\ref{fig:hd_comb}, or equivalently the top half of the diagram as drawn in Figure~\ref{fig:hd_cut}. The ``left grids" combine to give the left half of  Figure~\ref{fig:hd_comb}, or equivalently the bottom half of Figure~\ref{fig:hd_cut}.

We denote the underlying surface for the combined diagram by $\Sigma$, and let the common $\alpha$ and $\beta$ curves inherit their labels from $\HD_{\infty}$, $\HD_0$, and $\HD_1$. Recall, for example,  that the $\alpha_j^L$'s are the $\alpha$ arcs that intersect the left boundary in Figure~\ref{fig:hd_comb}, and the $\alpha_j^R$'s the right boundary.  Note also the positions of $\alpha_i^L, \alpha_i^R$ and $\beta_{i,k}$; in particular, $\beta_{i, \infty}, \beta_{i, 0}, \beta_{i, 1}$ are between $\beta_{i-1}$ and $\beta_{i+1}$. We write $\alphas[L] = \set{\alpha_j^L}_{j=0}^n, \alphas[R] = \set{\alpha_j^R}_{j=0}^n$, and $\alphas = \alphas[L] \cup \alphas[R]$.  Likewise, for $k \in \set{\infty, 0, 1}$, we write $\betas[k] = \set{\beta_0, \dotsc, \beta_{i-1}, \beta_{i, k}, \beta_{i+1}, \dotsc, \beta_n}$, and $\betas = \betas[\infty] \cup \betas[0] \cup \betas[1]$.

We introduce a couple of more labels that we will use later. The front half and back half of $\bdy^R\Sigma$, as seen on Figure~\ref{fig:hd_comb}, are denoted $\bdy^{\mathit{FR}}\Sigma$ and $\bdy^{\mathit{BR}}\Sigma$ respectively, and translate to the top right edge and top left edge of the diagram in Figure~\ref{fig:hd_cut}, respectively. Similarly, we denote the front and back sides of $\bdy^L\Sigma$ by $\bdy^{\mathit{FL}}\Sigma$ and $\bdy^{\mathit{BL}}\Sigma$, respectively. We let \begin{align*}
  c^{\mathit{FL}}_i &= \alpha_i^L\cap \bdy^{\mathit{FL}}\Sigma, \qquad & c^{\mathit{FR}}_i &= \alpha_i^R\cap \bdy^{\mathit{FR}}\Sigma,\\
  c^{\mathit{BL}}_i &= \alpha_i^L\cap \bdy^{\mathit{BL}}\Sigma, \qquad & c^{\mathit{BR}}_i &= \alpha_i^R\cap \bdy^{\mathit{BR}}\Sigma.
\end{align*}
Last let $u_k$ and $v_k$ be the two intersection points in $\beta_{i,k}\cap \beta_{i,k+1}$, so that $u_k$ is to the left of $v_k$ as seen in Figure~\ref{fig:hd_cut}; in other words, $u_k$ lies on the boundary of the unique annulus in $\Sigma\setminus \betas$ with no $X$'s in it.

\begin{figure}
  \centering
  \labellist
  \pinlabel  \textcolor{red}{$c^{\mathit{FR}}_0$} at 295 280
  \pinlabel  \textcolor{red}{$c^{\mathit{FR}}_i$} at 295 414
  \pinlabel  \textcolor{red}{$c^{\mathit{FR}}_{n}$} at 295 554
  \pinlabel  \textcolor{red}{$c^{\mathit{FL}}_0$} at 295 242
  \pinlabel  \textcolor{red}{$c^{\mathit{FL}}_i$} at 295 172
  \pinlabel  \textcolor{red}{$c^{\mathit{FL}}_{n}$} at 295 35
  \pinlabel  \textcolor{red}{$c^{\mathit{BR}}_0$} at -20 280
  \pinlabel  \textcolor{red}{$c^{\mathit{BR}}_i$} at -20 414
  \pinlabel  \textcolor{red}{$c^{\mathit{BR}}_{n}$} at -20 554
  \pinlabel  \textcolor{red}{$c^{\mathit{BL}}_0$} at -20 242
  \pinlabel  \textcolor{red}{$c^{\mathit{BL}}_i$} at -20 172
  \pinlabel  \textcolor{red}{$c^{\mathit{BL}}_{n}$} at -20 35
  \pinlabel \rotatebox{90}{$\underbrace{\hspace{4.4cm}}$} at 335 136
  \pinlabel $\bdy^{\mathit{FL}}\Sigma$ at 375 136
  \pinlabel \rotatebox{90}{$\underbrace{\hspace{5.9cm}}$} at 335 410
  \pinlabel $\bdy^{\mathit{FR}}\Sigma$ at 375 410
  \pinlabel \rotatebox{90}{$\overbrace{\hspace{4.4cm}}$} at -55 136
  \pinlabel $\bdy^{\mathit{BL}}\Sigma$ at -95 136
  \pinlabel \rotatebox{90}{$\overbrace{\hspace{5.9cm}}$} at -55 410
  \pinlabel $\bdy^{\mathit{BR}}\Sigma$ at -95 410
  \pinlabel $v_0$ at 177 428
  \pinlabel $u_1$ at 140 503
  \pinlabel $u_{\infty}$ at 142 395
  \pinlabel $u_0$ at 140 154
  \pinlabel $v_1$ at 177 322
  \pinlabel $v_{\infty}$ at 181 79
  \pinlabel \textcolor{blue!60!}{$\beta_{n}\ldots$} at 50 -15
  \pinlabel \textcolor{blue}{$\beta_{i, \infty}$} at 145 -15
  \pinlabel \textcolor{Plum}{$\beta_{i, 1}$} at 107 -15
  \pinlabel \textcolor{OliveGreen}{$\beta_{i, 0}$} at 180 -15
  \pinlabel \textcolor{blue!60!}{$\ldots\beta_{0}$} at 232 -15
  \endlabellist
  \includegraphics[scale = .6]{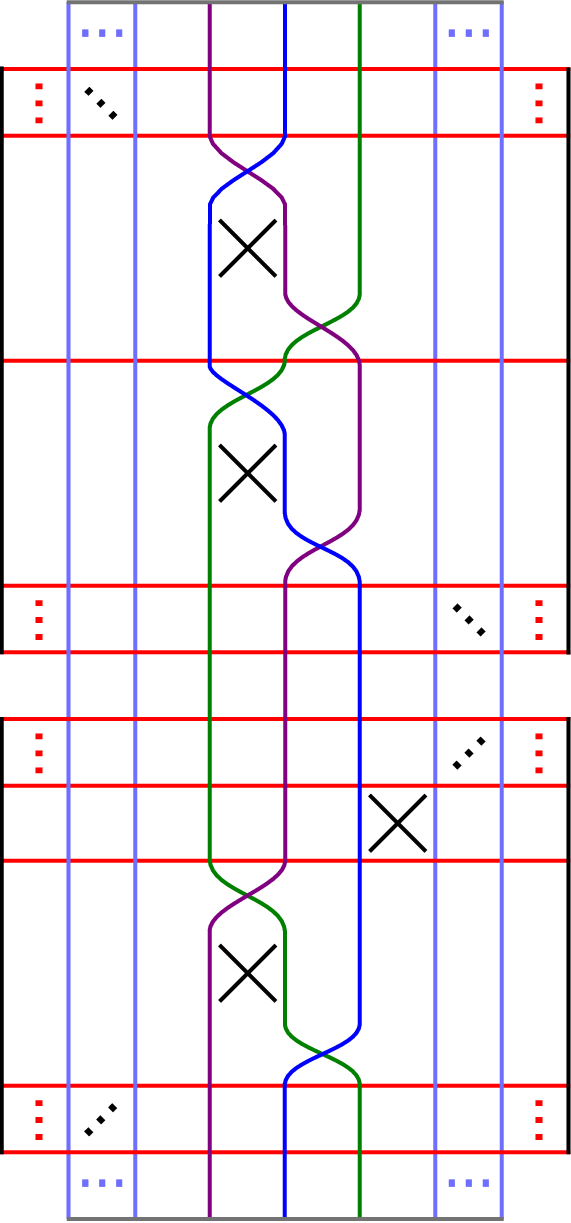}
  \vskip .5 cm
  \caption{The combined diagram for the three elementary tangles, obtained by cutting open the diagram in Figure~\ref{fig:hd_comb} along the indicated grey circle and deleting the non-combinatorial regions.}
  \label{fig:hd_cut}
\end{figure}

We let $\CDTDu (\HD_k)$ be the type~$\DD$ structure associated to $\HD_k$, for each $k \in \set{\infty, 0, 1}$; these are type~$\DD$ structures over $(\au{\bdyL\eT_k}, \au{\bdyR \eT_k}) = (\aus{n}, \aus{n})$. We also endow the set $\set{\infty, 0, 1}$ with an action by $\cycgrp{3}$ by identifying $\infty$ with $2$, so that $\infty + 1 = 0$ and $1 + 1 = \infty$.

In this setting, Theorem~\ref{thm:ourtheorem} will follow from the following proposition:

\begin{prop}
  \label{prop:main}
  There exists a type~$\DD$ homomorphism $f_0 \colon \CDTDu (\HD_0) \to \CDTDu (\HD_1)$ such that
  \[
    \CDTDu (\HD_\infty) \simeq \Cone (f_0 \colon \CDTDu (\HD_0) \to \CDTDu (\HD_1))
  \]
  as type~$\DD$ structures.
\end{prop}

From now on, we will write $\DDm{M}_k = \CDTDu (\HD_k)$. To prove Proposition~\ref{prop:main}, we will use Lemma~\ref{lem:hom_alg}. We shall define the morphisms to which we will apply Lemma~\ref{lem:hom_alg}, $f_k \colon \DDm{M}_k \to \DDm{M}_{k+1}, \phi_k \colon \DDm{M}_k \to \DDm{M}_{k+2}, \psi_k \colon \DDm{M}_k \to \DDm{M}_k$, by counting polygons.

\begin{defn}
  Given $\x \in \SS (\HD_k)$ and $\y \in \SS (\HD_\ell)$ (where $k, \ell \in \set{\infty, 0, 1}$), a \emph{polygon $p$ from $\x$ to $\y$} is an embedded disk in the surface $\Sigma$, which we also call $p$ by abuse of notation, satisfying the following conditions:
  \begin{enumerate}
    \item The boundary of $p$ lies on the $\alpha$ curves, $\beta$ curves and the boundary of $\Sigma$: $\bdy p \subset \alphas \cup \betas \cup \bdy \Sigma$;
    \item The interior angles of $p$ are all acute;
    \item If we write $\bdyb p = \bdy p \cap \betas$, then
      \[
        \bdy (\bdyb p) = \x - \y,
      \]
      where the orientation of $p$, and hence that of $\bdy p$, is inherited from $\Sigma$;
    \item Traversing each connected component of $\bdyb p$ in the inherited orientation, $\beta_{i, k}$ is always followed by $\beta_{i, k-1}$. In other words, if $u_k$ is a vertex of $p$, then the ``east-west" multiplicity of $p$ is greater than its ``north-south" multiplicity; similarly, if $v_k$ is a vertex of $p$, then the ``north-south" multiplicity of $p$ is greater than its ``east-west" multiplicity.
  \end{enumerate}
  A polygon $p$ from $\x$ and $\y$ is \emph{empty} if the interior of the embedded disk does not intersect any components of $\x$ (or equivalently $\y$), and also does not intersect $\XX$.
\end{defn}

It may be helpful to note here that in the proofs that follow in this section, $u_k$ and $v_k$ cannot arise as the shared corner of two polygons.

In Heegaard Floer homology, there is a more general notion of a \emph{domain}, which is a more general $2$-chain together with the initial and terminal generators.  In this paper, the domains that we investigate are always juxtapositions of multiple polygons: Given a polygon $p$ from $\x$ to $\y$, and a polygon $p'$ from $\y$ to $\z$, we can form the \emph{juxtaposition} $p * p'$, which is a domain from $\x$ to $\z$. The underlying $2$-chain of $p * p'$ is the sum of the underlying $2$-chains of $p$ and $p'$.

To clarify (following \cite{OSSbook}), when we speak of a domain, we always think of it as the underlying $2$-chain together with the initial and terminal generators $\x$ and $\y$. So if $(\x, \y) \neq (\x',\y')$, a domain from $\x$ to $\y$ is viewed as different from a domain from $\x'$ to $\y'$, even if the underlying $2$-chains are the same. The underlying $2$-chain is called the \emph{support} of the domain.

Fix a domain $p$. Like $\bdyb p$, we can similarly define $\bdya p = \bdy p \cap \alphas, \bdyL p = \bdy p \cap \bdy^L \Sigma$, and $\bdyR p = \bdy p \cap \bdy^R \Sigma$. Then $\bdya p, \bdyb p, \bdyL p, \bdyR p$ inherit an orientation from $\bdy p$ so that $\bdy p = \bdya p + \bdyb p + \bdyL p + \bdyR p$. We say that $p$ is a \emph{left-bordered domain} if $\bdyL p \neq \eset$, \emph{right-bordered domain} if $\bdyR p \neq \eset$, \emph{two-bordered domain} if it is both left-bordered and right-bordered, and \emph{interior domain} if $\bdyL p = \bdyR p = \eset$. (In \cite{bfh2, bimod, pv}, interior domains are called \emph{provincial domains} instead.)

As a warm-up example, we express the rectangles in the definition of $\delta_k^1$ (for $\CDTDu (\HD_k)$) as defined in Section~\ref{sec:background} in the present language. For $\CDTDm$ as in Section~\ref{sec:background}, the structure map counts rectangles of seven types; however, since we are only dealing with $\CDTDu$, only rectangles of the first three types are counted. In fact, in our present context, for $\x, \y \in \SS (\HD_k)$, a rectangle from $\x$ to $\y$ is just a polygon from $\x$ to $\y$ whose boundary consists of four oriented segments. The structure map $\delta_k^1$ then counts empty rectangles. Specifically, for $\x, \y \in \SS (\HD_k)$, denote the space of empty interior rectangles from $\x$ to $\y$ by $\eRectI_k (\x, \y)$, the space of empty left-bordered rectangles from $\x$ to $\y$ by $\eRectL_k (\x, \y)$, and the space of empty right-bordered rectangles from $\x$ to $\y$ by $\eRectR_k (\x, \y)$. Denote the union of these three spaces, the space of empty rectangles from $\x$ to $\y$, by $\eRect_k (\x, \y)$. Now
\begin{enumerate}
  \item if $r \in \eRectI_k (\x, \y)$, define $\algl{r} = \ideml{\x}$ and $\algr{r} = \idemr{\x}$;
  \item if $r \in \eRectL_k (\x, \y)$, then the oriented arc $\bdyL r$ is either an arc on $\bdyFL \Sigma$ or an arc on $\bdyBL \Sigma$. In the former case, it goes from $\frontleftbdy{c_{j_1}}$ to $\frontleftbdy{c_{j_2}}$, for some $j_1 > j_2$; in the latter case, the arc goes from $\backleftbdy{c_{j_1}}$ to $\backleftbdy{c_{j_2}}$, for some $j_1 < j_2$. In either case, define $\algl{r}$ to be the bijection from $\unocl{\x}$ to $\unocl{\y}$ that sends $j_2$ to $j_1$ and is the identity elsewhere. Define $\algr{r} = \idemr{\x}$;
  \item if $r \in \eRectR_k (\x, \y)$, then the oriented arc $\bdyR r$ is either an arc on $\bdyFR \Sigma$ or an arc on $\bdyBR \Sigma$. In the former case, it goes from $\frontrightbdy{c_{j_1}}$ to $\frontrightbdy{c_{j_2}}$, for some $j_1 < j_2$; in the latter case, the arc goes from $\backrightbdy{c_{j_1}}$ to $\backrightbdy{c_{j_2}}$, for some $j_1 > j_2$. In either case, define $\algr{r}$ to be the bijection from $\unocr{\y}$ to $\unocr{\x}$ that sends $j_1$ to $j_2$ and is the identity elsewhere. Define $\algl{r} = \ideml{\x}$.
\end{enumerate}
Then we can write $\delta_k^1 \colon M_k \to \an \otimes M_k \otimes \an$ as
\[
  \delta_k^1 (\x) = \sum_{\y \in \SS (\HD_k)} \sum_{r \in \eRect_k (\x, \y)} \algl{r}\otimes \y \otimes \algr{r}.
\]

We now turn to defining the polygons to be counted in our maps $f_k, \phi_k$ and $\psi_k$.

\begin{defn}
  Let $\x \in \SS (\HD_k)$.
  \begin{enumerate}
    \item For $\y \in \SS (\HD_{k+1})$, a \emph{triangle from $\x$ to $\y$} is a polygon from $\x$ to $\y$ whose boundary consists of three oriented segments. Note that triangles are always interior domains. 
    \item For  $\y \in \SS (\HD_{k+1})$, a \emph{pentagon from $\x$ to $\y$} is a polygon from $\x$ to $\y$ whose boundary consists of five oriented segments. Note that a pentagon can be a left-bordered, right-bordered, or interior domain. 
        \item For $\y \in \SS (\HD_{k+2})$, a \emph{quadrilateral from $\x$ to $\y$} is a polygon from $\x$ to $\y$ whose boundary consists of four oriented segments. Note that quadrilaterals are always interior domains, and always empty. 
    \item For $\y \in \SS (\HD_{k+2})$, a \emph{hexagon from $\x$ to $\y$} is a polygon from $\x$ to $\y$ whose boundary consists of six oriented segments. Note that a hexagon can only be a right-bordered or interior domain. 
    \item For $\y \in \SS (\HD_{k})$, a \emph{heptagon from $\x$ to $\y$} is a polygon from $\x$ to $\y$ whose boundary consists of seven oriented segments. Note that a heptagon can only be an interior domain. 
  \end{enumerate}
  We denote the respective spaces of each type of polygons  by $\Tri_k (\x, \y)$, $\Pent_k (\x, \y)$, $\Quad_k (\x, \y)$, $\Hex_k (\x, \y)$, and $\Hept_k (\x, \y)$.
  We also write, for example, $\ePent_k (\x, \y)$ for the space of empty pentagons from $\x$ to $\y$, and $\ePentL_k (\x, \y)$ and $\eHexR_k (\x, \y)$ for the obvious spaces of left-bordered and right-bordered domains. Triangles and quadrilaterals are called \emph{triangle-like} polygons; rectangles, pentagons, hexagons and heptagons are called \emph{rectangle-like} polygons.
\end{defn}

We emphasize here that $\x$ and $\y$ must be generators of the appropriate diagrams for these spaces to make sense; for example, to mention $\eHexI_k (\x, \y)$, $\x$ must be in $\SS (\HD_k)$ and $\y$ must be in $\SS (\HD_{k+2})$.

Like rectangles, other polygons have algebra elements associated to them.

\begin{defn}
  Let $p$ be a polygon from $\x$ to $\y$, where $\x$ and $\y$ are generators in their respective type~$\DD$ bimodules; then $\algl{p}$ and $\algr{p}$ are defined as follows.
  \begin{enumerate}
    \item If $p$ is interior, then define $\algl{p} = \ideml{\x}$ and $\algr{p} = \idemr{\x}$. (Recall that $\ideml{\x} = \ideml{\y}$ and $\idemr{\x} = \idemr{\y}$.)
    \item If $p$ is left-bordered, then the oriented arc $\bdyL p$ is either an arc on $\bdyFL \Sigma$ or an arc on $\bdyBL \Sigma$. In the former case, it goes from $\frontleftbdy{c_{j_1}}$ to $\frontleftbdy{c_{j_2}}$, for some $j_1 > j_2$; in the latter case, the arc goes from $\backleftbdy{c_{j_1}}$ to $\backleftbdy{c_{j_2}}$, for some $j_1 < j_2$. In either case, define $\algl{p}$ to be the bijection from $\unocl{\x}$ to $\unocl{\y}$ that sends $j_2$ to $j_1$ and is the identity elsewhere. Define $\algr{p} = \idemr{\x}$.
    \item If $p$ is right-bordered, then the oriented arc $\bdyR p$ is either an arc on $\bdyFR \Sigma$ or an arc on $\bdyBR \Sigma$. In the former case, it goes from $\frontrightbdy{c_{j_1}}$ to $\frontrightbdy{c_{j_2}}$, for some $j_1 < j_2$; in the latter case, the arc goes from $\backrightbdy{c_{j_1}}$ to $\backrightbdy{c_{j_2}}$, for some $j_1 > j_2$. In either case, define $\algr{p}$ to be the bijection from $\unocr{\y}$ to $\unocr{\x}$ that sends $j_1$ to $j_2$ and is the identity elsewhere. Define $\algl{p} = \ideml{\x}$.
  \end{enumerate}
\end{defn}

With these definitions, we can now define the following polygon counts, which are morphisms of type~$\DD$ bimodules:
\begin{enumerate}
  \item The triangle count $\TT_k \colon \DDm{M}_k \to \DDm{M}_{k+1}$ is defined by
    \[
      \TT_k (\x) = \sum_{\y \in \SS (\HD_{k+1})} \sum_{p \in \eTri_k (\x, \y)} \algl{p} \otimes \y \otimes \algr{p}.
    \]
  \item The pentagon count $\PP_k \colon \DDm{M}_k \to \DDm{M}_{k+1}$ is defined by
    \[
      \PP_k (\x) = \sum_{\y \in \SS (\HD_{k+1})} \sum_{p \in \ePent_k (\x, \y)} \algl{p} \otimes \y \otimes \algr{p}.
    \]
  \item The quadrilateral count $\Q_k \colon \DDm{M}_k \to \DDm{M}_{k+2}$ is defined by
    \[
      \Q_k (\x) = \sum_{\y \in \SS (\HD_{k+2})} \sum_{p \in \eQuad_k (\x, \y)} \algl{p} \otimes \y \otimes \algr{p}.
    \]
  \item The hexagon count $\HH_k \colon \DDm{M}_k \to \DDm{M}_{k+2}$ is defined by
    \[
      \HH_k (\x) = \sum_{\y \in \SS (\HD_{k+2})} \sum_{p \in \eHex_k (\x, \y)} \algl{p} \otimes \y \otimes \algr{p}.
    \]
  \item The heptagon count $\K_k \colon \DDm{M}_k \to \DDm{M}_k$ is defined by
    \[
      \K_k (\x) = \sum_{\y \in \SS (\HD_k)} \sum_{p \in \eHept_k (\x, \y)} \algl{p} \otimes \y \otimes \algr{p}.
    \]
\end{enumerate}

We can finally define the morphisms $f_k, \phi_k$ and $\psi_k$:
\begin{enumerate}
  \item The morphism $f_k \colon \DDm{M}_k \to \DDm{M}_{k+1}$ is defined by
    \[
      f_k = \TT_k + \PP_k.
    \]
  \item The morphism $\phi_k \colon \DDm{M}_k \to \DDm{M}_{k+2}$ is defined by
    \[
      \phi_k = \Q_k + \HH_k.
    \]
  \item The morphism $\psi_k \colon \DDm{M}_k \to \DDm{M}_k$ is defined by
    \[
      \psi_k = \K_k.
    \]
\end{enumerate}

\begin{lem}
  \label{lem:cond1}
  The morphisms $f_k$ are type~DD homomorphisms, i.e. they satisfy Condition~(1) of Lemma~\ref{lem:hom_alg}. In fact, $\TT_k$ and $\PP_k$ are both type~DD homomorphisms.
\end{lem}

\begin{proof}
  The proof is similar to that of Lemma~3.3 of \cite{wong}, which is in turn inspired by Lemma~3.1 of \cite{most}. In fact, we shall see that
  \begin{gather}
    \mathcenter{
      \begin{tikzpicture}
        \node at (6,0) (from) {};
        \node at (6,-2) (delta) {$\delta_{k+1}^1$};
        \node at (6,-1) (action) {$\TT_k$};
        \node at (5,-3) (lprod) {$\mu_{\aun}$};
        \node at (7,-3) (rprod) {$\mu_{\aun}$};
        \node at (6,-4) (to) {};
        \node at (5,-4) (ablank) {};
        \node at (7,-4) (bblank) {};
        \draw[als] (delta) to (lprod);
        \draw[al] (action) to (lprod);
        \draw[algarrow] (lprod) to (ablank);
        \draw[ars] (delta) to (rprod);
        \draw[ar] (action) to (rprod);
        \draw[algarrow] (rprod) to (bblank);
        \draw[modarrow] (from) to (action);
        \draw[modarrow] (action) to (delta);
        \draw[modarrow] (delta) to (to);
      \end{tikzpicture}
    }
    +
    \mathcenter{
      \begin{tikzpicture}
        \node at (6,0) (from) {};
        \node at (6,-1) (delta) {$\delta_k^1$};
        \node at (6,-2) (action) {$\TT_k$};
        \node at (5,-3) (lprod) {$\mu_{\aun}$};
        \node at (7,-3) (rprod) {$\mu_{\aun}$};
        \node at (6,-4) (to) {};
        \node at (5,-4) (ablank) {};
        \node at (7,-4) (bblank) {};
        \draw[al] (delta) to (lprod);
        \draw[als] (action) to (lprod);
        \draw[algarrow] (lprod) to (ablank);
        \draw[ar] (delta) to (rprod);
        \draw[ars] (action) to (rprod);
        \draw[algarrow] (rprod) to (bblank);
        \draw[modarrow] (from) to (delta);
        \draw[modarrow] (delta) to (action);
        \draw[modarrow] (action) to (to);
      \end{tikzpicture}
    } = 0. \label{eqn:tri_homo}\\
    \mathcenter{
      \begin{tikzpicture}
        \node at (6,0) (from) {};
        \node at (6,-2) (delta) {$\delta_{k+1}^1$};
        \node at (6,-1) (action) {$\PP_k$};
        \node at (5,-3) (lprod) {$\mu_{\aun}$};
        \node at (7,-3) (rprod) {$\mu_{\aun}$};
        \node at (6,-4) (to) {};
        \node at (5,-4) (ablank) {};
        \node at (7,-4) (bblank) {};
        \draw[als] (delta) to (lprod);
        \draw[al] (action) to (lprod);
        \draw[algarrow] (lprod) to (ablank);
        \draw[ars] (delta) to (rprod);
        \draw[ar] (action) to (rprod);
        \draw[algarrow] (rprod) to (bblank);
        \draw[modarrow] (from) to (action);
        \draw[modarrow] (action) to (delta);
        \draw[modarrow] (delta) to (to);
      \end{tikzpicture}
    }
    +
    \mathcenter{
      \begin{tikzpicture}
        \node at (6,0) (from) {};
        \node at (6,-1) (delta) {$\delta_k^1$};
        \node at (6,-2) (action) {$\PP_k$};
        \node at (5,-3) (lprod) {$\mu_{\aun}$};
        \node at (7,-3) (rprod) {$\mu_{\aun}$};
        \node at (6,-4) (to) {};
        \node at (5,-4) (ablank) {};
        \node at (7,-4) (bblank) {};
        \draw[al] (delta) to (lprod);
        \draw[als] (action) to (lprod);
        \draw[algarrow] (lprod) to (ablank);
        \draw[ar] (delta) to (rprod);
        \draw[ars] (action) to (rprod);
        \draw[algarrow] (rprod) to (bblank);
        \draw[modarrow] (from) to (delta);
        \draw[modarrow] (delta) to (action);
        \draw[modarrow] (action) to (to);
      \end{tikzpicture}
    }
    +
    \mathcenter{
      \begin{tikzpicture}
        \node at (6,0) (from) {};
        \node at (6,-1.5) (action) {$\PP_k$};
        \node at (5,-3) (lprod) {$d_{\aun}$};
        \node at (6,-4) (to) {};
        \node at (5,-4) (ablank) {};
        \node at (7,-4) (bblank) {};
        \draw[ar] (action) to (bblank);
        \draw[als] (action) to (lprod);
        \draw[algarrow] (lprod) to (ablank);
        \draw[Amodar] (from) to (action);
        \draw[Amodar] (action) to (to);
      \end{tikzpicture}
    }
    +
    \mathcenter{
      \begin{tikzpicture}
        \node at (6,0) (from) {};
        \node at (6,-1.5) (action) {$\PP_k$};
        \node at (7, -3) (rprod) {$d_{\aun}$};
        \node at (6,-4) (to) {};
        \node at (5,-4) (ablank) {};
        \node at (7,-4) (bblank) {};
        \draw[al] (action) to (ablank);
        \draw[ars] (action) to (rprod);
        \draw[algarrow] (rprod) to (bblank);
        \draw[Amodar] (from) to (action);
        \draw[Amodar] (action) to (to);
      \end{tikzpicture}
    } = 0. \label{eqn:pent_homo}
  \end{gather}

  We first prove Equation~\ref{eqn:tri_homo}. Fix a domain that can be written as a juxtaposition $p * r$ (resp.\ $r * p$), where $p$ is a triangle and $r$ is a rectangle. Recall that triangles are always interior polygons; this means that $\algl{p} = \ideml{\x}$ and $\algr{p} = \idemr{\x}$, and so the algebra elements that a juxtaposition $p * r$ (resp.\ $r * p$) outputs are always $\algl{r}$ and $\algr{r}$. At least one of these is an idempotent element, since the rectangle $r$ cannot be two-bordered. Focusing on $p * r$ (resp.\ $r * p$), there are three cases; the polygons may be disjoint, their interiors may overlap, or they may share a common corner.

  If the polygons are disjoint or if their interiors overlap, the domain can be alternatively decomposed as $r' * p'$ (resp.\ $p' * r'$), where $r$ and $r'$ have the same support, and so do $p$ and $p'$. Thus, the domain contributes twice to the sum in the first equation above.  Since the base rings are of characteristic $2$, the total contribution is $0$.  The output algebra elements are obviously the same for both juxtapositions, since the underlying rectangles are the same in the two canceling juxtapositions. All possibilities of $p * r$ and $r * p$, where $p$ and $r$ have overlapping interiors, are listed in Figure~\ref{fig:tri_overlap}.  In this and following figures, all possibilities of composite domains are to be understood up to rotation by $\pi$. The reader may verify that these lists are complete by examining Figure~\ref{fig:hd_cut} and using basic planar geometry.

  \begin{figure}[h]
    \centering
    \includegraphics[scale = .41]{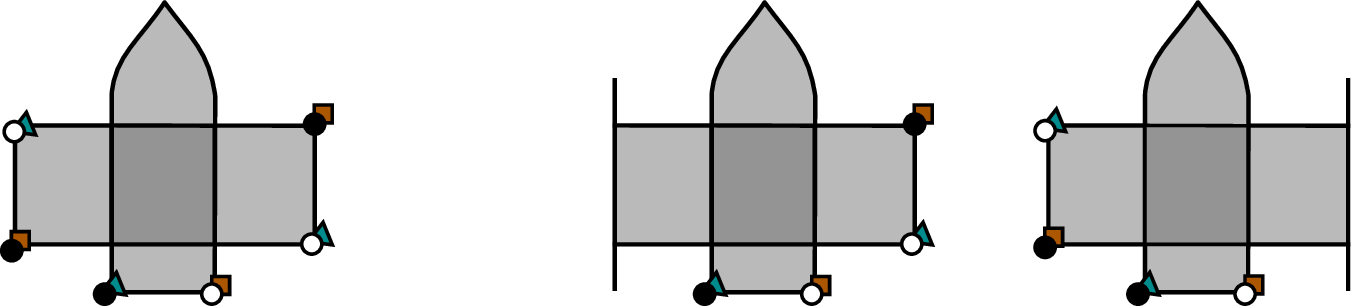}
    \caption{All possibilities of $p * r$ and $r * p$, where $p$ and $r$ are a triangle and a rectangle that have overlapping interiors, along with the alternate decomposition. In each figure, the black dots represent a generator $\x$, the brown squares a generator $\y$, the teal triangles a generator $\gt$, and the white dots a generator $\z$; the domain can be decomposed as $p * r$, where $p \in \eTri (\x, \y)$ and $p \in \eRectI (\y, \z)$, or as $r' * p'$, where $r' \in \eRectI (\x, \gt)$ and $p' \in \eTri (\gt, \z)$.}
    \label{fig:tri_overlap}
  \end{figure}

  If, instead, the polygons $p$ and $r$ share a common corner, $p * r$ (resp.\ $r * p$) always has exactly one alternative decomposition as $r' * p'$ (resp.\ $p' * r'$), where $p$ and $p'$ are triangles with distinct supports, and $r$ and $r'$ are rectangles with distinct supports. See Figure~\ref{fig:tri_corner_sp}. Again, this domain does not contribute to the sum. It is also apparent from the same figure that the intersections of the domain in question with both $\bdyL \Sigma$ and $\bdyR \Sigma$ (which may or may not be empty) are the same in both decompositions, and so again the output algebra elements are the same. All possibilities where $p$ and $r$ share a common corner are listed in Figure~\ref{fig:tri_corner}.

  \begin{figure}[h]
    \centering
    \includegraphics[scale = .41]{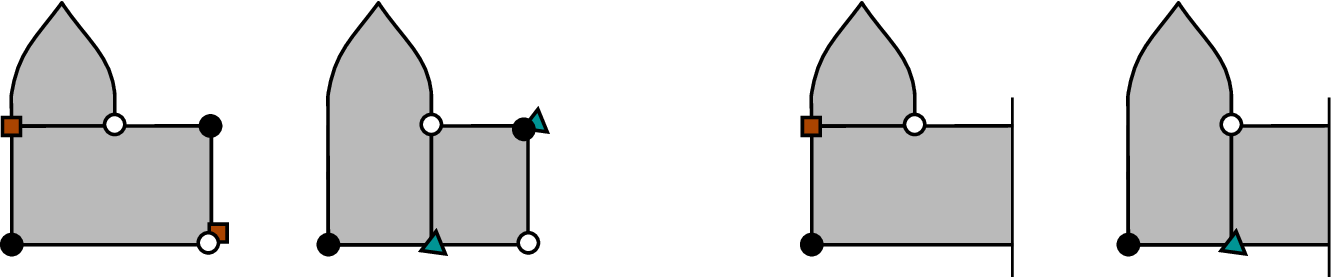} \caption{The two figures on the left show the two decompositions of the same interior domain; the first figure decomposes the domain into $r * p$, where $r \in \eRectI (\x, \y)$ and $p \in \eTri (\y, \z)$, while the second figure decomposes the domain into $p' * r'$, where $r' \in \eTri (\x, \gt)$ and $p' \in \eRectI  (\gt, \z)$. The two figures on the right show a similar decomposition for a bordered domain.}
    \label{fig:tri_corner_sp}
  \end{figure}

  \begin{figure}[h]
    \centering
    \includegraphics[scale = .41]{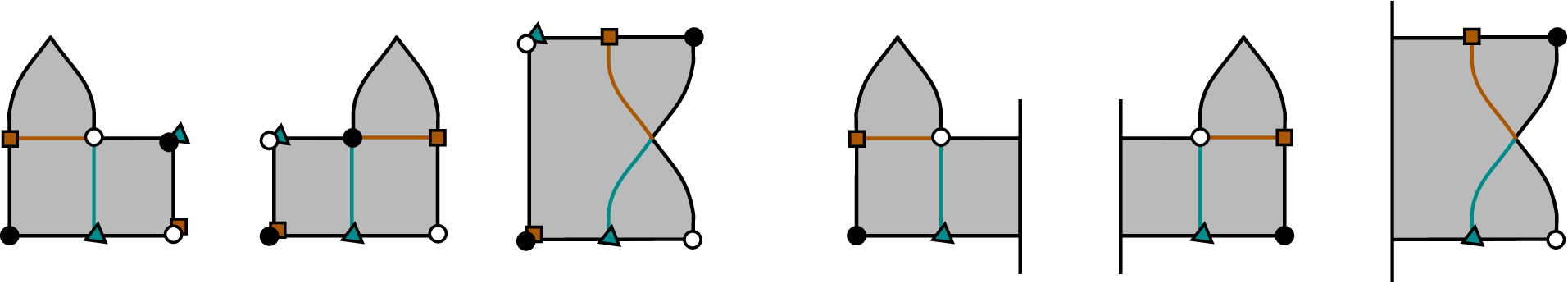} \caption{All possibilities of $p * r$ or $r * p$, where $p$ and $r$ are a triangle and a rectangle that share a common corner, along with the alternate decomposition. In each example, the teal cut gives one decomposition of the domain, and the brown cut gives the canceling decomposition.}
    \label{fig:tri_corner}
  \end{figure}

  We now turn to proving Equation~\ref{eqn:pent_homo}. Fix a domain that can be written as a juxtaposition $p * r$ (resp.\ $r * p$), where $p$ is a pentagon and $r$ is a rectangle. There are four cases this time: The two polygons may be disjoint, their interiors may overlap, they may share exactly one common corner, or they may share exactly one edge and two corners. The first three cases are similar to the cases with triangles. Note, however, that there is an additional possibility in the case where $p$ and $r$ share exactly one common corner: the alternative decomposition is not necessarily $r' * p'$ (resp.\ $p' * r'$), but is sometimes $p' * r'$ (resp.\ $r' * p'$); see the fourth figure from the left in the top row of Figure~\ref{fig:pent_corner}. All possibilities for the second case and the third case are listed in Figures~\ref{fig:pent_overlap} and \ref{fig:pent_corner} respectively.

  \begin{figure}
    \centering
    \includegraphics[scale = .41]{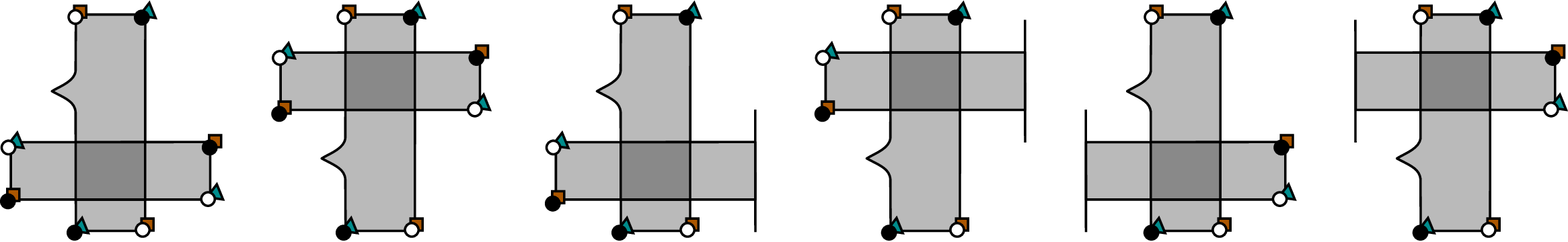} \caption{All possibilities of $p * r$ and $r * p$, where $p$ and $r$ are a pentagon and a rectangle that have overlapping interiors, along with the alternate decomposition.}
    \label{fig:pent_overlap}
  \end{figure}

  \begin{figure}
    \centering
    \includegraphics[scale = .37]{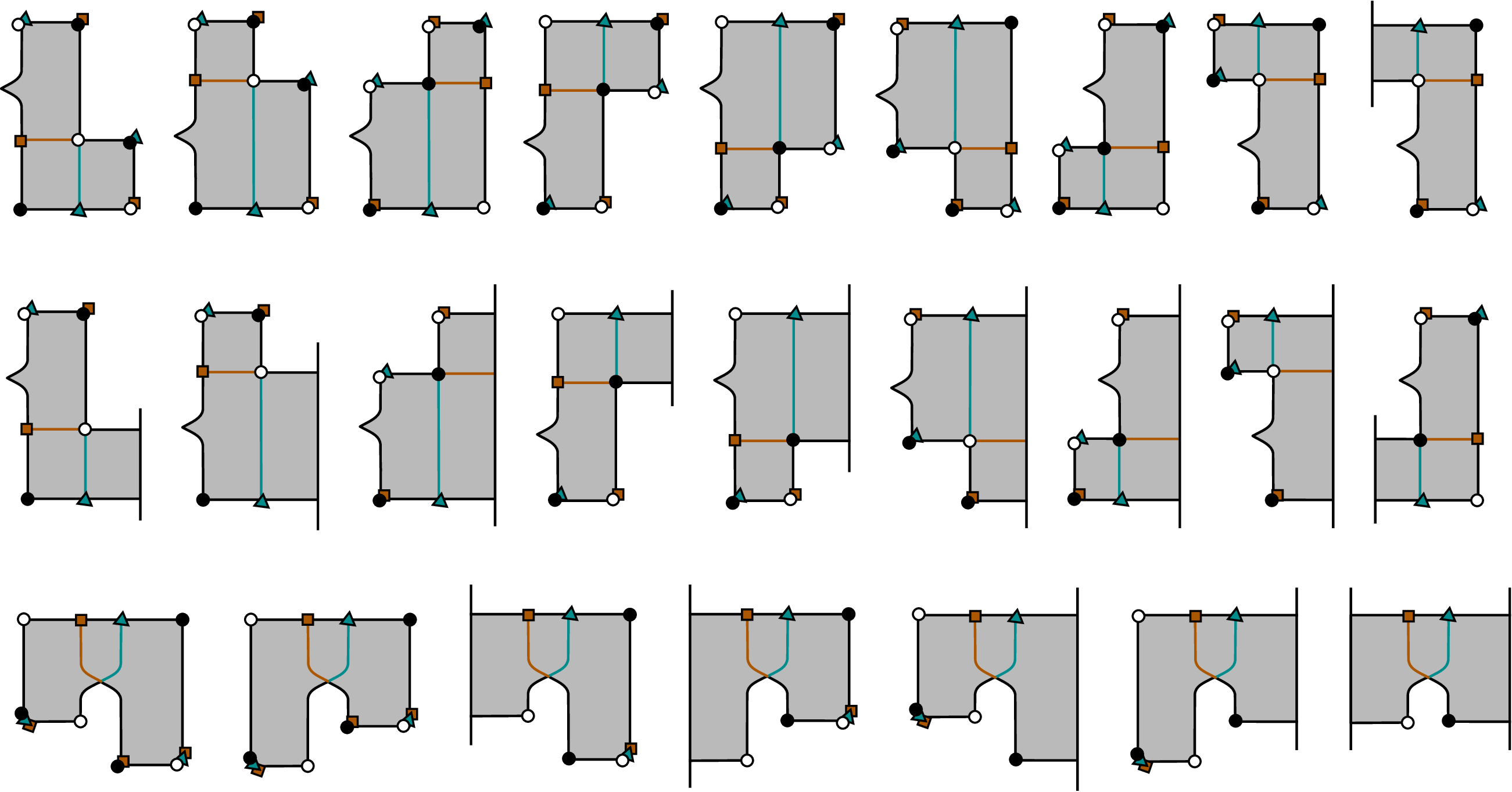} \caption{All possibilities of $p * r$ and $r * p$, where $p$ and $r$ are a pentagon and a rectangle that share exactly one common corner, along with the alternate decomposition.}
    \label{fig:pent_corner}
  \end{figure}

  The last case, where the two polygons share exactly one edge and two corners, can only occur if both $p$ and $r$ are right-bordered (or left-bordered). Let us illustrate this case more closely by the following example. 

  \begin{figure}[h]
    \centering
    \labellist
    \pinlabel $r$ at 42 150
    \pinlabel $p$ at 42 80
    \pinlabel $p'$ at 470 108
    \pinlabel $\xrightarrow{d_{\an}}$ at 595 100
    \pinlabel $\cdot$ at 150 100
    \pinlabel $=$ at 238 100
    \pinlabel \small{$\algr{r}$} at 112 -15
    \pinlabel \small{$\algr{p}$} at 185 -15
    \pinlabel \small{$\algr{r}\cdot\algr{p}$} at 285 -15
    \pinlabel \small{$\algr{p'}$} at 539 -15
    \pinlabel \small{$d_{\an}\algr{p'}$} at 657 -15
    \pinlabel $\alpha_{j_1}^R$ at -10 15
    \pinlabel $\alpha_{j_2}^R$ at -10 125
    \pinlabel $\alpha_{j_3}^R$ at -10 165
    \pinlabel $\alpha_{j_1}^R$ at 418 15
    \pinlabel $\alpha_{j_2}^R$ at 418 125
    \pinlabel $\alpha_{j_3}^R$ at 418 165
    \endlabellist
    \includegraphics[scale = .5]{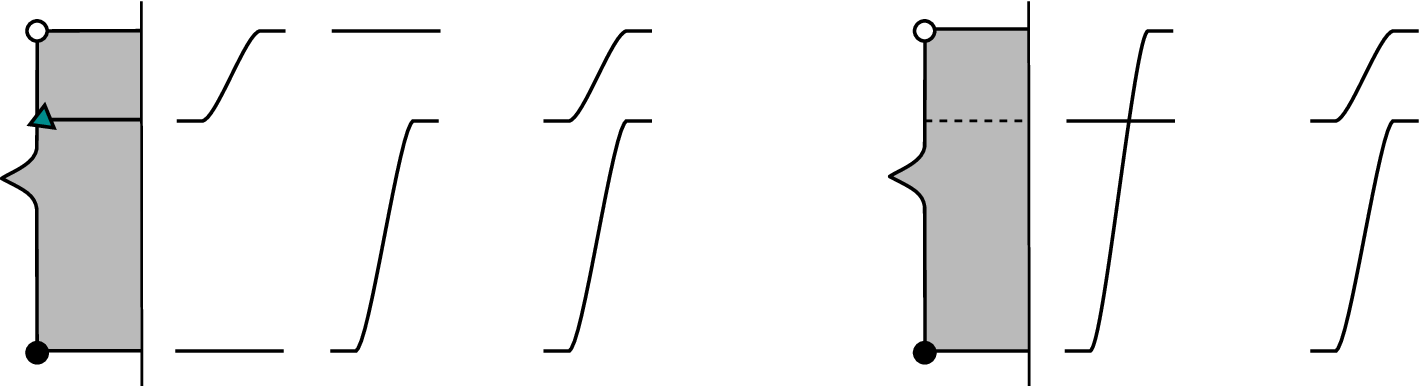} \vskip .3 cm
    \caption{Left: The domain $p * r$. The two algebra elements $\algr{r}$ and $\algr{p}$ in $\aun$ multiply to a non-zero algebra element. Right: The same domain, considered as a single pentagon $p'$. The differential of the algebra element $\algr{p'}$ is the same algebra element as $\algr{r} \cdot \algr{p}$. The generator $\x$ is represented by a black dot, $\y$ by a teal triangle, and $\z$ by a white dot.}
    \label{fig:pent_d}
  \end{figure}

  Consider the left of Figure~\ref{fig:pent_d}. Here, the generators are $\x \in \SS (\HD_k)$ and $\y, \z \in \SS (\HD_{k+1})$, and the relevant components of $\x, \y, \z$ lie on $\alpha_{j_1}^R, \alpha_{j_2}^R, \alpha_{j_3}^R$ respectively. There is a pentagon $p$ from $\x$ to $\y$, and a rectangle $r$ from $\y$ to $\z$. This gives us a term in $\bdy \PP_k (\x)$ that is illustrated in Figure~\ref{fig:pent_d_tree_1}. At first glance, the domain $p * r$ seems to contribute no other terms in $\bdy \PP_k (\x)$; however, upon closer inspection, we notice that the domain $p' = p * r$ is itself a pentagon! This means that $\ideml{\x} \otimes \z \otimes \algr{p'}$ is a term in $\PP_k (\x)$. Furthermore, note that since there is a pentagon $p$ from $\x$ to $\y$, $\alpha_{j_2}^R$ must be unoccupied by $\x$, and so we must have $j_2 \in \unocr{\x}$. For simplicity, let us assume that there are no other $j$ with $j_1 < j < j_3$ such that $j \in \unocr{\x}$. Then the algebra element $\algr{p'}$ is exactly as shown in the right hand side of Figure~\ref{fig:pent_d}, and we can take its differential to obtain $d_{\aun} \algr{p'}$. This gives us another term in $\bdy \PP_k (\x)$, as in Figure~\ref{fig:pent_d_tree_2}. Note that $d_{\aun} \algr{p'} = \algr{r} \cdot \algr{p}$, and so we see that the two terms above cancel out in this simple case.

  \begin{figure}[h]
    \centering
    \begin{subfigure}[b]{0.45\textwidth}
      \centering
      \[
        \mathcenter{
          \begin{tikzpicture}
            \node at (6,0) (from) {$\x$};
            \node at (6,-1) (action1) {$\PP_k$};
            \node at (6,-2) (first) {$\y$};
            \node at (6,-3) (action2) {$\delta_{k+1}^1$};
            \node at (6,-6) (to) {$\z$};
            \node at (4,-2) (lfirst) {$\ideml{\x}$};
            \node at (8,-2) (rfirst) {$\algr{p}$};
            \node at (5,-4) (lsecond) {$\ideml{\x}$};
            \node at (7,-4) (rsecond) {$\algr{r}$};
            \node at (4,-5) (lprod) {$\mu_{\an}$};
            \node at (8,-5) (rprod) {$\mu_{\an}$};
            \node at (4,-6) (lto) {$\ideml{\x}$};
            \node at (8,-6) (rto) {$\algr{r}\cdot\algr{p}$};
            \draw[modarrow] (from) to (action1);
            \draw[modarrow] (action1) to (first);
            \draw[modarrow] (first) to (action2);
            \draw[modarrow] (action2) to (to);
            \draw[als] (action1) to (lfirst);
            \draw[ars] (action1) to (rfirst);
            \draw[als] (action2) to (lsecond);
            \draw[ars] (action2) to (rsecond);
            \draw[algarrow] (lfirst) to (lprod);
            \draw[algarrow] (rfirst) to (rprod);
            \draw[als] (lsecond) to (lprod);
            \draw[ars] (rsecond) to (rprod);
            \draw[algarrow] (lprod) to (lto);
            \draw[algarrow] (rprod) to (rto);
          \end{tikzpicture}
        }.
      \]
      \subcaption{The term in $\bdy \PP_k (\x)$ arising from $p * r$.}
      \label{fig:pent_d_tree_1}
    \end{subfigure}
    \begin{subfigure}[b]{0.45\textwidth}
      \centering
      \[
        \mathcenter{
          \begin{tikzpicture}
            \node at (6,0) (from) {$\x$};
            \node at (6,-1) (action1) {$\PP_k$};
            \node at (6,-4) (to) {$\z$};
            \node at (8,-2) (rfirst) {$\algr{p'}$};
            \node at (8,-3) (rprod) {$d_{\an}$};
            \node at (4,-4) (lto) {$\ideml{\x}$};
            \node at (8,-4) (rto) {$d_{\aun} \algr{p'}$};
            \draw[modarrow] (from) to (action1);
            \draw[modarrow] (action1) to (to);
            \draw[al] (action1) to (lto);
            \draw[ars] (action1) to (rfirst);
            \draw[algarrow] (rfirst) to (rprod);
            \draw[algarrow] (rprod) to (rto);
          \end{tikzpicture}
        }.
      \]
      \subcaption{The term in $\bdy \PP_k (\x)$ arising from $p'$.}
      \label{fig:pent_d_tree_2}
    \end{subfigure}
    \caption{Two terms in $\bdy \PP_k (\x)$ that cancel each other.}
    \label{fig:pent_d_tree}
  \end{figure}

  In general, there may be other $j$'s with $j_1 < j < j_3$ such that $j \in \unocr{\x}$. If there are $m$ such $j$'s, then $d_{\aun} \algr{p'}$ is a sum of $m$ terms, one for each $j$. See Figure~\ref{fig:pent_d_mult} for an illustration. This means that Figure~\ref{fig:pent_d_tree_2} now represents $m$ terms in $\bdy \PP_k (\x)$. Each of these terms corresponds to a decomposition of the domain $p'$ into some $p * r$ or $r * p$, as in Figure~\ref{fig:pent_d}, and consequently cancels with a term as in Figure~\ref{fig:pent_d_tree_1}. The situation when $p'$ is left-bordered is completely analogous.

  \begin{figure}[h]
    \centering
    \labellist
    \pinlabel $p'$ at 42 295
    \pinlabel $\xrightarrow{d_{\an}}$ at 118 97
    \pinlabel $+$ at 270 88
    \pinlabel $+$ at 391 88
    \endlabellist
    \includegraphics[scale = .48]{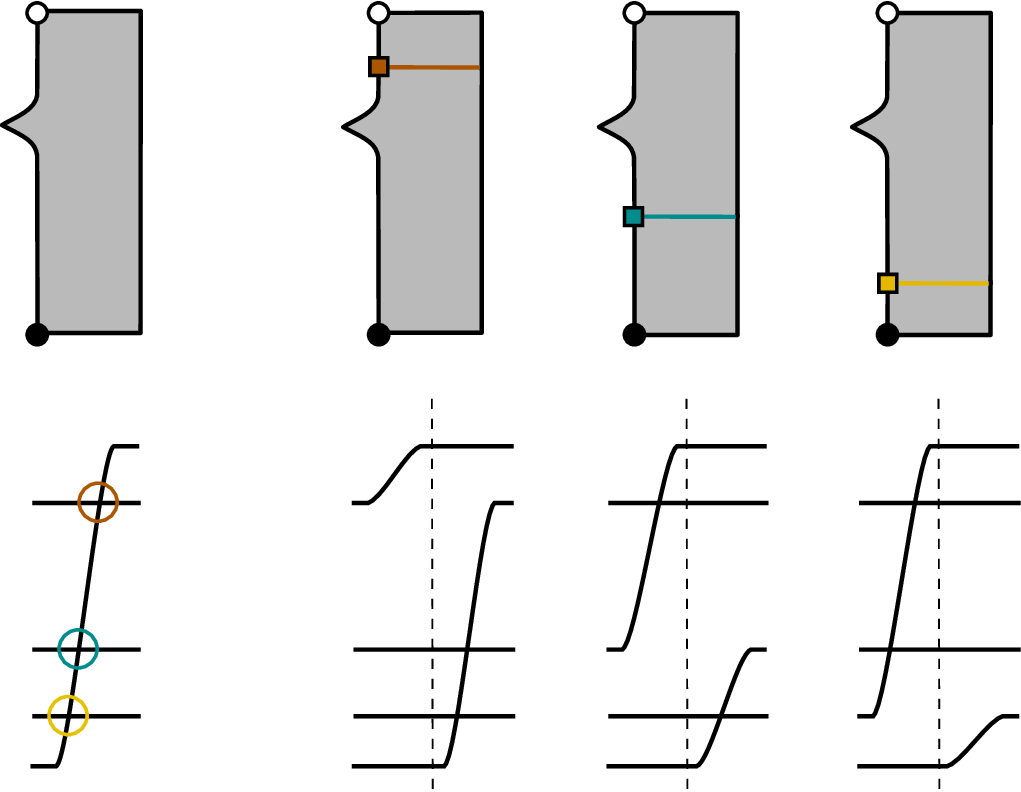} \caption{Left: A pentagon $p'$  from the black dot $\x$ to the white dot  $\z$. The algebra element $\algr{p'}$ is displayed below. Right: Decompositions of $p'$ into $p * r$ or $r * p$, with the corresponding algebra elements $\algr{p} \cdot \algr{r}$ or $\algr{r} \cdot \algr{p}$. Note that $d_{\aun} \algr{p'}$ is equal to the sum of these elements.}
    \label{fig:pent_d_mult}
  \end{figure}

  We have thus proven Equations~\ref{eqn:tri_homo} and \ref{eqn:pent_homo}. Since Equation~\ref{eqn:pent_homo} immediately shows that $\bdy \PP_k = 0$, we see that $\PP_k$ is a type~$\DD$ homomorphism. The left-hand side of \ref{eqn:tri_homo} differs from $\bdy \TT_k$ by two terms involving $d_{\aun}$. However, since all triangles $p$ are interior, $\algl{p}$ and $\algr{p}$ are both idempotents, and so $d_{\aun} \algl{p}$ and $d_{\aun} \algr{p}$ are necessarily zero. Therefore, we see that $\bdy \PP_k = 0$, and $\TT_k$ is also a type~$\DD$ homomorphism.
\end{proof}

In the proofs of the following lemmas, in which we prove that our maps satisfy the remaining conditions of Lemma~\ref{lem:hom_alg}, we will again be considering domains of the form $p * p'$, where $p$ and $p'$ are different polygons, showing that such juxtapositions cancel each other. Often, as in the previous proof, $p$ and $p'$ may be disjoint, or they may have overlapping interiors. In these cases, the domain can also be decomposed as $p' * p$, and so does not contribute a term to the morphisms.

The case where $p$ and $p'$ share an edge and two corners also arises frequently. In this case, $p$ and $p'$ are both bordered polygons. We can always handle such domains as in the proof of the previous lemma, canceling terms as in Figures~\ref{fig:pent_d_tree} and \ref{fig:pent_d_mult}.

Therefore, from now on, we shall omit all domains described in the previous two paragraphs, and only focus on the case where $p$ and $p'$ share exactly one common corner.

Most of the time, as before, a domain that can be written as a juxtaposition $p * p'$ can always be written as exactly one alternative juxtaposition $p'' * p'''$, and the two terms cancel out. However, a new situation arises in the proofs of the following lemmas that was not present in the proof of Lemma~\ref{lem:cond1}. There are now some \emph{special cases}, in which a juxtaposition resulting in one domain may cancel a juxtaposition resulting in a different domain! We shall discuss all special cases and provide figures in each lemma. We shall also provide a table in each lemma that shows all cancelations, including the special cases. For the expert reader, the existence of such special cases is not uncommon in multipointed Heegaard Floer theory.

In each upcoming proof, we will enumerate all relevant juxtapositions of polygons. To facilitate the enumeration, we now set up some notation to help us categorize polygons.

\begin{figure}[h]
  \centering
  \includegraphics[scale=0.44]{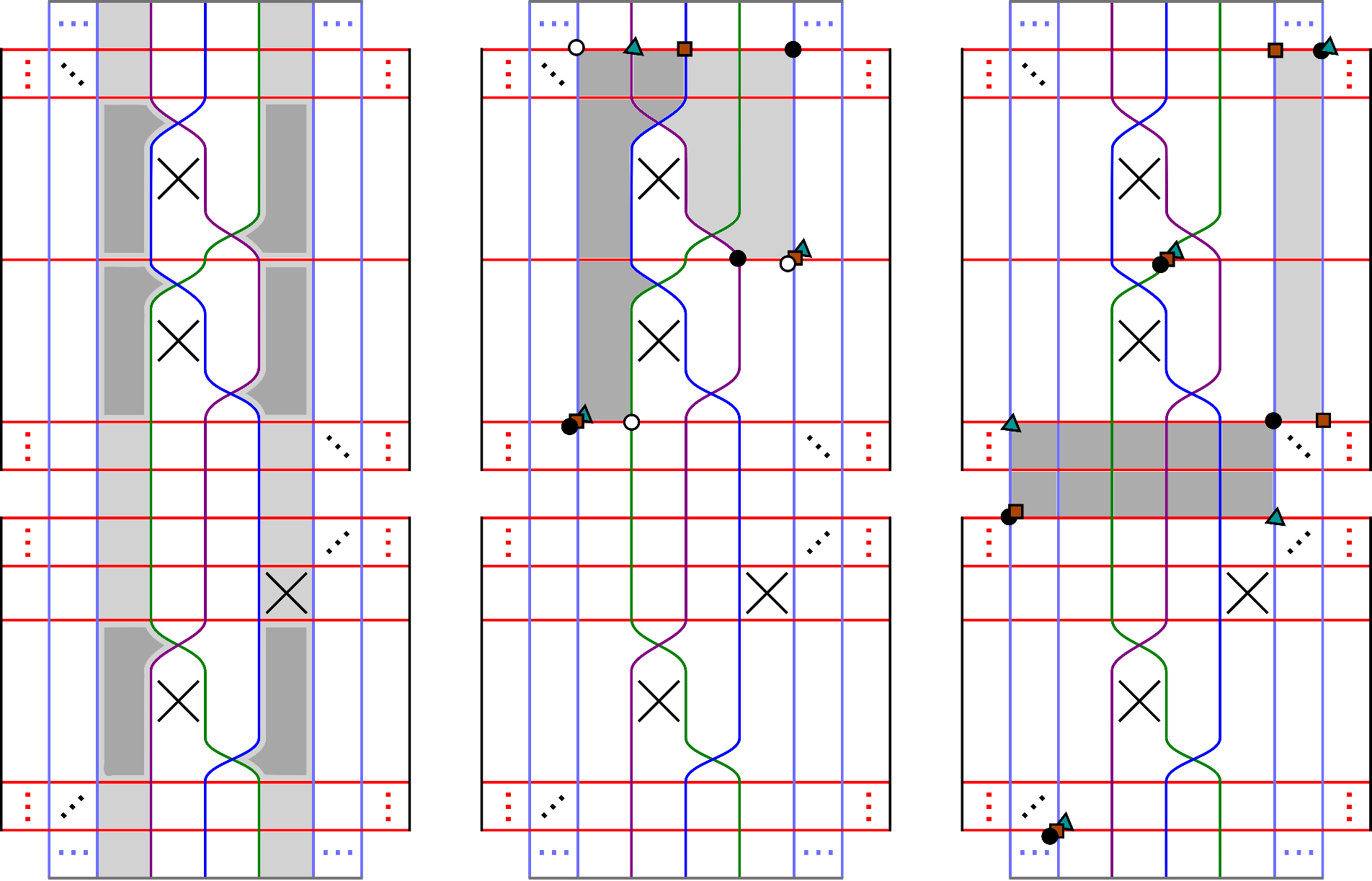}
  \caption{Left: The annuli $\lan$ and $\ran$ are lightly shaded. The $\lan$-height (resp.\ $\ran$-height) of a domain $p$ is the sum of the multiplicities of the dark pentagonal regions on the left (resp.\ right) in the support of $p$. Center: An $\lran$-domain of $\lan$-height $2$ and $\ran$-height $1$. It admits two decompositions as $p * p' \in \PP_\infty \circ \PP_1$ with path $R_1 \htd{0}{1} C_\infty \htd{2}{0} L_0$ (as reflected in the shading), and as $r * p'' \in \HH_1 \circ \delta_1^1$ with path $R_1 \htd{0}{1} L_1 \htd{2}{0} L_0$. Right: The light domain is a rectangle of type~$\nan$, and the dark one is a rectangle of type~$\lran$. Both rectangles have path $C_0 \htdnm C_0$, and their $\lan$- and $\ran$-heights are both $0$.}
  \label{fig:annuli}
\end{figure}

We define $\lan$ (resp.\ $\ran$) to be the unique annulus in $\Sigma \setminus \betas$ whose boundary contains the points $u_k$ (resp.\ $v_k$). See the left of Figure~\ref{fig:annuli}. Observe that $\lan \setminus (\lan \cap \alphas)$ has $2n + 2$ connected components, three of which are pentagonal regions. Given a domain $p$, we define the \emph{$\lan$-height} of $p$ to be the sum of the multiplicities of each of these pentagonal regions in the support of $p$.  Likewise, $\ran \setminus (\ran \cap \alphas)$ has three connected components that are pentagonal regions, and we define the \emph{$\ran$-height} of a domain $p$ to be the sum of the multiplicities of each of these pentagons in the support of $p$.

\begin{figure}[h]
  \centering
  \includegraphics[scale=0.44]{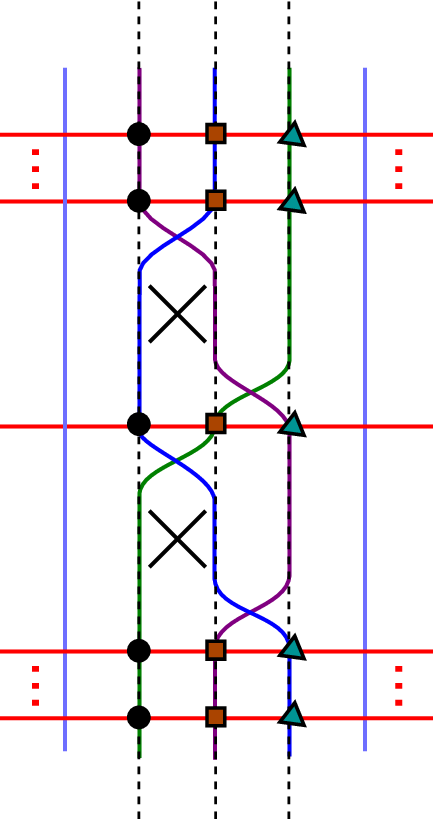}
  \caption{Suppose we have a generator $\x \in \SS (\HD_k)$. If $\x$ has a component on one of the black dots, then $\x \in L_k$. The brown squares correspond to $C_k$, and the teal triangles to $R_k$. Note that $\x$ has exactly one component on $\beta_{i,\infty} \cup \beta_{i,0} \cup \beta_{i,1}$, so these three cases are mutually exclusive.}
  \label{fig:lcr}
\end{figure}

For a fixed $k\in \set{\infty, 0, 1}$, the module $M_k$ splits as a direct sum of modules $M_k = L_k\oplus C_k\oplus R_k$, where $L_k$ is spanned by those generators in $\SS(\HD_k)$ whose $\beta_{i,k}$-component lies on the boundary of the annulus $\lan$, $R_k$ is spanned by the generators whose $\beta_{i,k}$-component lies on the boundary of the annulus $\ran$, and $C_k$ is spanned by the remaining generators. Visually, the points in $\alphas \cap (\beta_{i,\infty} \cup \beta_{i,0} \cup \beta_{i,1})$ lie on three distinct vertical lines, as seen in Figure~\ref{fig:lcr}; then the modules $L_k$, $C_k$, and $R_k$ are spanned by the generators that intersect the left, central, and right line, respectively.

Furthermore, we say that a domain $p$ is of
\begin{enumerate}
  \item \emph{type~$\nan$} if $p \cap (\lan \cup \ran) = \emptyset$;
  \item \emph{type~$\lan$} if $p \cap \lan \neq \emptyset$ and $p \cap \ran = \emptyset$;
  \item \emph{type~$\ran$} if $p \cap \lan = \emptyset$ and $p \cap \ran \neq \emptyset$; and \item \emph{type~$\lran$} if $p \cap \lan \neq \emptyset$ and $p \cap \ran \neq \emptyset$.
\end{enumerate}
We will also use the words \emph{$\lan$-domains}, \emph{$\ran$-domains}, etc.\ to describe domains as appropriate.

If $p$ is a polygon from $\x \in M_k$ to $\y \in M_{k+*}$, we define its \emph{path} to be the following information:
\begin{enumerate}
  \item whether $\x$ belongs to $L_k$, $C_k$, or $R_k$;
  \item whether $\y$ belongs to $L_{k+*}$, $C_{k+*}$, or $R_{k+*}$; and
  \item the $\lan$- and $\ran$-heights of $p$.
\end{enumerate}
For example, in the central figure of Figure~\ref{fig:annuli}, the light pentagon $p$ from $\x \in R_1$ (the black dots) to $\y \in C_\infty$ (the brown squares) has $\lan$-height $0$ and $\ran$-height $1$. We use the notation $R_1 \htd{0}{1} C_\infty$ to denote the path of $p$. Likewise, the dark pentagon $p'$ from $\y \in C_\infty$ to $\z \in L_0$ (the white dots) has $\lan$-height $2$ and $\ran$-height $0$, so its path is $C_\infty \htd{2}{0} L_0$.

We can extend this to composite domains $p * p'$, where $p$ and $p'$ are both polygons. In the example above, the path of the composite domain $p * p'$ from $\x$ to $\z$ is $R_1 \htd{0}{1} C_\infty \htd{2}{0} L_0$. The $\lan$-height of the domain is $2$, and the $\ran$-height is $1$. Note that $\lan$- and $\ran$-heights are additive under composition of domains.

It is also useful to note whether the support of a polygon $p$ has any corners on $\beta_{i,\infty} \cup \beta_{i,0} \cup \beta_{i,1}$ or not. The latter case is possible only if $p$ is a rectangle. See the right figure of Figure~\ref{fig:annuli} for two examples. The support of the light rectangle $r$ from $\x \in C_0$ to $\y \in C_0$ does not have any corners on $\beta_{i,\infty} \cup \beta_{i,0} \cup \beta_{i,1}$.  Note that this implies that the components of $\x$ and $\y$ on $\beta_{i,\infty} \cup \beta_{i,0} \cup \beta_{i,1}$ are exactly the same, i.e. the component of $\x$ on $\beta_{i,\infty} \cup \beta_{i,0} \cup \beta_{i,1}$ is not moved by $p$. In this case, instead of writing $C_0 \htd{0}{0} C_0$ for the path of $r$, we will write $C_0 \htdnm C_0$ (where $\NM$ stands for ``not moving''). No information is lost by doing so, since the $\lan$- and $\ran$-heights of such rectangles are always $0$.

Figure~\ref{fig:annuli} also highlights the following fact. It may seem at first glance that the type of a domain $p$ is a redundant piece of information because it seems to be available from the $\lan$- and $\ran$-heights of $p$; however, this is not true. In the right figure, the light rectangle is of type~$\nan$, while the dark rectangle is of type~$\lran$, even though both rectangles have $\lan$ and $\ran$--heights $0$, because they both have the label $\NM$. In fact, rectangles with the label $\NM$ are the only polygons that can have different types, and they are always $\nan$-domains or $\lran$-domains. This little piece of information eases our enumerations greatly.

We can now give a complete list of all possible paths and types for the defined polygons, simply by inspecting Figure~\ref{fig:hd_cut}. For convenience, we suppress all subscripts.
\begin{itemize}
  \item Rectangles:
    \begin{itemize}
      \item type~$\lan$: $L \htd{0}{0} L, \enspace C \htd{1}{0} L, \enspace C \htd{0}{0} C, \enspace R \htd{0}{0} R$.
      \item type~$\ran$: $L \htd{0}{0} L, \enspace C \htd{0}{2} L, \enspace C \htd{0}{0} C, \enspace C \htd{0}{1} R, \enspace R \htd{0}{1} L, \enspace R \htd{0}{0} R$.
      \item labelled $\NM$: $L \htdnm L$ (type~$\nan$), $\enspace L \htdnm L$ (type~$\lran$), $\enspace C \htdnm C$ (type~$\nan$), $\enspace C \htdnm C$ (type~$\lran$), $\enspace R \htdnm R$ (type~$\nan$), $\enspace R \htdnm R$ (type~$\lran$).
    \end{itemize}
  \item Triangles:
    \begin{itemize}
      \item type~$\nan$: $C \htd{0}{0} R, \enspace R \htd{0}{0} L$.
    \end{itemize}
  \item Pentagons:
    \begin{itemize}
      \item type~$\lan$: $L \htd{1}{0} L, \enspace C \htd{2}{0} L$.
      \item type~$\ran$: $C \htd{0}{2} C, \enspace C \htd{0}{3} R, \enspace R \htd{0}{1} C, \enspace R \htd{0}{2} R, \enspace R \htd{0}{3} L$.
    \end{itemize}
  \item Quadrilaterals:
    \begin{itemize}
      \item type~$\nan$: $R \htd{0}{0} C$.
    \end{itemize}
  \item Hexagons:
    \begin{itemize}
      \item type~$\lan$: $L \htd{2}{0} L, \enspace C \htd{3}{0} L$.
      \item type~$\ran$: $R \htd{0}{3} C$.
    \end{itemize}
  \item Heptagons:
    \begin{itemize}
      \item type~$\lan$: $L \htd{3}{0} L$.
    \end{itemize}
    \label{list:poly}
\end{itemize}

Using this list, we will obtain a complete list of all possible paths and types for the relevant juxtapositions $p * p'$ in each lemma. In general, given the path and type of a given domain $p * p'$, we can almost always recover the general shape of the support of $p * p'$ on Figure~\ref{fig:hd_cut}, which will help us determine the canceling juxtaposition. The only ambiguity arises when $p * p'$ is of type~$\lan$ or type~$\ran$ (and not type~$\lran$), and $p$ and $p'$ are both rectangle-like polygons; we eliminate this ambiguity in the following paragraph.

Fix a domain $p * p'$ of type~$\lan$ or type~$\ran$ from $\x$ to $\z$, where $p$ and $p'$ are rectangle-like polygons that share exactly one common corner. Suppose that the support of $p * p'$ is not an annulus. This implies that the support of $p * p'$ has exactly one distinguished corner whose internal angle is reflex (i.e.\ larger than $\pi$). We say that $p * p'$ has \emph{shape~A} if $\x$ does not contain the distinguished corner as one of its components, and \emph{shape~B} otherwise. In the former case, $\z$ must contain the distinguished corner as a component.  If instead the support of $p * p'$ is an annulus, we say that $p * p'$ has \emph{shape~C}. Since there is a marker on the annulus $\ran$, this is possible only if $p * p'$ is of type~$\lan$, with $\lan$-height $3$; also, neither $p$ nor $p'$ can have label $\NM$. See Figure~\ref{fig:abc} for examples of each of the three shapes.

We remark here the notion of shapes is unnecessary for $\lran$-domains $p * p'$.

\begin{figure}
  \centering
  \includegraphics[scale=0.46]{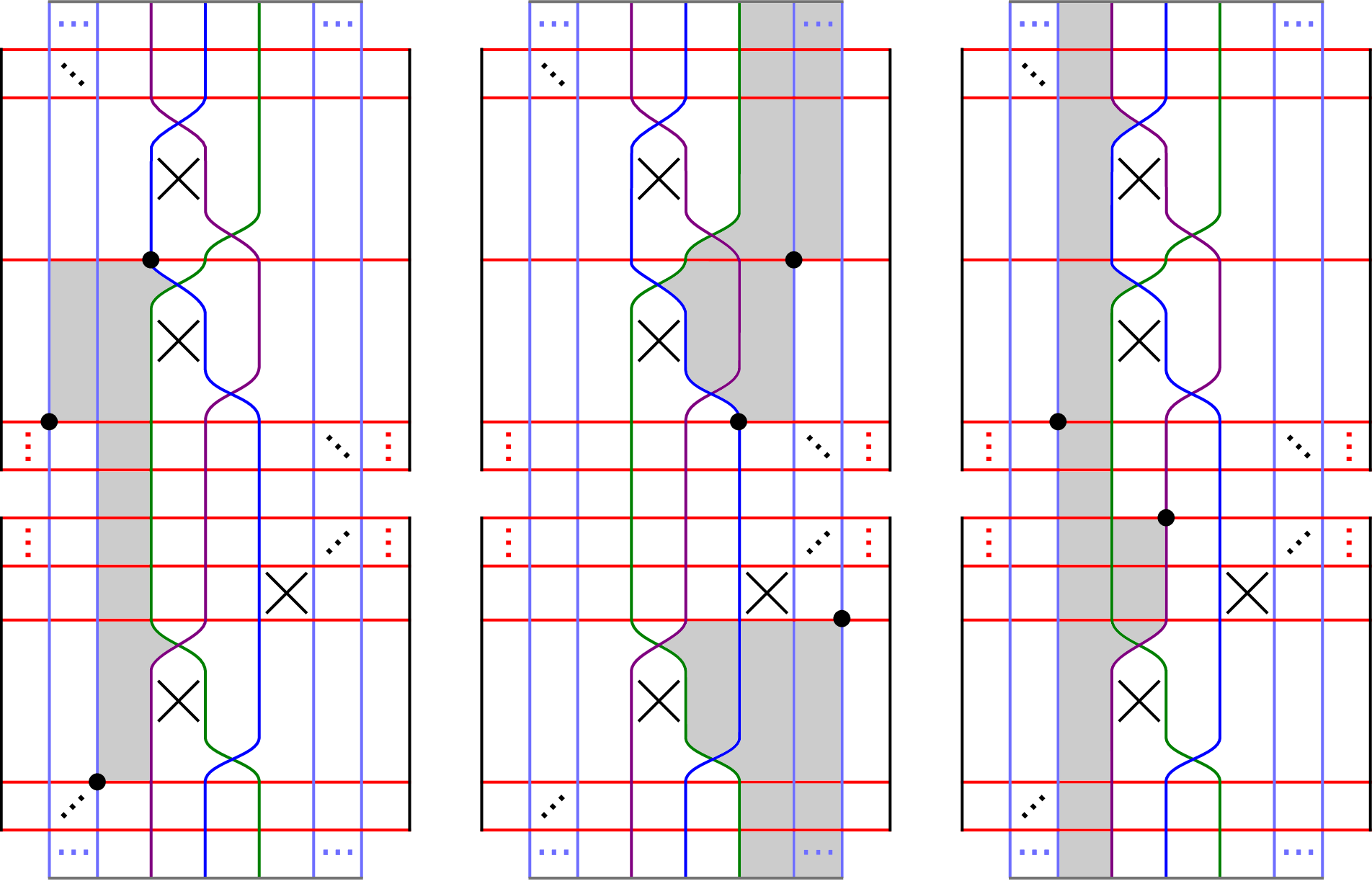}
  \caption{Left: A domain of shape~A. The black dots represent the relevant components of the initial generator $\x$. Center: A domain of shape~B. Right: A domain of shape~C. Note that, in each case, the domain is either of type~$\lan$ or of type~$\ran$, and is a juxtaposition of two rectangle-like polygons.}
  \label{fig:abc}
\end{figure}

\begin{lem}
  \label{lem:cond2}
  The composition $f_{k+1} \circ f_k$ is homotopic to zero via homotopy $\phi_k$. In other words, the morphisms $f_k$ and $\phi_k$ satisfy Condition~(2) of Lemma~\ref{lem:hom_alg}.
\end{lem}

\begin{proof}  

  The strategy of the proof is to enumerate all possible paths, types, and shapes for a domain appearing in $f_{k+1} \circ f_k + \delta_{k+2}^1 \circ \phi_k + \phi_k \circ \delta_k^1$, and find a cancelation for each case. We do so in Table~\ref{tab:cond2}, displayed at the end of this proof. The domains $p * p'$ in Table~\ref{tab:cond2} are ordered by
  \begin{enumerate}
    \item the polygons that $p$ and $p'$ belong to (Column~2);
    \item the path of $p$, following the order in the list of paths on p.~\pageref{list:poly} (Column~3);
    \item the path of $p'$ (Column~3); and
    \item the shape of $p * p'$ (Column~5), if applicable.
  \end{enumerate}

  As mentioned before, there are special cases, in which a juxtaposition resulting in one domain cancels a juxtaposition resulting in a different domain. Table~\ref{tab:cond2} includes all special cases. The special cases are further illustrated in Figures~\ref{fig:v1_tt}, \ref{fig:tp_pt}, \ref{fig:tp_pt_2}, \ref{fig:qr_pt}, and \ref{fig:qr_pt_2}, and discussed in more detail below, along with references to the corresponding rows of Table~\ref{tab:cond2}.

  Let us look at Figure~\ref{fig:v1_tt}, for example. If $\x\in \SS(\HD_1)$ occupies $\alpha_t\cap \beta_{i,1}$, for $\alpha_t\in \{\alpha_0^L, \ldots, \alpha_i^L\}\cup\{\alpha_0^R, \ldots, \alpha_{i-1}^R\}$, then $\TT_{\infty}\circ\TT_1(\x)$ cancels out with a term in $\HH_1\circ\delta^1_1(\x)$, $\PP_{\infty}\circ\PP_1(\x)$, or $\delta^1_0\circ \HH_1(\x)$, depending on the position of the $\beta_{i+1}$-component $x_{i+1}$ of $\x$, as seen on Figure~\ref{fig:hd_cut}. If $x_{i+1}$ is below $\alpha_t$ or above $\alpha_i^R$, the canceling term is in $\HH_1\circ\delta^1_1(\x)$; if $x_{i+1}$ is on $\alpha_i^R$, the canceling term is in $\PP_{\infty}\circ\PP_1(\x)$; if $x_{i+1}$ is below $\alpha_i^R$ and above $\alpha_t$, the canceling term is in $\delta^1_0\circ \HH_1(\x)$.  Analogous special cases occur when $\x\in \SS(\HD_0)$ occupies $\alpha_i^R\cap \beta_{i,0}$, and when $\x\in \SS(\HD_{\infty})$ occupies $\alpha_t\cap \beta_{i,\infty}$ for $\alpha_t\in \{\alpha_{i+1}^L, \ldots, \alpha_n^L\}\cup\{\alpha_{i+1}^R, \ldots, \alpha_n^R\}$. See Table~\ref{tab:cond2}, Rows~$14$, $40$, $60$, and $63$.

  \begin{figure}
    \centering
    \includegraphics[scale = .44]{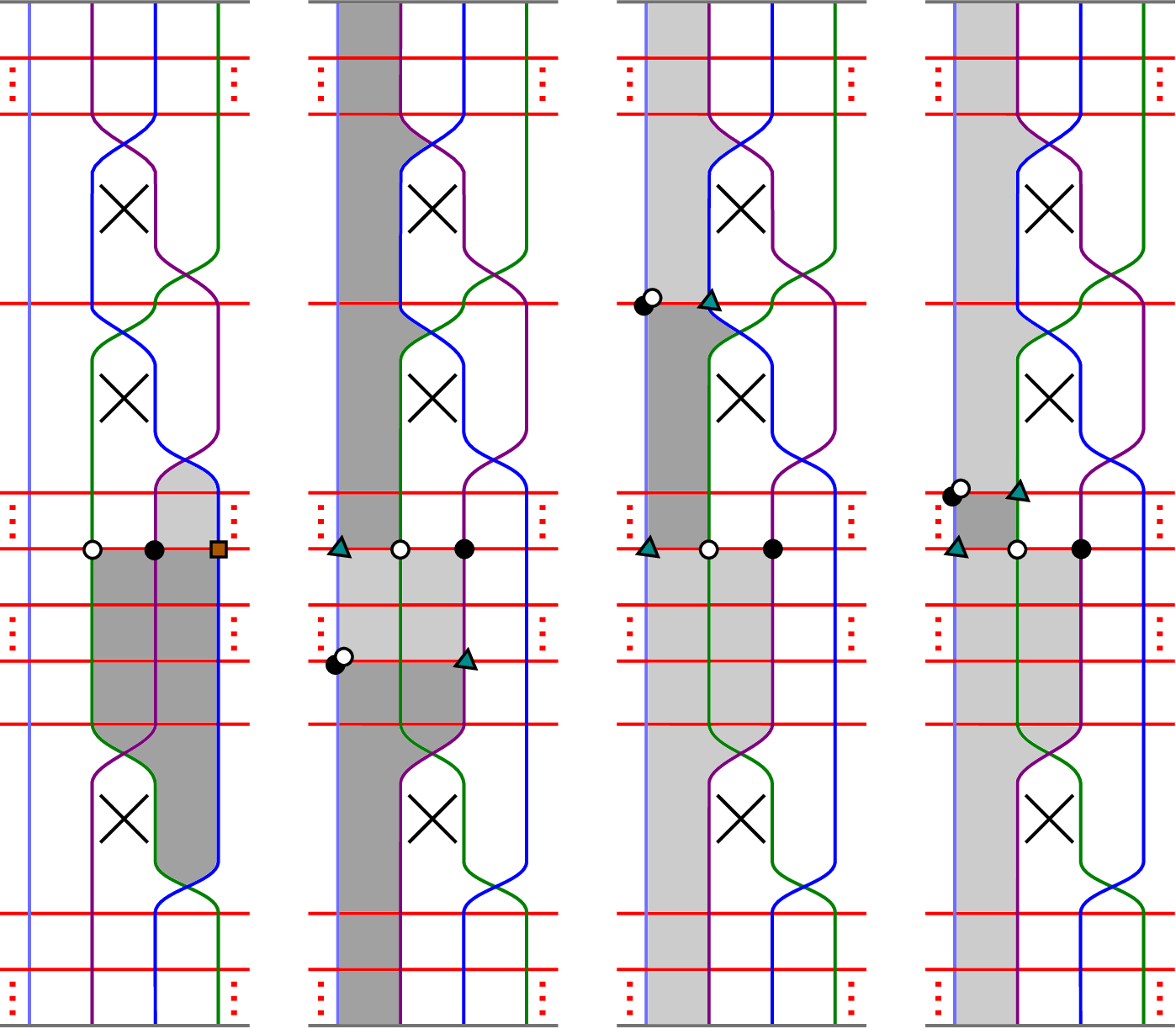}
    \caption{A special case of a term in  $\TT\circ \TT$ canceling out a term in $\HH\circ \delta^1$, $\PP\circ \PP$, or $\delta^1\circ \HH$.}
    \label{fig:v1_tt}
  \end{figure}

  The remaining special cases go as follows. 

  Figure~\ref{fig:tp_pt} illustrates a term in $\PP_{\infty}\circ \TT_1$ canceling out a term in $\TT_{\infty}\circ \PP_1$. Depending on the starting generator, some terms may output non-trivial algebra elements, corresponding to a domain juxtaposition that is bordered.  In an analogous special case, a term in $\PP_1\circ \TT_0$ cancels out a term in $\TT_1\circ \PP_0$. See Table~\ref{tab:cond2}, Rows~$3$ and $6$.

  \begin{figure}
    \centering
    \includegraphics[scale = .44]{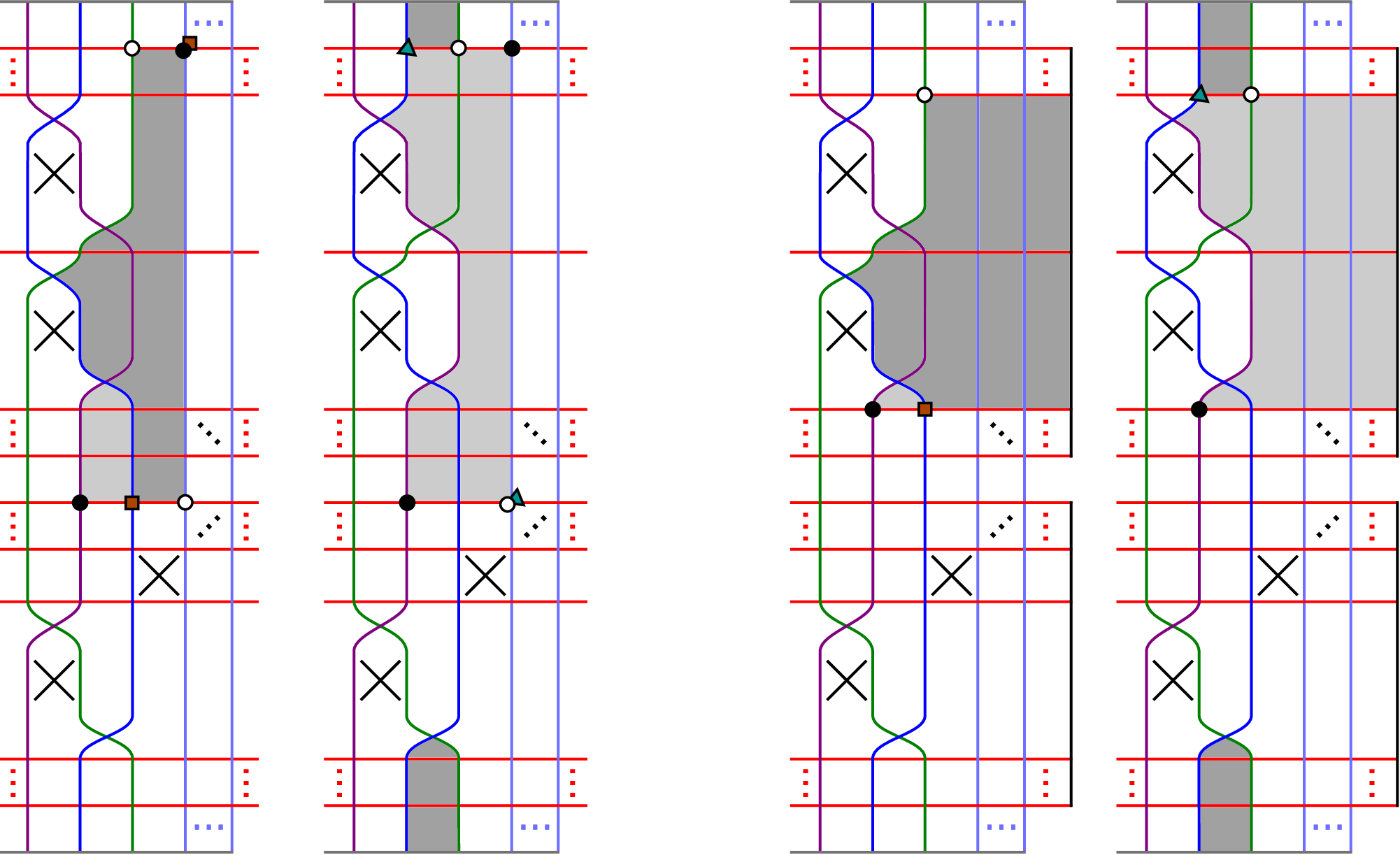}
    \caption{A special case of a term in $\PP\circ \TT$ canceling out a term in $\TT\circ \PP$.}
    \label{fig:tp_pt}
  \end{figure}

  Figure~\ref{fig:tp_pt_2} illustrates another term in $\PP_{\infty}\circ \TT_1$ canceling out a term in $\TT_{\infty}\circ \PP_1$, with paths different from those of the terms in Figure~\ref{fig:tp_pt}. See Table~\ref{tab:cond2}, Rows~$4$ and $7$.

  \begin{figure}
    \centering
    \includegraphics[scale = .44]{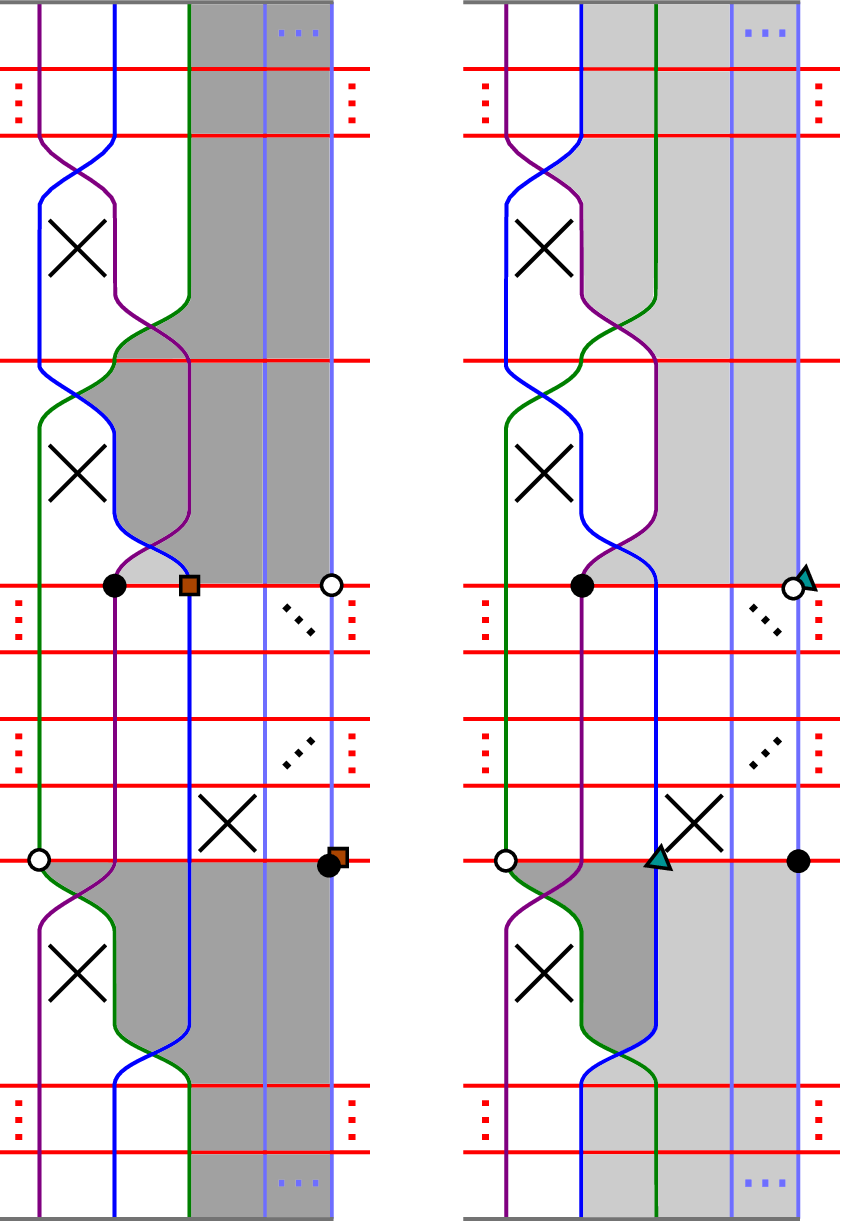}
    \caption{Another special case of a term in $\PP\circ \TT$ canceling out a term in $\TT\circ \PP$.}
    \label{fig:tp_pt_2}
  \end{figure}

  Figure~\ref{fig:qr_pt} illustrates a term in $\TT_0\circ \PP_{\infty}$ canceling out a term in $\delta^1_1\circ \Q_{\infty}$. Depending on the starting generator, some terms may output non-trivial algebra elements, corresponding to a domain juxtaposition that is bordered. See Table~\ref{tab:cond2}, Rows~$8$ and $31$.

  \begin{figure}
    \centering
    \includegraphics[scale = .44]{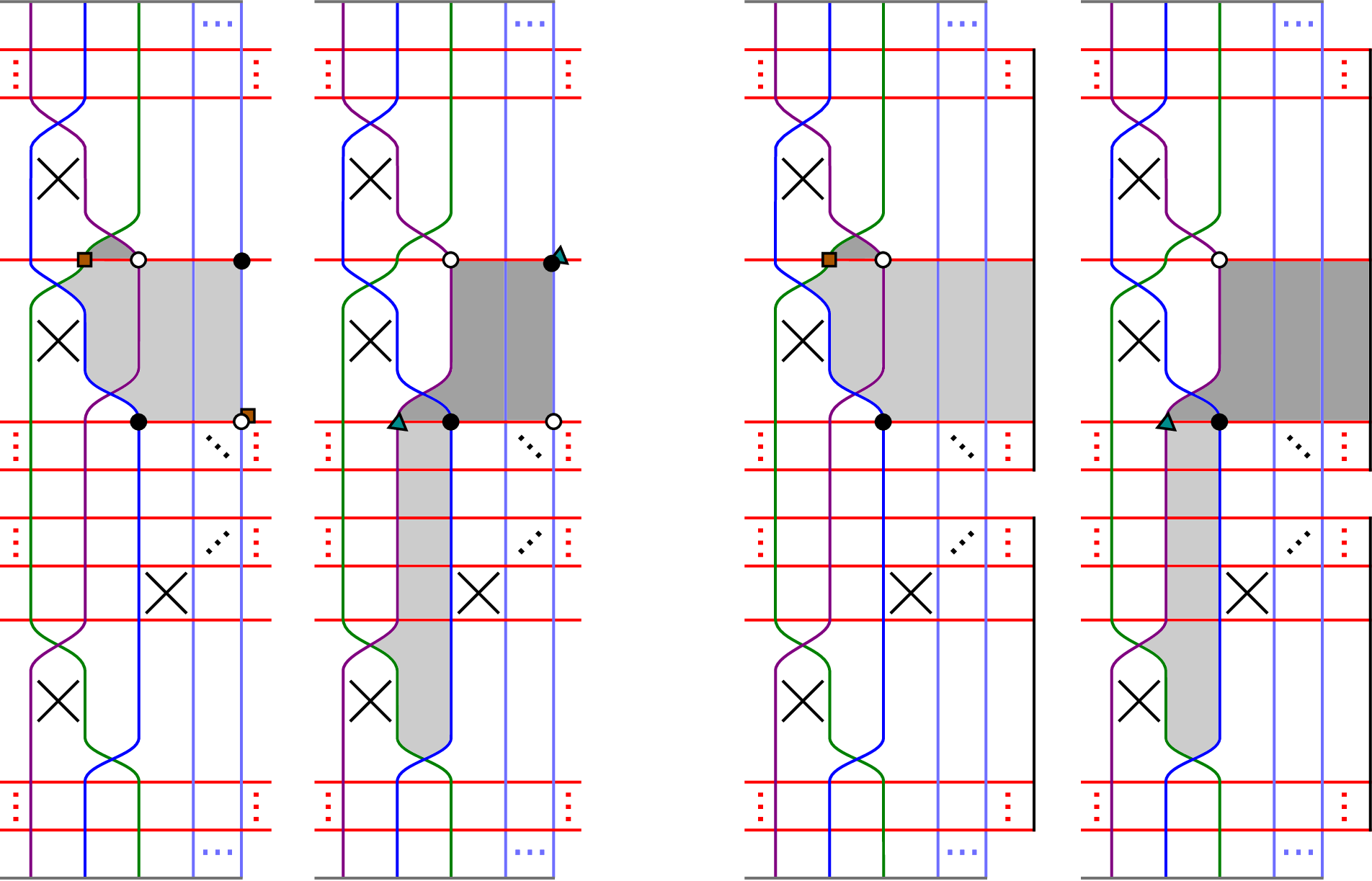}
    \caption{A special case of a term in $\TT\circ \PP$ canceling out a term in $\delta^1\circ \Q$.}
    \label{fig:qr_pt}
  \end{figure}

  Figure~\ref{fig:qr_pt_2} illustrates another term in $\TT_0\circ \PP_{\infty}$ canceling out a term in $\delta^1_1\circ \Q_{\infty}$, with paths different from those of the terms in Figure~\ref{fig:qr_pt}. Depending on the starting generator, some terms may output non-trivial algebra elements, corresponding to a domain juxtaposition that is bordered.

  \begin{figure}
    \centering
    \includegraphics[scale = .44]{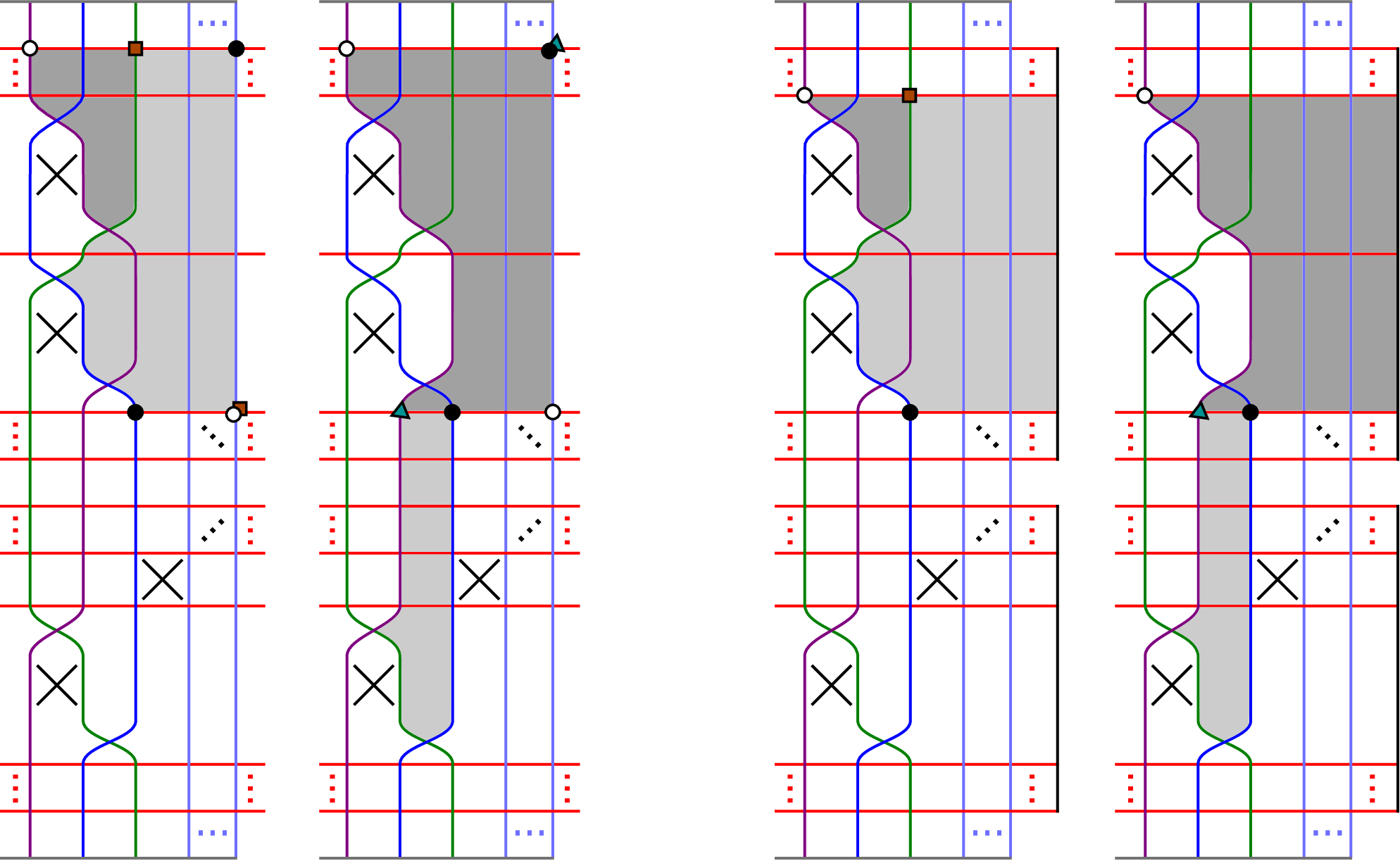}
    \caption{Another special case of a term in $\delta^1\circ \Q$ canceling out a term in $\TT\circ \PP$.}
    \label{fig:qr_pt_2}
  \end{figure}

  Table~\ref{tab:cond2} summarizes the cancelations of domains in this lemma. The reader is reminded that it only shows composite domains $p * p'$ such that $p$ and $p'$ share exactly one common corner.

  \clearpage

  \begin{ThreePartTable}
    \begin{TableNotes}
      \small
    \item[(i)] Special case.
    \item[(ii)] Depends on the position of certain components of $\x$.
    \item[(iii)] Does not exist because the $\ran$-height is larger than $3$.
    \item[(iv)] Does not exist because any such domain necessarily contains markers.
    \item[(v)] Figure~\ref{fig:v1_tt}.
    \item[(vi)] Figure~\ref{fig:tp_pt}.
    \item[(vii)] Figure~\ref{fig:tp_pt_2}.
    \item[(viii)] Figure~\ref{fig:qr_pt}.
    \item[(ix)] Figure~\ref{fig:qr_pt_2}.
    \end{TableNotes}
    \begin{longtable}{|c|c|c|c|c||c|c|}
      \captionsetup{belowskip=10pt}
      \hline
      No. & Term in & Path & Type & Shape & Cancels with & Notes\\
      \hline \hline
      \endfirsthead
      \hline
      No. & Term in & Path & Type & Shape & Cancels with & Notes\\
      \hline \hline
      \endhead
      \hline
      \insertTableNotes
      \endfoot
      \hline
      \insertTableNotes\\
      \caption{The cancelations of the terms in Lemma~\ref{lem:cond2}. The special cases are shown in Figures~\ref{fig:v1_tt}--\ref{fig:qr_pt_2}. For convenience, we suppress all subscripts.}
      \label{tab:cond2}
      \endlastfoot

      1 & $\TT \circ \TT$ & $C \htd{0}{0} R \htd{0}{0} L$ & $\nan$ & & 14, 40, 60, or 63 & (i), (ii), (v)\\
      \hline
      2 & \mra{4}{22}{$\PP \circ \TT$} & $C \htd{0}{0} R \htd{0}{1} C$ & $\ran$ & & 54 &\\
      \cline{1-1} \cline{3-7}
      3 & & $C \htd{0}{0} R \htd{0}{2} R$ & $\ran$ & & 6 & (i), (vi)\\
      \cline{1-1} \cline{3-7}
      4 & & $C \htd{0}{0} R \htd{0}{3} L$ & $\ran$ & & 7 & (i), (vii)\\
      \cline{1-1} \cline{3-7}
      5 & & $R \htd{0}{0} L \htd{1}{0} L$ & $\lan$ & & 27 &\\
      \hline
      6 & \mra{4}{22}{$\TT \circ \PP$} & $C \htd{0}{2} C \htd{0}{0} R$ & $\ran$ & & 3 & (i), (vi)\\
      \cline{1-1} \cline{3-7}
      7 & & $C \htd{0}{3} R \htd{0}{0} L$ & $\ran$ & & 4 & (i), (vii)\\
      \cline{1-1} \cline{3-7}
      8 & & $R \htd{0}{1} C \htd{0}{0} R$ & $\ran$ & & 31 & (i), (viii)\\
      \cline{1-1} \cline{3-7}
      9 & & $R \htd{0}{2} R \htd{0}{0} L$ & $\ran$ & & 29 & (i), (ix)\\
      \hline
      10 & \mra{17}{60}{$\PP \circ \PP$} & \mr{2}{$L \htd{1}{0} L \htd{1}{0} L$} & \mra{2}{13}{$\lan$} & A & 35 &\\
      \cline{1-1} \cline{5-7}
      11 & & & & B & 73 &\\
      \cline{1-1} \cline{3-7}
      12 & & \mr{3}{$C \htd{2}{0} L \htd{1}{0} L$} & \mra{3}{13}{$\lan$} & A & 42 &\\ \cline{1-1} \cline{5-7}
      13 & & & & B & 76 &\\
      \cline{1-1} \cline{5-7}
      14 & & & & C & 1 & (i), (v)\\
      \cline{1-1} \cline{3-7}
      15 & & $C \htd{0}{2} C \htd{2}{0} L$ & $\lran$ & & 66 &\\
      \cline{1-1} \cline{3-7}
      16 & & $C \htd{0}{2} C \htd{0}{2} C$ & $\ran$ & & None & (iii)\\
      \cline{1-1} \cline{3-7}
      17 & & $C \htd{0}{2} C \htd{0}{3} R$ & $\ran$ & & None & (iii)\\
      \cline{1-1} \cline{3-7}
      18 & & $R \htd{0}{1} C \htd{2}{0} L$ & $\lran$ & & 69 & \\
      \cline{1-1} \cline{3-7}
      19 & & \mr{2}{$R \htd{0}{1} C \htd{0}{2} C$} & \mra{2}{13}{$\ran$} & A & 50 & \\
      \cline{1-1} \cline{5-7}
      20 & & & & B & 79 &\\
      \cline{1-1} \cline{3-7}
      21 & & $R \htd{0}{1} C \htd{0}{3} R$ & $\ran$ & & None & (iii)\\
      \cline{1-1} \cline{3-7}
      22 & & \mr{2}{$R \htd{0}{2} R \htd{0}{1} C$} & \mra{2}{13}{$\ran$} & A & 50 &\\
      \cline{1-1} \cline{5-7}
      23 & & & & B & 79 &\\
      \cline{1-1} \cline{3-7}
      24 & & $R \htd{0}{2} R \htd{0}{2} R$ & $\ran$ & & None & (iii)\\
      \cline{1-1} \cline{3-7}
      25 & & $R \htd{0}{2} R \htd{0}{3} L$ & $\ran$ & & None & (iii)\\
      \cline{1-1} \cline{3-7}
      26 & & $R \htd{0}{3} L \htd{1}{0} L$ & $\lran$ & & 45 & \\
      \hline
      27 & \mra{5}{30}{$\delta^1 \circ \Q$} & $R \htd{0}{0} C \htd{1}{0} L$ & $\lan$ & & 5 &\\
      \cline{1-1} \cline{3-7}
      28 & & $R \htd{0}{0} C \htd{0}{0} C$ & $\lan$ & & 53 &\\
      \cline{1-1} \cline{3-7}
      29 & & $R \htd{0}{0} C \htd{0}{2} L$ & $\ran$ & & 9 & (i), (ix)\\
      \cline{1-1} \cline{3-7}
      30 & & $R \htd{0}{0} C \htd{0}{0} C$ & $\ran$ & & 55 &\\
      \cline{1-1} \cline{3-7}
      31 & & $R \htd{0}{0} C \htd{0}{1} R$ & $\ran$ & & 8 & (i), (viii)\\
      \hline
      32 & \mra{3}{20}{$\delta^1 \circ \HH$} & \mr{2}{$L \htd{2}{0} L \htd{0}{0} L$} & \mra{2}{13}{$\lan$} & A & 35 &\\
      \cline{1-1} \cline{5-7}
      33 & & & & B & 73 &\\
      \cline{1-1} \cline{3-7}
      34 & & $L \htd{2}{0} L \htd{0}{0} L$ & $\lran$ & & 74 &\\
      \cline{1-1} \cline{3-7}
      35 & \mra{18}{70}{$\delta^1 \circ \HH$} & \mr{2}{$L \htd{2}{0} L \htdnm L$} & \mra{2}{13}{$\lan$} & A & 10, 32, 56, or 72 & (ii)\\
      \cline{1-1} \cline{5-7}
      36 & & & & B & 73 &\\
      \cline{1-1} \cline{3-7}
      37 & & $L \htd{2}{0} L \htdnm L$ & $\lran$ & & 65 &\\
      \cline{1-1} \cline{3-7}
      38 & & \mr{3}{$C \htd{3}{0} L \htd{0}{0} L$} & \mra{3}{13}{$\lan$} & A & 42 &\\
      \cline{1-1} \cline{5-7}
      39 & & & & B & 76 &\\
      \cline{1-1} \cline{5-7} 40 & & & & C & 1 & (i), (v)\\
      \cline{1-1} \cline{3-7}
      41 & & $C \htd{3}{0} L \htd{0}{0} L$ & $\lran$ & & 77 &\\
      \cline{1-1} \cline{3-7}
      42 & & \mr{2}{$C \htd{3}{0} L \htdnm L$} & \mra{2}{13}{$\lan$} & A & 12, 38, 58, 61, or 75 & (ii)\\
      \cline{1-1} \cline{5-7}
      43 & & & & B & 76 &\\
      \cline{1-1} \cline{3-7}
      44 & & $C \htd{3}{0} L \htdnm L$ & $\lran$ & & 67 &\\
      \cline{1-1} \cline{3-7}
      45 & & $R \htd{0}{3} C \htd{1}{0} L$ & $\lran$ & & 26 &\\
      \cline{1-1} \cline{3-7}
      46 & & $R \htd{0}{3} C \htd{0}{0} C$ & $\lran$ & & None & (iv)\\
      \cline{1-1} \cline{3-7}
      47 & & $R \htd{0}{3} C \htd{0}{2} L$ & $\ran$ & & None & (iii)\\
      \cline{1-1} \cline{3-7}
      48 & & $R \htd{0}{3} C \htd{0}{0} C$ & $\ran$ & & None & (iv)\\
      \cline{1-1} \cline{3-7}
      49 & & $R \htd{0}{3} C \htd{0}{1} R$ & $\ran$ & & None & (iii)\\
      \cline{1-1} \cline{3-7}
      50 & & \mr{2}{$R \htd{0}{3} C \htdnm C$} & \mra{2}{13}{$\ran$} & A & 19, 22, 70, or 78 & (ii)\\
      \cline{1-1} \cline{5-7}
      51 & & & & B & 79 &\\
      \cline{1-1} \cline{3-7}
      52 & & $R \htd{0}{3} C \htdnm C$ & $\lran$ & & 64 &\\
      \hline
      53 & \mra{3}{20}{$\Q \circ \delta^1$} & $R \htd{0}{0} R \htd{0}{0} C$ & $\lan$ & & 28 &\\
      \cline{1-1} \cline{3-7}
      54 & & $C \htd{0}{1} R \htd{0}{0} C$ & $\ran$ & & 2 &\\
      \cline{1-1} \cline{3-7}
      55 & & $R \htd{0}{0} R \htd{0}{0} C$ & $\ran$ & & 30 &\\
      \hline
      56 & \mra{15}{50}{$\HH \circ \delta^1$} & \mr{2}{$L \htd{0}{0} L \htd{2}{0} L$} & \mra{2}{13}{$\lan$} & A & 35 &\\
      \cline{1-1} \cline{5-7}
      57 & & & & B & 73 &\\
      \cline{1-1} \cline{3-7}
      58 & & \mr{3}{$C \htd{1}{0} L \htd{2}{0} L$} & \mra{3}{13}{$\lan$} & A & 42 & \\
      \cline{1-1} \cline{5-7}
      59 & & & & B & 76 &\\
      \cline{1-1} \cline{5-7}
      60 & & & & C & 1 & (i), (v)\\
      \cline{1-1} \cline{3-7}
      61 & & \mr{3}{$C \htd{0}{0} C \htd{3}{0} L$} & \mra{3}{13}{$\lan$} & A & 42 & \\
      \cline{1-1} \cline{5-7}
      62 & & & & B & 76 &\\
      \cline{1-1} \cline{5-7}
      63 & & & & C & 1 & (i), (v)\\
      \cline{1-1} \cline{3-7}
      64 & & $R \htd{0}{0} R \htd{0}{3} C$ & $\lran$ & & 52 &\\
      \cline{1-1} \cline{3-7}
      65 & & $L \htd{0}{0} L \htd{2}{0} L$ & $\lran$ & & 37 &\\
      \cline{1-1} \cline{3-7}
      66 & & $C \htd{0}{2} L \htd{2}{0} L$ & $\lran$ & & 15 &\\
      \cline{1-1} \cline{3-7}
      67 & & $C \htd{0}{0} C \htd{3}{0} L$ & $\lran$ & & 44 &\\
      \cline{1-1} \cline{3-7}
      68 & & $C \htd{0}{1} R \htd{0}{3} C$ & $\ran$ & & None & (iii)\\
      \cline{1-1} \cline{3-7}
      69 & & $R \htd{0}{1} L \htd{2}{0} L$ & $\lran$ & & 18 &\\
      \cline{1-1} \cline{3-7}
      70 & & $R \htd{0}{0} R \htd{0}{3} C$ & $\ran$ & A & 50 &\\
      \cline{1-1} \cline{5-7}
      71 & \mra{10}{30}{$\HH \circ \delta^1$} & $R \htd{0}{0} R \htd{0}{3} C$ & $\ran$ & B & 79 &\\
      \cline{1-1} \cline{3-7}
      72 & & \mr{2}{$L \htdnm L \htd{2}{0} L$} & \mra{2}{13}{$\lan$} & A & 35 &\\
      \cline{1-1} \cline{5-7}
      73 & & & & B & 11, 33, 36, or 57 & (ii)\\
      \cline{1-1} \cline{3-7}
      74 & & $L \htdnm L \htd{2}{0} L$ & $\lran$ & & 34 &\\
      \cline{1-1} \cline{3-7}
      75 & & \mr{2}{$C \htdnm C \htd{3}{0} L$} & \mra{2}{13}{$\lan$} & A & 42 &\\
      \cline{1-1} \cline{5-7}
      76 & & & & B & 13, 39, 43, 59, or 62 & (ii)\\
      \cline{1-1} \cline{3-7}
      77 & & $C \htdnm C \htd{3}{0} L$ & $\lran$ & & 41 &\\
      \cline{1-1} \cline{3-7}
      78 & & \mr{2}{$R \htdnm R \htd{0}{3} C$} & \mra{2}{13}{$\ran$} & A & 50 &\\
      \cline{1-1} \cline{5-7}
      79 & & & & B & 20, 23, 51, or 71 & (ii)\\
      \cline{1-1} \cline{3-7}
      80 & & $R \htdnm R \htd{0}{3} C$ & $\lran$ & & None & (iv)\\
    \end{longtable}
  \end{ThreePartTable}
  This concludes the proof of the lemma.
\end{proof}

\begin{lem}
  \label{lem:cond3}
  The morphism $f_{k+2} \circ \phi_k + \phi_{k+1} \circ f_k$ is homotopic to the identity $\Id_k$ via homotopy $\psi_k$. In other words, the morphisms $f_k, \phi_k$ and $\psi_k$ satisfy Condition~(3) of Lemma~\ref{lem:hom_alg}.
\end{lem}

\begin{proof}
  The strategy is as before. Table~\ref{tab:cond3}, displayed at the end of this proof, shows all cancelations relevant to this lemma, including the special cases. The special cases are further illustrated in Figures~\ref{fig:pq_qp}, \ref{fig:pq_th_2}, \ref{fig:pq_th_1},  and \ref{fig:id}, and discussed in more detail below, along with references to the corresponding rows of Table~\ref{tab:cond3}.

  Figure~\ref{fig:pq_qp} illustrates a term in $\Q_0\circ \PP_{\infty}$ canceling out a term in $\PP_1\circ \Q_{\infty}$. Depending on the starting generator, some terms may output non-trivial algebra elements, corresponding to a domain juxtaposition that is bordered.  In an analogous special case, a term in $\Q_{\infty}\circ \PP_1$ cancels a term in $\PP_0\circ \Q_1$ (here only interior domains are possible).  See Table~\ref{tab:cond3}, Rows~$3$ and $17$.

  Figure~\ref{fig:pq_th_2} illustrates a term in $\TT_1\circ \HH_{\infty}$ canceling out a term in $\PP_1\circ \Q_{\infty}$. See Table~\ref{tab:cond3}, Rows~$4$ and $5$.

  Figure~\ref{fig:pq_th_1} illustrates a term in $\HH_{\infty}\circ \TT_1$ canceling out a term in $\Q_{\infty}\circ \PP_1$. See Table~\ref{tab:cond3}, Rows~$14$ and $16$.

  Figure~\ref{fig:id} illustrates the terms that $\Id_0$ cancels with. If $\x \in C_0$, then $\Id_0$ cancels with $\Q_1 \circ \TT_0$; if $\x \in R_0$, then $\Id_0$ cancels with $\TT_\infty \circ \Q_0$; if $\x \in L_0$, then $\Id_0$ cancels with $\delta_0^1 \circ \K_0$, $\K_0 \circ \delta_0^1$, $\PP_\infty \circ \HH_0$, or $\HH_1 \circ \PP_0$, depending on the position of the $\beta_{i+1}$-component of $\x$. There is an analogous special case for $\Id_1$ and $\Id_{\infty}$. See Table~\ref{tab:cond3}, Rows~$1$, $8$, $13$, $20$, $29$, $36$, $44$, $45$, and $46$.

  \begin{figure}[h]
    \centering
    \includegraphics[scale = .45]{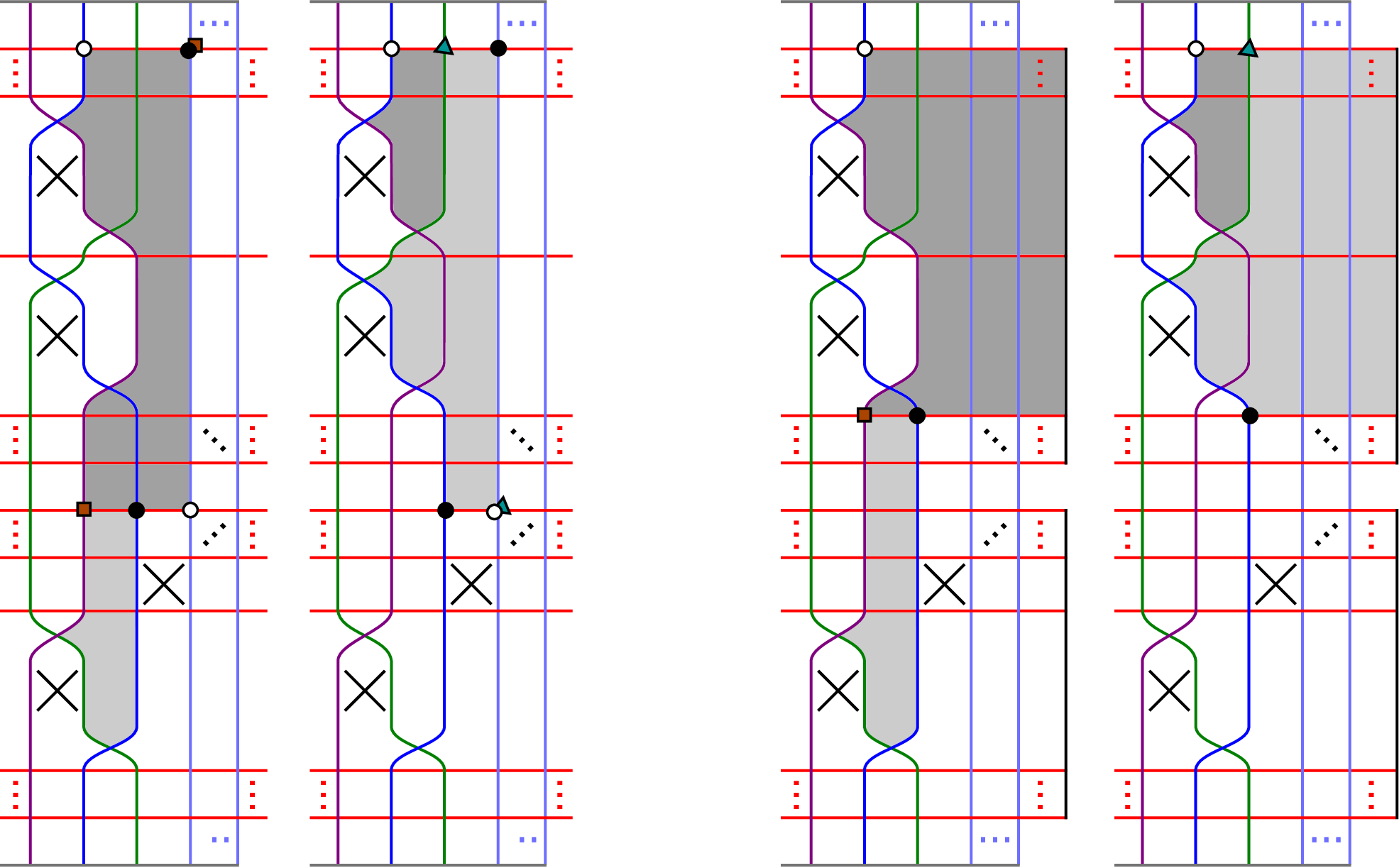}
    \caption{A special case of a term in $\Q\circ \PP$ canceling out a term in $\PP\circ \Q$.}
    \label{fig:pq_qp}
  \end{figure}

  \begin{figure}[h]
    \centering
    \includegraphics[scale = .45]{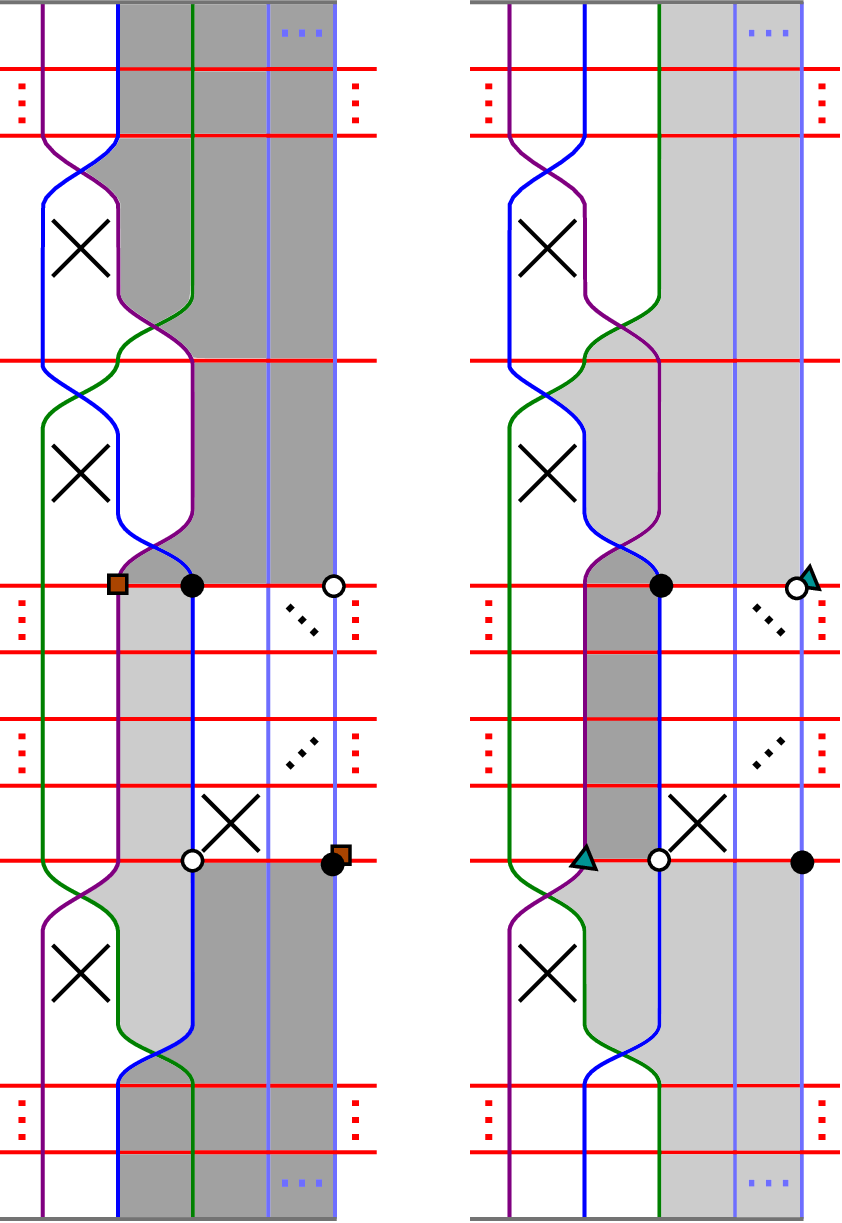}
    \caption{A special case of a term in $\TT\circ \HH$ canceling out a term in $\PP\circ \Q$.}
    \label{fig:pq_th_2}
  \end{figure}

  \begin{figure}[h]
    \centering
    \includegraphics[scale = .45]{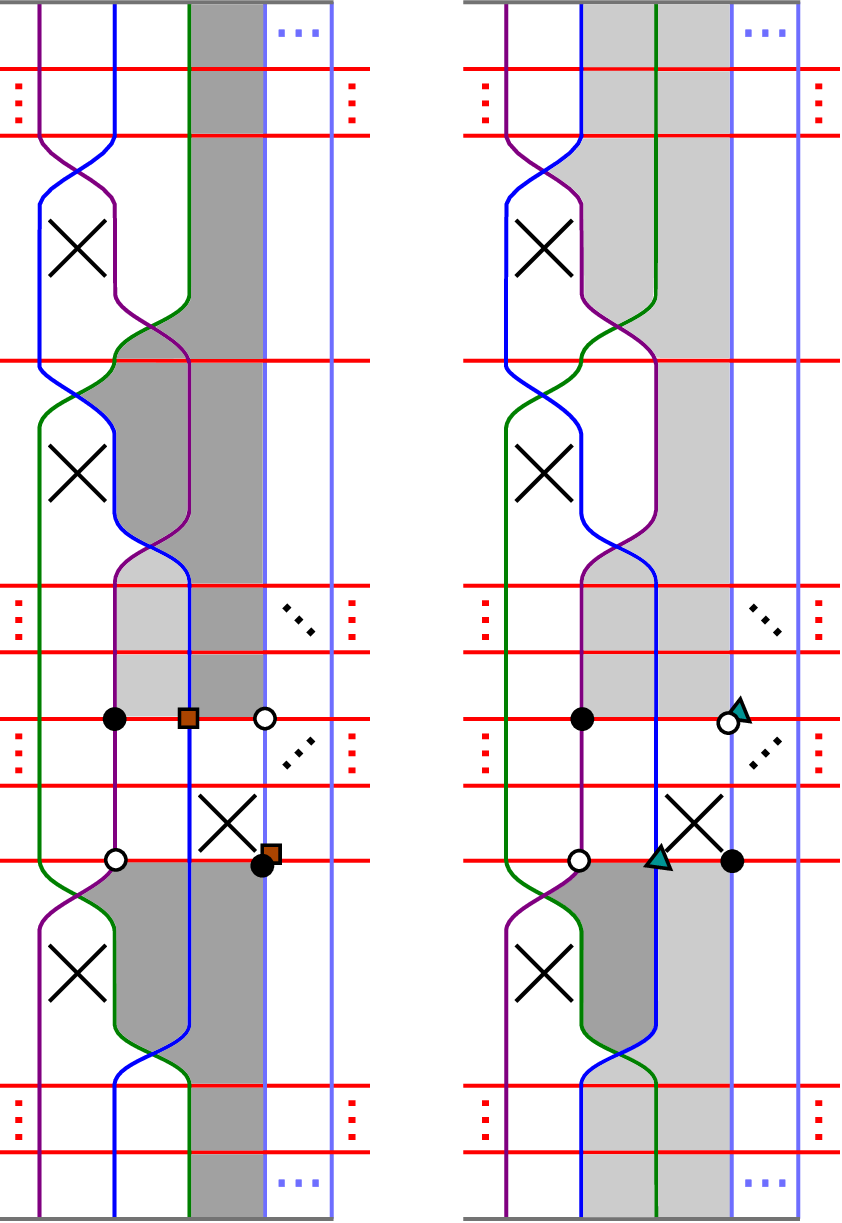}
    \caption{A special case of a term in $\HH\circ \TT$ canceling out a term in $\Q\circ \PP$.}
    \label{fig:pq_th_1}
  \end{figure}

  \begin{figure}[h]
    \centering
    \includegraphics[scale = .43]{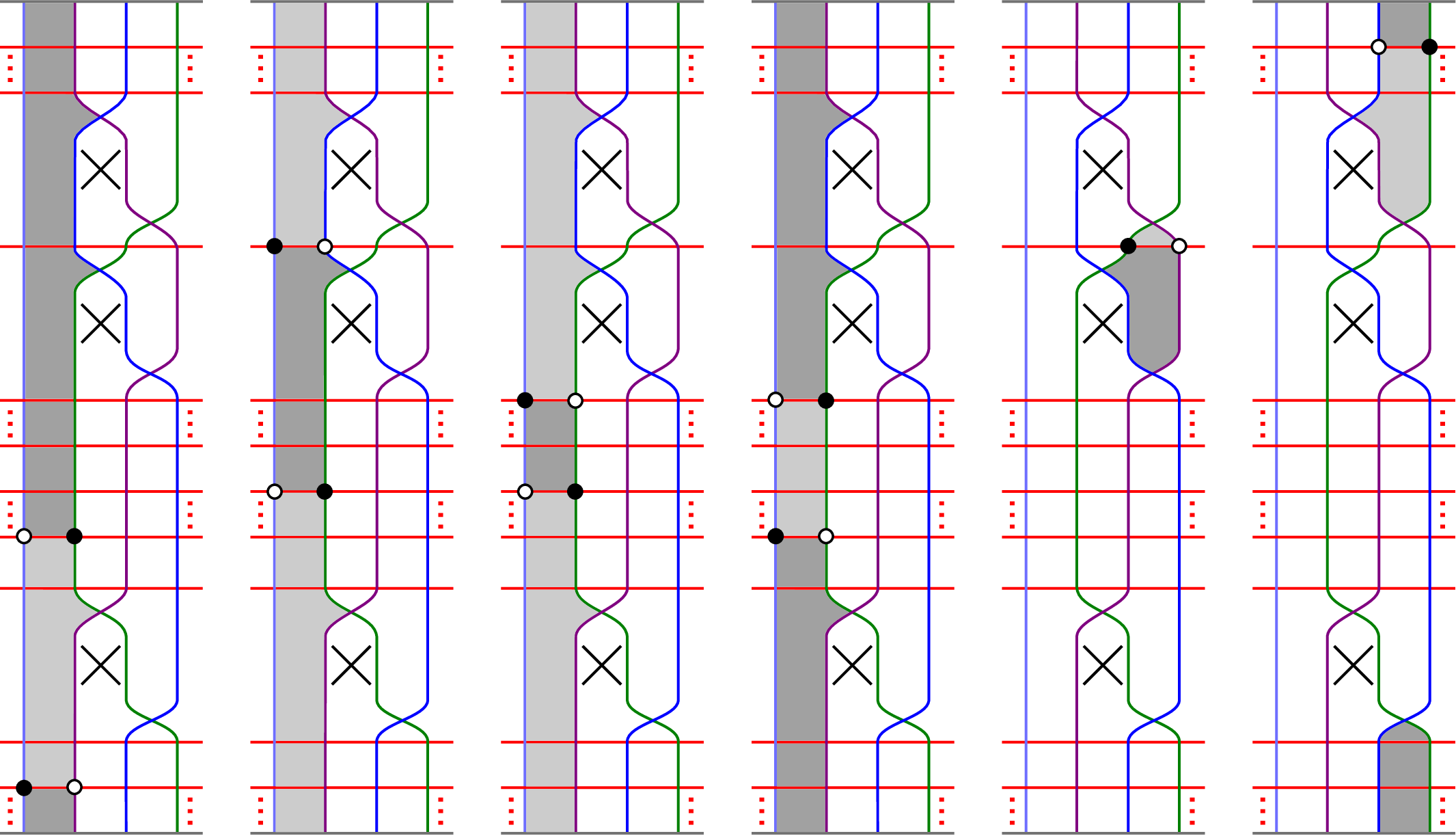}
    \vskip -.1 cm
    \caption{A special case of a term in $\Id$ canceling out a term in $\delta^1\circ \K$, $\delta^1\circ\K$, $\PP\circ \HH$, $\HH\circ \PP$, $\TT\circ \Q$, or $\Q\circ \TT$.}
    \label{fig:id}
  \end{figure}

  \clearpage

  \begin{ThreePartTable}
    \begin{TableNotes}
      \small
    \item[(i)] Special case.
    \item[(ii)] Depends on the position of certain components of $\x$.
    \item[(iii)] Does not exist because the $\lan$-height is larger than $3$.
    \item[(iv)] Does not exist because the $\ran$-height is larger than $3$.
    \item[(v)] Figure~\ref{fig:pq_qp}.
    \item[(vi)] Figure~\ref{fig:pq_th_2}.
    \item[(vii)] Figure~\ref{fig:pq_th_1}.
    \item[(viii)] Figure~\ref{fig:id}.
    \end{TableNotes}
    \begin{longtable}{|c|c|c|c|c||c|c|}
      \captionsetup{belowskip=10pt}
      \hline
      No. & Term in & Path & Type & Shape & Cancels with & Notes\\
      \hline \hline
      \endfirsthead
      \hline
      No. & Term in & Path & Type & Shape & Cancels with & Notes\\
      \hline \hline
      \endhead
      \hline
      \insertTableNotes
      \endfoot
      \hline
      \insertTableNotes\\
      \caption{The cancelations of the terms in Lemma~\ref{lem:cond3}. The special cases are shown in Figures~\ref{fig:pq_qp}--\ref{fig:id}. For convenience, we suppress all subscripts.} \label{tab:cond3}
      \endlastfoot

      1 & $\TT \circ \Q$ & $R \htd{0}{0} C \htd{0}{0} R$ & $\nan$ & & 46 & (i), (viii)\\
      \hline 2 & \mra{3}{20}{$\PP \circ \Q$} & $ R \htd{0}{0} C \htd{2}{0} L$ & $\lan$ & & 15 &\\
      \cline{1-1} \cline{3-7}
      3 & & $R \htd{0}{0} C \htd{0}{2} C$ & $\ran$ & & 17 & (i), (v)\\
      \cline{1-1} \cline{3-7}
      4 & & $R \htd{0}{0} C \htd{0}{3} R$ & $\ran$ & & 5 & (i), (vi)\\
      \hline
      5 & $\TT \circ \HH$ & $R \htd{0}{3} C \htd{0}{0} R$ & $\ran$ & & 4 & (i), (vi)\\
      \hline
      6 & \mra{7}{35}{$\PP \circ \HH$} & \mr{3}{$L \htd{2}{0} L \htd{1}{0} L$} & \mra{3}{13}{$\lan$} & A & 31 &\\
      \cline{1-1} \cline{5-7}
      7 & & & & B & 42 &\\
      \cline{1-1} \cline{5-7}
      8 & & & & C & 44 & (i), (viii)\\
      \cline{1-1} \cline{3-7}
      9 & & $C \htd{3}{0} L \htd{1}{0} L$ & $\lan$ & & None & (iii)\\
      \cline{1-1} \cline{3-7}
      10 & & $R \htd{0}{3} C \htd{2}{0} L$ & $\lran$ & & 26 &\\
      \cline{1-1} \cline{3-7}
      11 & & $R \htd{0}{3} C \htd{0}{2} C$ & $\ran$ & & None & (iv)\\
      \cline{1-1} \cline{3-7}
      12 & & $R \htd{0}{3} C \htd{0}{3} R$ & $\ran$ & & None & (iv)\\
      \hline
      13 & $\Q \circ \TT$ & $C \htd{0}{0} R \htd{0}{0} C$ & $\nan$ & & 45 & (i), (viii)\\
      \hline
      14 & \mra{2}{13}{$\HH \circ \TT$} & $C \htd{0}{0} R \htd{0}{3} C$ & $\ran$ & & 16 & (i), (vii)\\
      \cline{1-1} \cline{3-7}
      15 & & $R \htd{0}{0} L \htd{2}{0} L$ & $\lan$ & & 2 &\\
      \hline
      16 & \mra{2}{13}{$\Q \circ \PP$} & $C \htd{0}{3} R \htd{0}{0} C$ & $\ran$ & & 14 & (i) (vii)\\
      \cline{1-1} \cline{3-7}
      17 & & $R \htd{0}{2} R \htd{0}{0} C$ & $\ran$ & & 3 & (i), (v)\\
      \hline
      18 & \mra{9}{45}{$\HH \circ \PP$} & \mr{3}{$L \htd{1}{0} L \htd{2}{0} L$} & \mra{3}{13}{$\lan$} & A & 31 &\\
      \cline{1-1} \cline{5-7}
      19 & & & & B & 42 &\\
      \cline{1-1} \cline{5-7}
      20 & & & & C & 44 & (i), (viii)\\
      \cline{1-1} \cline{3-7}
      21 & & $C \htd{2}{0} L \htd{2}{0} L$ & $\lan$ & & None & (iii)\\
      \cline{1-1} \cline{3-7}
      22 & & $C \htd{0}{2} C \htd{3}{0} L$ & $\lran$ & & 39 &\\
      \cline{1-1} \cline{3-7}
      23 & & $C \htd{0}{3} R \htd{0}{3} C$ & $\ran$ & & None & (iv)\\
      \cline{1-1} \cline{3-7}
      24 & & $R \htd{0}{1} C \htd{3}{0} L$ & $\lran$ & & 40 &\\
      \cline{1-1} \cline{3-7}
      25 & & $R \htd{0}{2} R \htd{0}{3} C$ & $\ran$ & & None & (iv)\\
      \cline{1-1} \cline{3-7}
      26 & & $R \htd{0}{3} L \htd{2}{0} L$ & $\lran$ & & 10 &\\
      \hline
      27 & \mra{7}{22}{$\delta^1 \circ \K$} & \mr{3}{$L \htd{3}{0} L \htd{0}{0} L$} & \mra{3}{13}{$\lan$} & A & 31 &\\
      \cline{1-1} \cline{5-7}
      28 & & & & B & 42 &\\
      \cline{1-1} \cline{5-7}
      29 & & & & C & 44 & (i), (viii)\\
      \cline{1-1} \cline{3-7}
      30 & & $L \htd{3}{0} L \htd{0}{0} L$ & $\lran$ & & 43 &\\
      \cline{1-1} \cline{3-7}
      31 & & \mr{2}{$L \htd{3}{0} L \htdnm L$} & \mra{2}{13}{$\lan$} & A & 6, 18, 27, 34, or 41 & (ii)\\
      \cline{1-1} \cline{5-7}
      32 & & & & B & 42 &\\
      \cline{1-1} \cline{3-7}
      33 & & $L \htd{3}{0} L \htdnm L$ & $\lran$ & & 38 &\\
      \hline
      34 & $\K \circ \delta^1$ & $L \htd{0}{0} L \htd{3}{0} L$ & $\lan$ & A & 31 &\\
      \cline{1-1} \cline{5-7}
      35 & \mra{9}{45}{$\K \circ \delta^1$} & \mra{2}{10}{$L \htd{0}{0} L \htd{3}{0} L$} & \mra{2}{10}{$\lan$} & B & 42 &\\
      \cline{1-1} \cline{5-7}
      36 & & & & C & 44 & (i), (viii)\\
      \cline{1-1} \cline{3-7}
      37 & & $C \htd{1}{0} L \htd{3}{0} L$ & $\lan$ & & None & (iii)\\
      \cline{1-1} \cline{3-7}
      38 & & $L \htd{0}{0} L \htd{3}{0} L$ & $\lran$ & & 33 &\\
      \cline{1-1} \cline{3-7}
      39 & & $C \htd{0}{2} L \htd{3}{0} L$ & $\lran$ & & 22 &\\
      \cline{1-1} \cline{3-7}
      40 & & $R \htd{0}{1} L \htd{3}{0} L$ & $\lran$ & & 24 &\\
      \cline{1-1} \cline{3-7}
      41 & & \mr{2}{$L \htdnm L \htd{3}{0} L$} & \mra{2}{13}{$\lan$} & A & 31 &\\
      \cline{1-1} \cline{5-7}
      42 & & & & B & 7, 19, 28, 32, or 35 & (ii)\\
      \cline{1-1} \cline{3-7}
      43 & & $L \htdnm L \htd{3}{0} L$ & $\lran$ & & 30 &\\
      \hline
      44 & \mra{3}{18}{$\Id$} & $L \htdid L$ & & & 8, 20, 29, or 36 & (i), (ii), (viii)\\
      \cline{1-1} \cline{3-7}
      45 & & $C \htdid C$ & & & 13 & (i), (viii)\\
      \cline{1-1} \cline{3-7}
      46 & & $R \htdid R$ & & & 1 & (i), (viii)
    \end{longtable}
  \end{ThreePartTable}

  This concludes the proof of the lemma.
\end{proof}

\begin{proof}[Proof of Proposition~\ref{prop:main}.]
  This is now a straightforward application of Lemma~\ref{lem:hom_alg}. The conditions in Lemma~\ref{lem:hom_alg} are satisfied according to Lemmas~\ref{lem:cond1}, \ref{lem:cond2} and \ref{lem:cond3}.
\end{proof}

\begin{proof}[Proof of Theorem~\ref{thm:ourtheorem}.]
  We wish to prove that there exists a type~$\DD$ homomorphism
  \[
    F_0 \colon \CDTDu (\T_0, n) \to \CDTDu (\T_1, n)
  \]
  such that
  \[
    \CDTDu (\T_\infty, n) \simeq \Cone (F_0).
  \]
  Write $\T_\infty, T_0, T_1$ as
  \[
    \T_\infty = \T' \circ \eT_\infty \circ \T'', \quad \T_0 = \T' \circ \eT_0 \circ \T'', \quad \T_1 = \T' \circ \eT_1 \circ \T'',
  \]
  where $\eT_\infty, \eT_0, \eT_1$ are the elementary tangles in Proposition~\ref{prop:main}, and $\T'$ and $\T''$ are two other tangles. Then, for each $k \in \set{\infty, 0, 1}$, we have
  \[
    \CDTDu (\T_k, n) \simeq \CDTAu (\T', n_1) \boxtimes \CDTDu (\eT_k, n_2) \boxtimes \CATDu (\T'', n_3)
  \]
  as type~$\DD$ structures, where $n = n_1 + n_2 + n_3$. For convenience, let
  \[
    \DDm{M'} = \CDTAu (\T', n_1), \quad \DDm{M''} = \CATDu (\T'', n_3).
  \]
  Working in the homotopy category, we can now define the morphisms $F_k, \Phi_k, \Psi_k$:
  \begin{enumerate}
    \item The morphism $F_k \colon \CDTDu (\T_k, n) \to \CDTDu (\T_{k+1}, n)$ is defined by
      \[
        F_k = \Id_{\DDm{M'}} \boxtimes f_k \boxtimes \Id_{\DDm{M''}}.
      \]
    \item The morphism $\Phi_k \colon \CDTDu (\T_k, n) \to \CDTDu (\T_{k+2}, n)$ is defined by
      \[
        \Phi_k = \Id_{\DDm{M'}} \boxtimes \phi_k \boxtimes \Id_{\DDm{M''}}.
      \]
    \item The morphism $\Psi_k \colon \CDTDu (\T_k, n) \to \CDTDu (\T_k, n)$ is defined by
      \[
        \Psi_k = \Id_{\DDm{M'}} \boxtimes \psi_k \boxtimes \Id_{\DDm{M''}}.
      \]
  \end{enumerate}
  (Here, $\Id_{\DDm{M'}}$ is a type~$\DA$ isomorphism while $\Id_{\DDm{M''}}$ is a type~$\AD$ isomorphism.) By \cite[Lemma~2.3.13]{bimod}, homotopies are preserved under box-tensoring, and so $F_k, \Phi_k, \Psi_k$ also satisfy the conditions in Lemma~\ref{lem:hom_alg}.

  Taking the box tensor on either side with the left-right type~$\AA$ bimodule for a tangle consisting of only straight strands, we obtain the analogous statements for type~$\DA$, $\AD$, and $\AA$ bimodules.
\end{proof}

\begin{rmk}
  In the proof of Theorem~\ref{thm:ourtheorem}, we could have directly used Proposition~\ref{prop:main} and made no reference to Lemma~\ref{lem:hom_alg}. We chose to present the proof as it is now to streamline the discussion in Section~\ref{sec:grading}.
\end{rmk}

%%%%%%%%%%%%%%%%%%%%%%%%%%%%%%%%%%%%%%%%%%%%%%%%%%%%%%%

%%%%%%%%%%%%%%%%%%%%%%%%%%%%%%%%%%%%%%%%%%%%%%%%%%%%%%%
% !TEX root = ../skein.tex
%%%%%%%%%%%%%%%%%%%%%%%%%%%%%%%%%%%%%%%%%%%%%%%%%%%%%%%

\section{Gradings} % (fold)
\label{sec:grading}

%%%%%%%%%%%%%%%%%%%%%%%%%%%%%%%%%%%%%%%%%%%%%%%%%%%%%%%

%%%%%%%%%%%%%%%%%%%%%%%%%%%%%%%%%%%%%%%%%%%%%%%%%%%%%%%
% section grading 
%%%%%%%%%%%%%%%%%%%%%%%%%%%%%%%%%%%%%%%%%%%%%%%%%%%%%%%

In this section, we prove Theorem~\ref{thm:gradings}.
Let $\T_\infty, \T_0$, and $\T_1$ be three oriented tangles such that, as unoriented tangles, they are identical except near a point $p$, as indicated in Figure~\ref{fig:T_inf01}.  The only difference between Theorems~\ref{thm:ourtheorem} and \ref{thm:gradings} is that Theorem~\ref{thm:gradings} contains information about the grading shifts (cf.\ Table~\ref{tab:all_flavors}). As such, Theorem~\ref{thm:gradings} follows from Theorem~\ref{thm:ourtheorem} and grading information. To begin, we modify the definitions of $F_k$, $\Phi_k$ and $\Psi_k$ from the proof of Theorem~\ref{thm:ourtheorem}, as follows.

Write $\T_\infty, \T_0, \T_1$ as
\[
  \T_\infty = \T_\infty' \circ \eT_\infty \circ \T_\infty'', \quad \T_0 = \T_0' \circ \eT_0 \circ \T_0'', \quad \T_1 = \T_1' \circ \eT_1 \circ \T_1'',
\]
where, as before, $\eT_\infty, \eT_0, \eT_1$ are the elementary tangles in Proposition~\ref{prop:main} (now endowed with  orientations). This time, $\T_\infty', \T_0', \T_1'$ are possibly different as oriented tangles, but are the same as unoriented tangles. The same statement is true for $\T_\infty'', \T_0'', \T_1''$. 

Fix oriented planar diagrams $D_\infty, D_0$, and $D_1$ for $\T_\infty, \T_0$, and $\T_1$ respectively, such that the diagrams, without their orientations, are identical except near $p$. If $\Neg (D)$ denotes the number of negative crossings in a diagram $D$, let
\[
  e_k = \Neg (D_{k+1}) - \Neg (D_k)
\]
for each $k \in \set{\infty, 0, 1}$. Note that $e_\infty, e_0$ and $e_1$ are independent of the choice of diagrams, and that
\[
  e_\infty + e_0 + e_1 = 0.
\]

Similarly, for each $k \in \set{\infty, 0, 1}$, fix oriented planar diagrams $D_k'$ and $D_k''$ for $\T_k'$ and $\T_k''$ respectively, and let $e_k'$ and $e_k''$ be defined analogously. Define isomorphisms
\begin{align*}
  \iota_k' & \colon \CDTAd (\T_k', n) \to \CDTAd (\T_{k+1}', n) \sqbrac{- \frac{e_k'}{2}}\\
  \iota_k'' & \colon \CATDd (\T_k'', n) \to \CATDd (\T_{k+1}'', n) \sqbrac{- \frac{e_k''}{2}}
\end{align*}
as follows. Let $\HD_k'$ and $\HD_{k+1}'$ be Heegaard diagrams for $\T_k'$ and $\T_{k+1}'$ that are identical if we do not distinguish the $X$ from the $O$ markings. Define $\iota_k'$ to be the map induced by the natural correspondence between generators and domains. Define $\iota_k''$ analogously. Indeed, these isomorphisms shift gradings as claimed:

\begin{lem}\label{lem:iota_gr}
  Let $\T_{o_1}$ and $\T_{o_2}$ be two oriented tangles that are the same after forgetting the orientation. Then
  \[
    \CTd(\T_{o_1}, n)\simeq \CTd(\T_{o_2}, n)\sqbrac{-\frac{e}{2}},
  \]
  where $e =   \Neg (D_{o_2}) - \Neg (D_{o_1})$ is the difference in the number of negative crossings between two diagrams for the two tangles that are the same without the orientation.  Here, $\CTd$ can stand for any one of the fully blocked, $\delta$-graded algebraic structures associated to the two tangles (same one for both tangles).
\end{lem}

\begin{proof}
  Since concatenating corresponds to tensoring, and the grading is additive 
  under taking tensor product, it suffices to show this for elementary tangles. 

  Let $\HD_{o_1}$ and $\HD_{o_2}$ be genus-one Heegaard diagrams for elementary tangles $\T_{o_1}$ and $\T_{o_2}$ respectively, such as the ones in Figures~\ref{fig:hd_elem_ori} and \ref{fig:hd_012},  that are identical if we do not distinguish the $X$ from the $O$ markings, i.e., for $i=1,2$ we can write $\HD_{o_i} =(\Sigma, \alphas, \betas, \XX_{o_i}, \OO_{o_i})$ with $\XX_{o_1}\sqcup \OO_{o_1} = \XX_{o_2}\sqcup \OO_{o_2}$.

  Let $\iota:\CTd(\T_{o_1}, n)\to \CTd(\T_{o_2}, n)$ be the isomorphism induced by the natural correspondence between generators and domains. We discuss the degree of $\iota$ below.

  Let $\x_{o_1} = \x_{o_1}^L\sqcup \x_{o_1}^R\in \SS(\HD_{o_1})$ and $\x_{o_2} = \x_{o_2}^L\sqcup \x_{o_2}^R\in \SS(\HD_{o_2})$ be corresponding generators.  The $\delta$-grading is given by
  \begin{eqnarray*}
    \delta(\x_{o_i}^R) &=& \inv(\x_{o_i}^R)- \frac 1 2 \inv(\x_{o_i}^R, \XX_{o_i}^R\sqcup \OO_{o_i}^R)  + \frac 1 2 \inv(\XX_{o_i}^R) +\frac 1 2 \inv(\OO_{o_i}^R) + \frac 1 2 |\XX_{o_i}^R|,\\
    \delta(\x_{o_i}^L) &=& -\inv(\x_{o_i}^L)+\frac 1 2 \inv(\x_{o_i}^L, \XX_{o_i}^L\sqcup \OO_{o_i}^L)  - \frac 1 2 \inv(\XX_{o_i}^L) - \frac 1 2 \inv(\OO_{o_i}^L)-\frac 1 2  |\OO_{o_i}^L|.
  \end{eqnarray*}
  The first and second term in each formula do not depend on orientations, and neither does the sum of the fifth terms. In other words, $\inv(\x_{o_1}^R)=\inv(\x_{o_2}^R)$, $|\OO_{o_1}^L|-|\XX_{o_1}^R| = |\OO_{o_2}^L|-|\XX_{o_2}^R|$, and so on.

  The only points in $\XX_{o_i}\sqcup \OO_{o_i}$ that contribute to the third and fourth terms in each formula are those that correspond to a strand that runs over or under another strand in the respective half of the diagram. If there are no crossings, then corresponding generators have the same $\delta$-grading, $e=0$, and we are done. If   $T_{o_1}$ and $T_{o_2}$ are elementary tangles for a crossing, we discuss the contribution of the two relevant basepoints below.

  Suppose the strand with the higher slope crosses over the strand with the lower slope. Then the crossing is encoded in the right grid of the diagram. There are four possible orientations of the two relevant strands. If they point in the same horizontal direction, then $ \frac 1 2 \inv(\XX_{o_i}^R) +\frac 1 2 \inv(\OO_{o_i}^R) = \frac 1 2 $; if they point in opposite horizontal directions, then $ \frac 1 2 \inv(\XX_{o_i}^R) +\frac 1 2 \inv(\OO_{o_i}^R) = 0$. Observe that in the former case the crossing is negative, whereas in the latter case the crossing is positive.

  Suppose the strand with the lower slope crosses over the strand with the higher slope, i.e. the crossing is encoded in the left half of the diagram. If the two strands point in the same horizontal direction, then $-\frac 1 2 \inv(\XX_{o_i}^L) -\frac 1 2 \inv(\OO_{o_i}^L) =- \frac 1 2 $; if they point in opposite horizontal directions, then $- \frac 1 2 \inv(\XX_{o_i}^L) -\frac 1 2 \inv(\OO_{o_i}^R) = 0$. In the former case the crossing is positive, whereas in the latter case the crossing is negative.

  Thus, $\delta(\x_{o_2}) - \delta(\x_{o_1}) = e/2$, and so $\iota$ shifts degrees by  $e/2$.
\end{proof}

Working in the homotopy category, as we did in the proof of Theorem~\ref{thm:ourtheorem}, we now define the morphisms $F_k, \Phi_k, \Psi_k$:
\begin{enumerate}
  \item The morphism $F_k \colon \CDTDd (\T_k, n) \to \CDTDd (\T_{k+1}, n)$ is defined by
    \[
      F_k = \iota_k' \boxtimes f_k \boxtimes \iota_k''.
    \]
  \item The morphism $\Phi_k \colon \CDTDd (\T_k, n) \to \CDTDd (\T_{k+2}, n)$ is defined by
    \[
      \Phi_k = (\iota_{k+1}' \circ \iota_k') \boxtimes \phi_k \boxtimes (\iota_{k+1}'' \circ \iota_k'').
    \]
  \item The morphism $\Psi_k \colon \CDTDd (\T_k, n) \to \CDTDd (\T_k, n)$ is defined by
    \[
      \Psi_k = \Id_k' \boxtimes \psi_k \boxtimes \Id_k''.
    \]
\end{enumerate}
where $f_k$, $\phi_k$, $\psi_k$ are defined in the same way as in Section~\ref{sec:proof}, except that the Heegaard diagrams here contain both $X$ and $O$ markings. In the definitions above, we omit the degree shifts, since they depend on $k$. We discuss these shifts below.

We first establish the following lemma. 

\begin{lem}\label{lem:deltarect}
  Suppose $\HD$ is a genus $1$ Heegaard diagram for a tangle such as the ones in Figures~\ref{fig:hd_elem_ori} and \ref{fig:hd_012}, $\x$ and $\y$ are generators in $\SS(\HD)$, and $r$ is a (not necessarily empty) rectangle from $\x$ to $\y$ of one of the first three types discussed in Section \ref{ssec:tf}  such that $\Int r \cap \x = \emptyset$. Then
  \[
    \delta(\algl{r})+ \delta(\y)+ \delta(\algr{r})-\delta(\x) = n_{\OO}(r)+ n_{\XX}(r)-1.
  \]
\end{lem}

\begin{proof}
  Cut the Heegaard diagram $\HD$ open, and embed it on the plane as in Figure~\ref{fig:hd_cut}. 

  First we consider the case when $r$ is an interior rectangle. Then $\algl{r}$ and $\algr{r}$ are idempontents, hence in degree $0$, and
  \begin{align*}
    \delta(\y) - \delta(\x) &= \inv(\y^R) - \inv(\x^R)- \frac 1 2 \inv(\y^R, \XX^R\sqcup \OO^R)+ \frac 1 2 \inv(\x^R, \XX^R\sqcup \OO^R)\\
    &-\inv(\y^L)+ \inv(\x^L) +\frac 1 2 \inv(\y^L, \XX^L\sqcup \OO^L) -\frac 1 2 \inv(\x^L, \XX^L\sqcup \OO^L).
  \end{align*}

  There are six sub-cases, depending on which parts of the diagram $r$ occupies; see Figure~\ref{fig:rect_gr}. In all cases, the two $\beta$-circles containing the vertical edges of $r$ bound an annulus, and points in $\x\cap \y$, $\XX$, and $\OO$ outside that annulus contribute the same amount to $\delta(\y)$ and to $\delta(\x)$. The two horizontal edges of $r$, along with the top, middle (take any horizontal line between $\alpha^L_0$ and $\alpha^R_0$), and bottom edges of $\HD$, divide the annulus into four rectangular regions, which we denote $A$, $B$, $C$, and $D$, from top to bottom. Define
  \[
    \begin{array}{lll}
      a = |\x\cap \Int A|, && a' = |(\XX\sqcup \OO)\cap \Int A|,\\
      b = |\x\cap \Int B|, && b' = |(\XX\sqcup \OO)\cap \Int B|,\\
      c = |\x\cap \Int C|, && c' = |(\XX\sqcup \OO)\cap \Int C|,\\
      d = |\x\cap \Int D|, && d' = |(\XX\sqcup \OO)\cap \Int D|.
    \end{array}
  \]
  Let $w$ be the width of the annulus, i.e. the number of components of $\Sigma\setminus \betas$ contained in it. Observe that $a+b+c+d = w-1$, and $a'+b'+c'+d' = 2w$.

  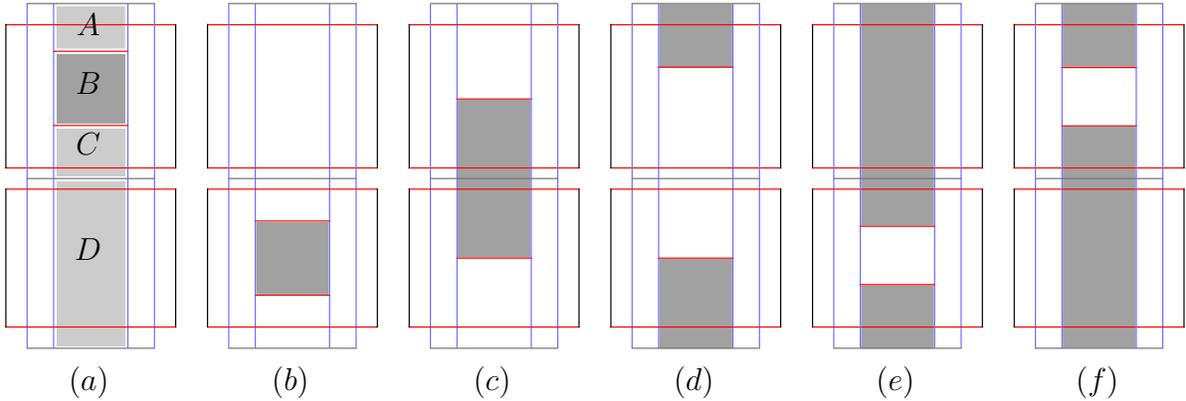
\begin{figure}[h]
    \centering
    \labellist
    \pinlabel  $A$ at 125 490
    \pinlabel  $B$ at 125 400
    \pinlabel  $C$ at 125 310
    \pinlabel  $D$ at 125 150
    \pinlabel $(a)$ at 126 -50
    \pinlabel $(b)$ at 430 -50
    \pinlabel $(c)$ at 734 -50
    \pinlabel $(d)$ at 1038 -50
    \pinlabel $(e)$ at 1342 -50
    \pinlabel $(f)$ at 1646 -50
    \endlabellist
    \includegraphics[scale = .26]{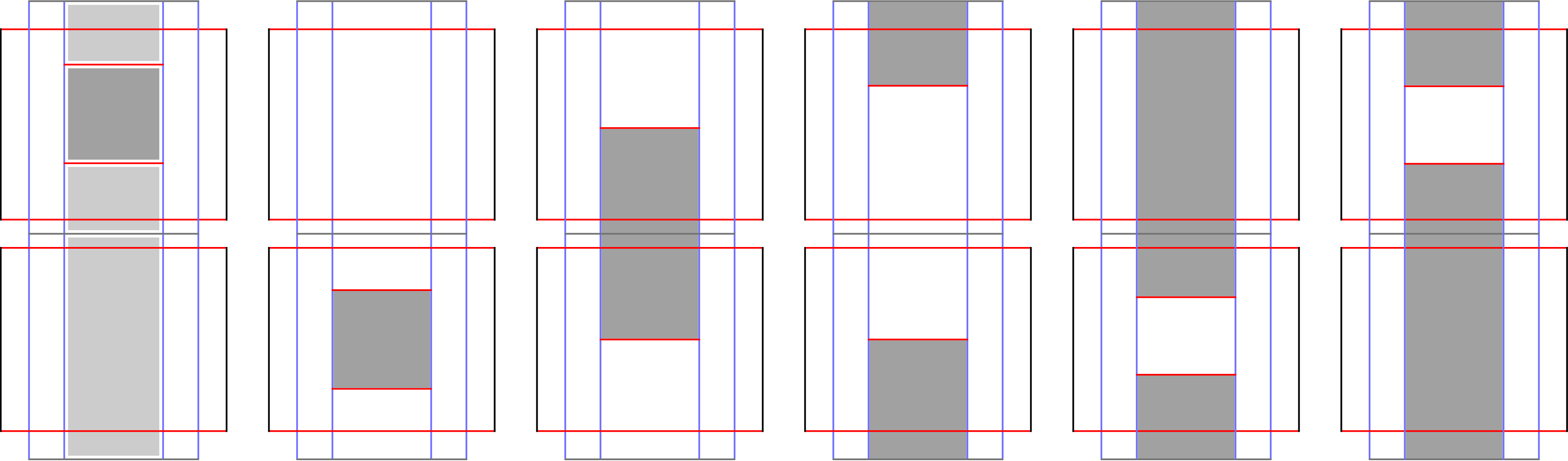}
    \vskip .6 cm
    \caption{The diagram $\HD$ on the plane. There are six sub-cases of an interior rectangle $r$ (depicted in dark grey). In the first case, the vertical annulus for $r$ is also illustrated; it is divided  by the top, middle, and bottom edges of $\HD$, along with the horizontal edges of $r$, into four rectangles, denoted $A$, $B$, $C$, and $D$ from top to bottom.}
    \label{fig:rect_gr}
  \end{figure}

  Suppose $r$ is entirely contained in the top half of the diagram; see Figure~\ref{fig:rect_gr}~(a). Then
  \[
    \delta(\y) - \delta(\x) = \inv(\y^R) - \inv(\x^R)- \frac 1 2 \inv(\y^R, \XX^R\sqcup \OO^R)+ \frac 1 2 \inv(\x^R, \XX^R\sqcup \OO^R).
  \]
  Since the support of $r$ coincide with $B$, we have $b=0$, so $\inv(\y^R) - \inv(\x^R) = -1$. If a point $p\in \XX^R\sqcup \OO^R$ contributes different amounts to  $\inv(\y^R, \XX^R\sqcup \OO^R)$ and  to $\inv(\x^R, \XX^R\sqcup \OO^R)$, it must be contained in the interior of $r$. Each such $p$ contributes two more times to the latter---when paired with the bottom-left and top-right corners of $r$---so $- \frac 1 2 \inv(\y^R, \XX^R\sqcup \OO^R)+ \frac 1 2 \inv(\x^R, \XX^R\sqcup \OO^R) = n_{\OO}(r)+ n_{\XX}(r)$.

  The computation when $r$ is entirely contained in the bottom half of the diagram  (see Figure~\ref{fig:rect_gr}~(b)) is analogous. 

  Suppose $r$ intersects both halves of the diagram, but not the top and bottom edges; see Figure ~\ref{fig:rect_gr}~(c). Note that in this case $b=c=0$. A point in $\x\cap \Int A = \y\cap \Int A$ contributes one more time to $\inv(\y^R)$ than to $ \inv(\x^R)$---when paired with the top left corner of $r$. Similarly, a point in $(\XX^R\sqcup \OO^R)\cap \Int A$  contributes one more time to $\inv(\y^R, \XX^R\sqcup \OO^R)$ than to $\inv(\x^R, \XX^R\sqcup \OO^R)$. Points inside $B$ contribute in the opposite way. Counting in the bottom half of the diagram is analogous. We see that
  \begin{align*}
    \inv(\y^R) - \inv(\x^R) &= a - b = a,\\
    -\inv(\y^R, \XX^R\sqcup \OO^R)+\inv(\x^R, \XX^R\sqcup \OO^R) &=  -a' + b',\\
    -\inv(\y^L)+ \inv(\x^L) &= -c+d = d,\\
    \inv(\y^L, \XX^L\sqcup \OO^L) - \inv(\x^L, \XX^L\sqcup \OO^L) &= c' - d',
  \end{align*}
  so
  \[
    \delta(\y) - \delta(\x) = a+d +\frac 1 2 (b'+c'-a'-d') = w - 1 + \frac 1 2 (2b'+2c'-2w) = b'+c'-1 = n_{\OO}(r)+ n_{\XX}(r)-1.
  \]

  The computation when $r$ intersects both halves of the diagram  but not the middle edge (see Figure~\ref{fig:rect_gr}~(d)) is analogous. In that case $a=d=0$, and we get
  \[
    \delta(\y) - \delta(\x) = b+c +\frac 1 2 (a'+d'-b'-c') = w - 1 + \frac 1 2 (2a'+2d'-2w) = a'+d'-1 = n_{\OO}(r)+ n_{\XX}(r)-1.
  \]

  Suppose $r$ intersects both halves of the diagram and has all four corners in the bottom half; see Figure~\ref{fig:rect_gr}~(e). In this case $a=b=d=0$, so $c=w-1$. If a point $p\in \x\cap \y$ contributes different amounts to $\inv(\y^L)$ and to $\inv(\x^L)$, it must be contained in rectangle $C$; each such $p$ contributes two more times to $\inv(\x)$---when paired with the two points in $\x\setminus \y$. The pair formed by the two points in $\y\setminus \x$ does not contribute to $\inv(\y)$, and the pair formed by the two points in $\x\setminus \y$ contributes once to $\inv(\x)$. Thus,
  \[
    -\inv(\y^L)+ \inv(\x^L) = 1+2c = 2w-1.
  \]
  Similarly,
  \[
    \inv(\y^L, \XX^L\sqcup \OO^L) - \inv(\x^L, \XX^L\sqcup \OO^L) = -2c' = 2(a'+b'+c'-2w).
  \]
  Since
  $\y^R = \x^R$, we have
  \[
    \delta(\y) - \delta(\x) = \delta(\y^L)- \delta(\x^L) = 2w - 1 +a'+b'+c'-2w  = n_{\OO}(r)+ n_{\XX}(r)-1.
  \]

  The computation when $r$  intersects both halves of the diagram and has all four corners in the top half (see Figure~\ref{fig:rect_gr}~(f)) is analogous. 

  There are two cases when $r$ intersects $\bdy \HD$, and they are similar to each other. We discuss the case when $r$ intersects $\bdy^R\HD$. Say that the boundary of $r$ intersects $\alpha_i^R$ and $\alpha_j^R$. Let $t$ be the number of arcs between $\alpha_i^R$ and $\alpha_j^R$ occupied by $\x\cap \y$. Since $r$ does not contains points in $\x\cap \y$, the $t$ points in $\x\cap\y$ at heights between  $\alpha_i^R$ and $\alpha_j^R$ are all outside $r$. Whether $r$ touches the left or the right edge of the diagram, these are exactly the points that contribute  differently to  $\inv(\y^R)$ and to $\inv(\x^R)$. Thus,
  \[
    \inv(\y^R)- \inv(\x^R) = t.
  \]
  The $|i-j|-t-1$ arcs between $\alpha_i^R$ and $\alpha_j^R$ that are unoccupied contribute to the horizontal black strands in $\algr{r}$ that intersect the unique non-horizontal strand. There are also $|i-j|$ orange strands that intersect the non-horizontal strand. So
  \[
    \delta(\algr{r}) =  \diagup \hspace{-.35cm}\diagdown(\algr{r}) - \frac{\diagup \hspace{-.37cm}{\color{orange2}{\diagdown}}(\algr{r}) +\diagdown \hspace{-.37cm}{\color{orange2}{\diagup}}(\algr{r})}{2} = |i-j|-t-1 - \frac{|i-j|}{2}.
  \]
  The $n_{\XX}(r) + n_{\OO}(r)$ points in $(\XX^R\sqcup\OO^R)\cap r$ contribute one more time to   $\inv(\x^R, \XX^R\sqcup \OO^R)$ than to $\inv(\y^R, \XX^R\sqcup \OO^R)$, and the remaning  $|i-j| - n_{\XX}(r) - n_{\OO}(r)$ points in $\XX^R\sqcup \OO^R$ at heights between $\alpha_i^R$ and $\alpha_j^R$  contribute one more time to   $\inv(\y^R, \XX^R\sqcup \OO^R)$ than to $\inv(\x^R, \XX^R\sqcup \OO^R)$. The remaining points in $\XX^R\sqcup \OO^R)$ contribute the same to both counts. Thus,
  \[
    -\inv(\y^R, \XX^R\sqcup \OO^R)+\inv(\x^R, \XX^R\sqcup \OO^R) = 2n_{\XX}(r) + 2n_{\OO}(r) - |i-j|.
  \]
  We see that
  \[
    \delta(\algl{r})+ \delta(\y)+ \delta(\algr{r})-\delta(\x) = n_{\OO}(r)+ n_{\XX}(r)-1.\qedhere
  \]
\end{proof}

Next, we set up some notation. 

Let $\HD_\infty, \HD_0, \HD_1$ be the Heegaard diagrams for $\eT_\infty, \eT_0, \eT_1$ from Section~\ref{sec:proof}. We can think of these three diagrams as the same diagram $(\Sigma, \alphas, \betas')$, but with three different choices of $\XX$ and $\OO$ markings, by identifying each of $\beta_{i, \infty}, \beta_{i, 0}, \beta_{i, 1}$ with a curve $\beta_i\in \betas'$; see Figure~\ref{fig:delta_diagram}. Denote the sets of $\XX$ and $\OO$ markings by $\XX_k$ and $\OO_k$ respectively, for $k\in \{\infty, 0, 1\}$;  Denote the four points of $\XX_k \sqcup \OO_k$ in regions borderding $\beta_i$ by  $X_{1,k}, X_{2,k}, X_{3,k}, X_{4,k}\in \XX_k\sqcup\OO_k$, indexed by relative height as seen in Figure~\ref{fig:delta_diagram}. 

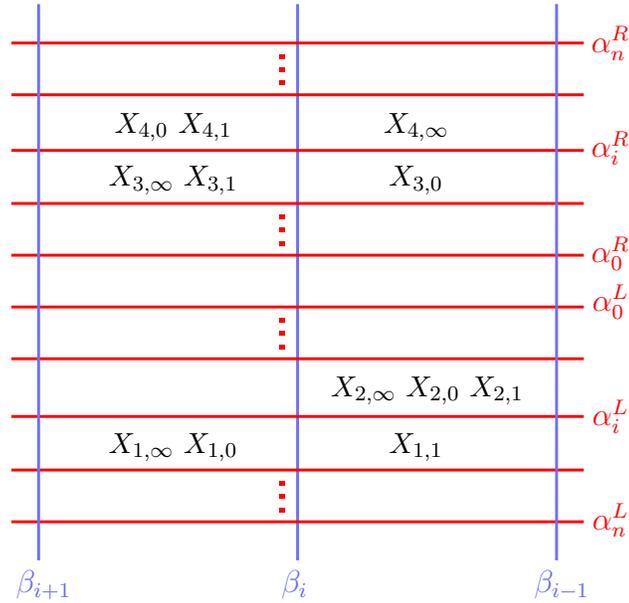
\begin{figure}[h]
  \centering
  \labellist
  \pinlabel  \small{$X_{1, \infty}$ $X_{1,0}$} at 100 70
  \pinlabel  \small{$X_{1,1}$} at 250 70
  \pinlabel  \small{$X_{2, \infty}$ $X_{2,0}$ $X_{2,1}$} at 257 105
  \pinlabel  \small{$X_{3, \infty}$ $X_{3,1}$} at 100 235
  \pinlabel  \small{$X_{4,0}$ $X_{4,1}$} at 100 268
  \pinlabel  \small{$X_{3,0}$} at 250 235
  \pinlabel  \small{$X_{4, \infty}$} at 250 268
  \pinlabel \textcolor{blue!60!}{$\beta_{i}$} at 175 -15
  \pinlabel \textcolor{blue!60!}{$\beta_{i-1}$} at 340 -15
  \pinlabel \textcolor{blue!60!}{$\beta_{i+1}$} at 20 -15
  \pinlabel \textcolor{red}{$\alpha_{n}^L$} at 370 25
  \pinlabel \textcolor{red}{$\alpha_{i}^L$} at 370 89
  \pinlabel \textcolor{red}{$\alpha_{0}^L$} at 370 160
  \pinlabel \textcolor{red}{$\alpha_{0}^R$} at 370 190
  \pinlabel \textcolor{red}{$\alpha_{i}^R$} at 370 255
  \pinlabel \textcolor{red}{$\alpha_{n}^R$} at 370 320
  \endlabellist
  \includegraphics[scale = .61]{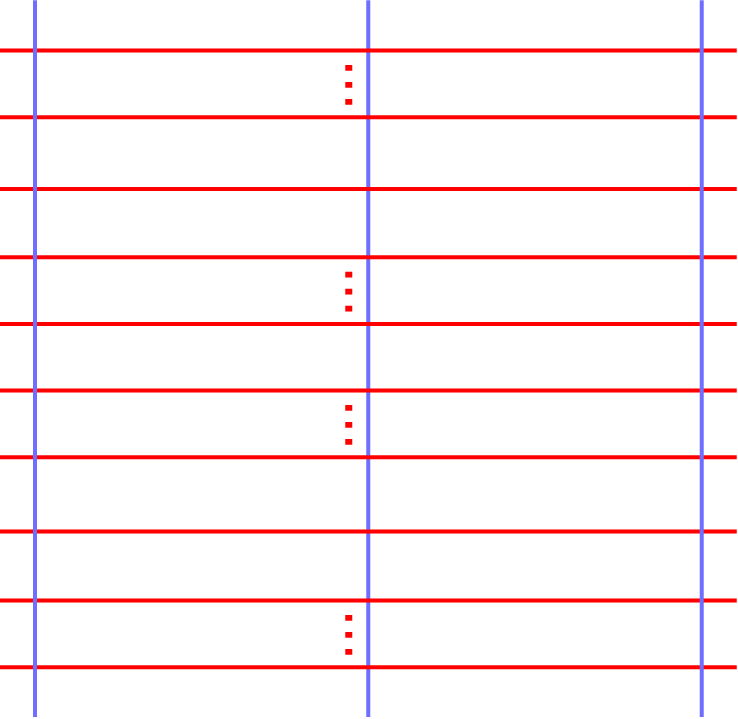}
  \vskip .4 cm
  \caption{The three diagrams $\HD_\infty, \HD_0, \HD_1$ seen as the same union of bordered grids, but with different markings.}
  \label{fig:delta_diagram}
\end{figure}

Given a set of intersection points $\x\in\alphas\cap\betas'$ with exactly one point on each $\beta\in \betas'$  and at most one point on each $\alpha\in \alphas$, there are corresponding generators $\x_{\infty}\in \SS(\HD_{\infty})$, $\x_0\in \SS(\HD_0)$, and $\x_1\in \SS(\HD_1)$. For $\set{k,l}\subset \set{\infty, 0, 1}$,  this induces a correspondence between generators of $\SS(\HD_k)$ and generators of $\SS(\HD_l)$, and we will use the notation $\x_k, \x_l$, or $\y_k, \y_l$, etc., to denote corresponding generators on different diagrams. To say this differently, we will think of a generator in any of the three sets $\SS(\HD_k)$ as a set of intersection points  $\x\in\alphas\cap\betas'$, and use the subscript $k$ to stress which generating set we are thinking of. 

In the following lemma, we compute the degree of any of $f_k$, $\phi_k$, and $\psi_k$ when the two elementary tangles associated to the given map have \emph{compatible orientations}, i.e. when they are identical as oriented tangles, except near a point; see Figure~\ref{lem:compat}. Note that since they are elementary tangles, $\eT_k$ and $\eT_l$ have compatible orientations exactly when they have the same oriented boundaries, so in particular the respective algebras are the same. 

\begin{figure}[ht]
  \centering
  \includegraphics[scale = .8]{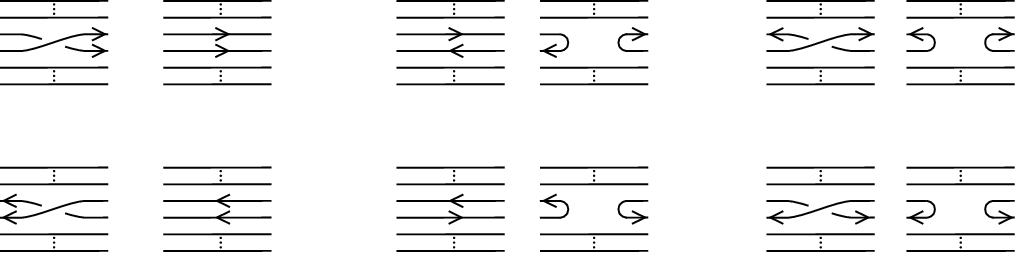}
  \vskip .4 cm
  \caption{The six pairs of elementary tangles with compatible orientations.}
  \label{fig:Tel_ori}
\end{figure}

\begin{lem}
  \label{lem:compat}
  If  $\eT_k$ and $ \eT_{k+1}$ have compatible orientations, then $\deg_{\delta} (f_k) = -1/2$. Similarly, if $\eT_k$ and $ \eT_{k+2}$ have compatible orientations, then $\deg_{\delta} (\phi_k) = 1/2$. For any orientation on $\eT_k$, $\deg_{\delta} (\psi_k) = 1$.
\end{lem}

\begin{proof}
  Suppose $\eT_k$ and $\eT_l$ have compatible orientations. Recall our maps are defined by counting polygons. For a polygon $p$ from $\x_k$ to $\y_l$, define
  \[
    \delta(p) = \delta(\algl{p})+ \delta(\y_l)+ \delta(\algr{p})-\delta(\x_k).
  \]

  Observe that triangle-like polygons always connect pairs of corresponding generators, so if $p$ is a triangle-like polygon from $\x_k$ to $\y_l$, then $\y_l=\x_l$, and we have
  \[
    \delta(p) =  \delta(\x_l)-\delta(\x_k).
  \]
  Computing the degrees of rectangle-like polygons will require more work. Given a rectangle-like polygon $p$ from $\x_k\in \SS(\HD_k)$ to $\y_l\in \SS(\HD_l)$, define the \emph{straightening} of $p$ to be the unique rectangle $\str{p}$ in $\HD_l$ from $\x_l$ to $\y_l$, and  observe that
  \[
    \delta(p) = [\delta(\algl{p})+ \delta(\y_l)+ \delta(\algr{p})-\delta(\x_l)] + [\delta(\x_l) - \delta(\x_k)].
  \]
  Note that $\algl{p} = \algl{\str{p}}$ and $\algr{p} = \algr{\str{p}}$, so the sum in the first set of brackets on the right hand side is $\delta(\str{p})$, which by Lemma~\ref{lem:deltarect} equals $ \numOO{\str{p}}+\numXX{\str{p}}-1$. The difference in the second set of brackets compares the degrees of corresponding generators, and can be reduced to comparing how the different $\XX$ and $\OO$ markings for $\HD_l$ and for $\HD_k$ affect the degree. In summary,
  \begin{equation*}
    \delta(p) =
    \begin{cases}
      \delta(\x_l)-\delta(\x_k) &  \text{if $p$ is triangle-like,}\\
      \delta(\str{p})+ \delta(\x_l)-\delta(\x_k) & \text{if $p$ is rectangle-like.}
    \end{cases}
  \end{equation*}

  For each point $q \in \alphas \cap \betas'$, define the \emph{relative height} $\hgt (q)$ as follows. We say that
  \begin{enumerate}
    \item $\hgt (q) = 1$ if $q$ is below $\leftbdy{\alpha_i}$ as seen in Figure~\ref{fig:delta_diagram};
    \item $\hgt (q) = 2$ if $q \in \leftbdy{\alpha_i}$;
    \item $\hgt (q) = 3$ if $q$ is above $\leftbdy{\alpha_i}$ and in the bottom grid  as seen in Figure~\ref{fig:delta_diagram} (this is what we called the \emph{left} grid in Section \ref{ssec:tf});
    \item $\hgt (q) = 4$ if $q$ is below $\rightbdy{\alpha_i}$ and in the top grid as seen in Figure~\ref{fig:delta_diagram}  (this is what we called the \emph{right} grid in Section \ref{ssec:tf});
    \item $\hgt (q) = 5$ if $q \in \rightbdy{\alpha_i}$; and
    \item $\hgt (q) = 6$ if $q$ is above $\rightbdy{\alpha_i}$ as seen in Figure~\ref{fig:delta_diagram}.
  \end{enumerate}
  Below, we provide a simple formula for  $\delta(\x_l) - \delta(\x_k)$ based on the relative height of the point of $\x$ that lies on $\beta_i$. Recall that
  \begin{align*}
    \delta(\x_l) - \delta(\x_k) &= \inv(\x_l^R)- \frac 1 2 \inv(\x_l^R, \XX_l^R\sqcup \OO_l^R)  + \frac 1 2 \inv(\XX_l^R) +\frac 1 2 \inv(\OO_l^R) + \frac 1 2 |\XX_l^R|\\
    & \quad -\inv(\x_l^L)+\frac 1 2 \inv(\x_l^L, \XX_l^L\sqcup \OO_l^L)  - \frac 1 2 \inv(\XX_l^L) - \frac 1 2 \inv(\OO_l^L)-\frac 1 2  |\OO_l^L|\\
    & \quad -\inv(\x_k^R)+ \frac 1 2 \inv(\x_k^R, \XX_k^R\sqcup \OO_k^R)  - \frac 1 2 \inv(\XX_k^R) -\frac 1 2 \inv(\OO_k^R) - \frac 1 2 |\XX_k^R|\\
    & \quad + \inv(\x_k^L)-\frac 1 2 \inv(\x_k^L, \XX_k^L\sqcup \OO_k^L)  + \frac 1 2 \inv(\XX_k^L) +\frac 1 2 \inv(\OO_k^L)+\frac 1 2  |\OO_k^L|.
  \end{align*}
  We will group up some terms on the right hand side to simplify. Since $\x_k$ and $\x_l$ are the same set $\x\in\alphas\cap\betas'$, then
  \begin{align*}
    \inv(\x_l^R) &= \inv(\x_k^R) = \inv(\x^R),\\
    \inv(\x_l^L) & = \inv(\x_k^L) = \inv(\x^L),
  \end{align*}
  and since the two tangles have compatible orientations, $|\XX_l^R| = |\XX_k^R|$ and $|\OO_l^L| = |\OO_k^L|$. Since the sets of basepoints for the two diagrams only differ in their subsets $\{X_{1,k}, X_{2,k}, X_{3,k}, X_{4,k}\}$ and $\{X_{1,l}, X_{2,l}, X_{3,l}, X_{4,l}\}$, and further $X_{t,k}$ and $X_{t,l}$ are both $X$s or both $O$s for $1\leq t \leq 4$, comparing $\inv(\XX_l^R)$ to $\inv(\XX_k^R)$ reduces to comparing $\inv(\XX_l^R\cap \{X_{3,l}, X_{4,l}\})$ to $\inv(\XX_k^R\cap\{X_{3,k}, X_{4,k}\})$. Each of the latter counts is nonzero exactly when the corresponding tangle is a negative crossing with both strands at the crossing oriented to the left. So
  \begin{equation*}
    \inv(\XX_l^R) - \inv(\XX_k^R) =
    \begin{cases}
      1 &  \text{if $\eT_l$ is a crossing, $X_{3,l}, X_{4,l}$ are both $X$s, and $k\neq l$,}\\
      -1 &  \text{if $\eT_k$ is a crossing, $X_{3,k}, X_{4,k}$ are both $X$s, and $k\neq l$,}\\
      0 & \text{otherwise.}
    \end{cases}
  \end{equation*}
  Similarly,
  \begin{equation*}
    \inv(\OO_l^R) - \inv(\OO_k^R) =
    \begin{cases}
      1 &  \text{if $\eT_l$ is a crossing, $X_{3,l}, X_{4,l}$ are both $O$s, and $k\neq l$,}\\
      -1 &  \text{if $\eT_k$ is a crossing, $X_{3,k}, X_{4,k}$ are both $O$s, and $k\neq l$,}\\
      0 & \text{otherwise.}
    \end{cases}
  \end{equation*}
  Note that since we are assuming compatible orientations, the non-zero counts in the displayed equations above occur exactly when $\{k,l\} = \{\infty, 0\}$. More precisely,
  \begin{equation*}
    \inv(\XX_l^R) + 	\inv(\OO_l^R) - \inv(\XX_k^R) - \inv(\OO_k^R) =
    \begin{cases}
      1 &  \text{if $(k,l) = (0,\infty)$,}\\
      -1 &  \text{if $(k,l) = (\infty, 0)$,}\\
      0 & \text{otherwise.}
    \end{cases}
  \end{equation*}
  In the elementary tangles that we consider, crossings only appear in the right grid, so it follows that $\inv(\XX_l^L) - \inv(\XX_k^L)=0$ and $\inv(\OO_l^L) - \inv(\OO_k^L)=0$.

  Now we look at the four terms that count inversions between generators and basepoints.  Since $\x_k$ and $\x_l$ are the same set of points on the common diagram, and $\XX_k\sqcup \OO_k$ and $\XX_l\sqcup \OO_l$ only differ in their subsets $\{X_{1,k}, X_{2,k}, X_{3,k}, X_{4,k}\}$ and $\{X_{1,l}, X_{2,l}, X_{3,l}, X_{4,l}\}$, we get
  \[
    \frac 1 2 \inv(\x_k^R, \XX_k^R\sqcup \OO_k^R) - \frac 1 2 \inv(\x_l^R, \XX_l^R\sqcup \OO_l^R) = \frac 1 2 \inv(\x_k^R, \{X_{3,k}, X_{4,k}\}) - \frac 1 2 \inv(\x_l^R, \{X_{3,l}, X_{4,l}\})),
  \]
  and
  \[
    \frac 1 2 \inv(\x_l^L, \XX_l^L\sqcup \OO_l^L)  -\frac 1 2 \inv(\x_k^L, \XX_k^L\sqcup \OO_k^L) =\frac 1 2  \inv(\x_l^L, \{X_{1,l}, X_{2,l}\}) - \frac 1 2 \inv(\x_k^L, \{X_{1,k}, X_{2,k}\})).
  \]
  Note that $X_{j, k}$ and $X_{j, l}$ are either in the same region, or lie in adjacent regions separated by $\beta_i$, so  for $p\in \x$,  the pairs $(p, X_{j,k})$ and $(p, X_{j,l})$ may only contribute differently to the above counts if $p$ lies on $\beta_i$. Let $x$ be the point in $\x$ that lies on $\beta_i$. If $\hgt(x)\leq 3$, then $x$ is in the left grid, and the sum of the four terms reduces to
  \[
    \frac 1 2  \inv(x, \{X_{1,l}, X_{2,l}\}) - \frac 1 2 \inv(x, \{X_{1,k}, X_{2,k}\})),
  \]
  By inspection of Figure~\ref{fig:delta_diagram}, we see that this count is zero if $\{k,l\} = \{\infty, 0\}$ or if $k=l$,  it is $\frac 1 2 $ if $k=1$, $l\in \{\infty, 0\}$, and $\hgt(x)=1$, or if $l=1$, $k\in \{\infty, 0\}$, and $\hgt(x)\in \{2,3\}$, and it is $-\frac 1 2$ in the remaining cases.  If $\hgt(x)\geq 4$, then $x$ is in the right grid, and the formula reduces to
  \[
    -\frac 1 2 \inv(x, \{X_{3,l}, X_{4,l}\})+ \frac 1 2 \inv(x, \{X_{3,k}, X_{4,k}\}),
  \] and one can compute this number by direct inspection of Figure~\ref{fig:delta_diagram} again. 

  To sum up, we reduced the grading difference to
  \begin{align*}
    \delta(\x_l) - \delta(\x_k) = & \frac 1 2 (\inv(\XX_l^R) + \inv(\OO_l^R) - \inv(\XX_k^R) - \inv(\OO_k^R))\\
    &+
    \begin{cases}
      \frac 1 2  \inv(x, \{X_{1,l}, X_{2,l}\}) - \frac 1 2 \inv(x, \{X_{1,k}, X_{2,k}\})) &  \text{if $\hgt(x)\leq 3$,}\\
      -\frac 1 2 \inv(x, \{X_{3,l}, X_{4,l}\})+ \frac 1 2 \inv(x, \{X_{3,k}, X_{4,k}\}) &  \text{if $\hgt(x)\geq 4$,}
    \end{cases}
  \end{align*}
  and computed the right hand side, depending on $k$ and $l$. The final result is summarized in Table~\ref{tab:delta_gen}.
  \begin{table}[h]
    \captionsetup{belowskip=10pt}
    \centering
    \begin{tabular}{|c||c|c|c|}
      \hline
      \backslashbox{$\hgt (x)$}{$(k, l)$} & $(\infty, 0)$ & $(0,1)$  & $(1, \infty)$ \\
      \hline \hline
      $1$ & $-1/2$ & $-1/2$ & $1/2$ \\
      \hline
      $2$ & $-1/2$ & $1/2$ & $-1/2$ \\
      \hline
      $3$ & $-1/2$ & $1/2$ & $-1/2$ \\
      \hline
      $4$ & $-1/2$ & $1/2$ & $-1/2$ \\
      \hline
      $5$ & $1/2$ & $-1/2$ & $-1/2$ \\
      \hline
      $6$ & $-1/2$ & $-1/2$ & $1/2$ \\
      \hline
    \end{tabular}
    \caption{The difference $\delta(\x_l) - \delta(\x_k)$, depending on $k$, $l$, and $\hgt(x)$, where  $x$ is the component of $\x$ in $\beta_i$.}
    \label{tab:delta_gen}
  \end{table}

  We have now done enough preliminary work to allow for a quick and simple computation of the degree of any polygon.

  We compute the change in the $\delta$-grading under $f_\infty \colon \CDTDd (\HD_\infty) \to \CDTDd (\HD_0)$ as follows. Recall that $f_{\infty} = \TT_{\infty}+ \PP_{\infty}$, and let $p$ be a triangle or a pentagon from $\x_\infty \in \SS (\HD_\infty)$ to $\y_0 \in \SS (\HD_0)$. We are to determine the value of $\delta(p)$. If $p$ is a triangle, this is already given by $\delta(\x_l) - \delta(\x_k)$, see Table~\ref{tab:delta_gen}. If $p$ is a pentagon, we need to understand $\delta(\str{p})$. 

  As above, if $\x$ is the set of intersection points in $\alphas \cap \betas'$ that corresponds to $\x_\infty$ and $\x_0$, then  let $x$ be the component of $\x$ in $\beta_i$. Similarly, if $\y$ is the set of intersection points in $\alphas \cap \betas'$ that corresponds to $\y_0$, then let $y$ be the component of $\y$ in $\beta_i$.

  By Lemma~\ref{lem:deltarect}, $\delta(\str{p}) = \numOO{\str{p}}+\numXX{\str{p}}-1$. Since $p$ is empty, and clearly $p$ and $\str{p}$ contain the same basepoints in regions not bordering $\beta_i$, then $\str{p}$ maybe only contain basepoints in $\{X_{1,0}, X_{2,0}, X_{3,0}, X_{4,0}\}$. So $\numOO{\str{p}}+\numXX{\str{p}}$ only depends on $\hgt(x)$ and $\hgt(y)$, and on the type of $p$. For example, if $\hgt(x)=4$ and $\hgt(y)=6$, then the pentagon $p$ must be of type $\ran$, and hence its straightening only contains $X_{3,0}$; see Figure~\ref{fig:delta_inf_0}.
  \begin{figure}[h]
    \centering
    \labellist
    \pinlabel  \small{$X_{1,\infty}$} at 55 70
    \pinlabel  \small{$X_{2,\infty}$} at 135 105
    \pinlabel  \small{$X_{3,\infty}$} at 55 235
    \pinlabel  \small{$X_{4,\infty}$} at 135 268
    \pinlabel  \small{$X_{1,0}$} at 360 70
    \pinlabel  \small{$X_{2,0}$} at 440 105
    \pinlabel  \small{$X_{3,0}$} at 440 235
    \pinlabel  \small{$X_{4,0}$} at 360 268
    \pinlabel \textcolor{red}{$\alpha_{n}^L$} at 510 25
    \pinlabel \textcolor{red}{$\alpha_{i}^L$} at 510 89
    \pinlabel \textcolor{red}{$\alpha_{0}^L$} at 510 160
    \pinlabel \textcolor{red}{$\alpha_{0}^R$} at 510 190
    \pinlabel \textcolor{red}{$\alpha_{i}^R$} at 510 255
    \pinlabel \textcolor{red}{$\alpha_{n}^R$} at 510 320
    \pinlabel $\HD_{\infty}$ at 100 -15
    \pinlabel $\HD_{0}$ at 405 -15
    \endlabellist
    \includegraphics[scale = .61]{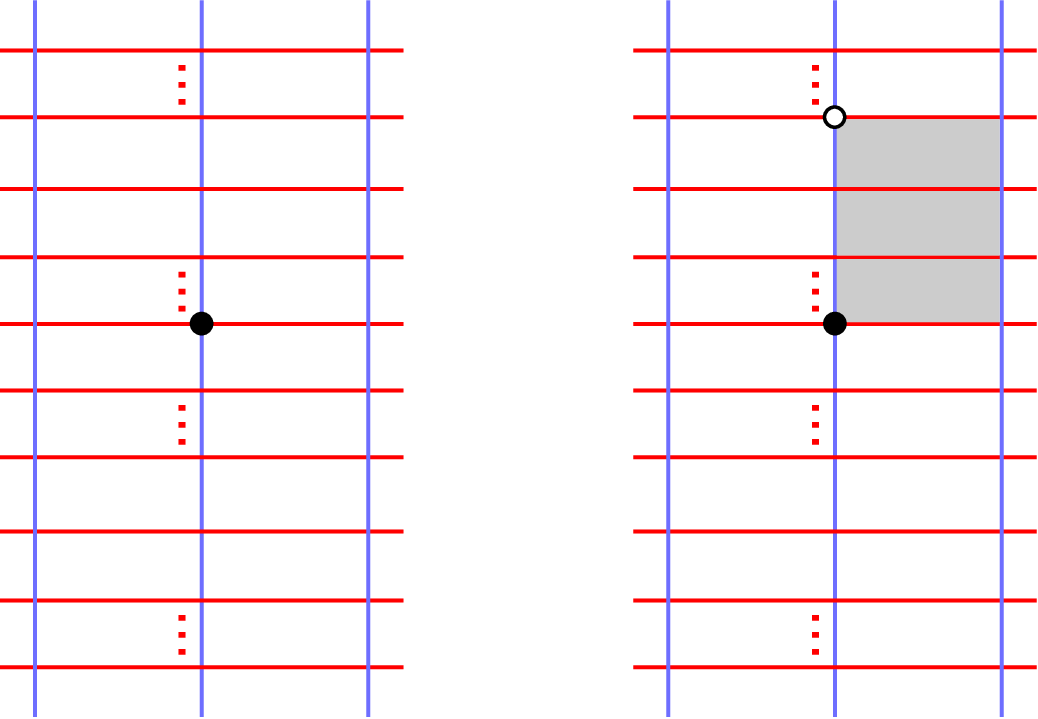}
    \vskip .5 cm
    \caption{Left: The diagram $\HD_{\infty}$; Right: The diagram $\HD_0$. The black dot, white dot, and grey rectangle are an example of of a triple $x,y,\str{p}$, where $p$ is a pentagon from a generator $\x_{\infty}$ to $\y_0$.}
    \label{fig:delta_inf_0}
  \end{figure}
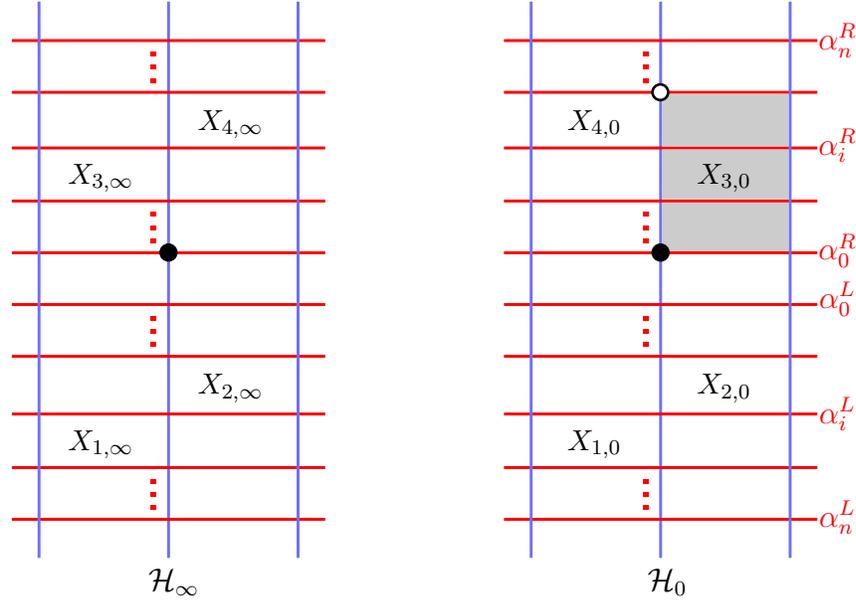
  One computes the value of $\str{p}$ for all the other possibilities of heights of generators and types of pentagons similarly, by inspecting the Heegaard diagrams directly. We summarize the computation in Table~\ref{tab:delta}.
  \begin{table}[h]
    \captionsetup{belowskip=10pt}
    \centering
    \begin{tabular}{|c||c||c|c|c|c||c||c|}
      \hline
      $\hgt (x)$ & $\delta (\x_0) - \delta (\x_\infty)$ & $p$ in & Type & $\hgt (y)$ & $(\OO_0 \cup \XX_0) \cap \str{p}$ & $\delta (\str{p})$ & $\delta(p)$\\
      \hline \hline
      \mr{2}{$1$} & \mr{2}{$-\frac{1}{2}$} & $\eTri_{\infty}$ & $\nan$ & $1$ & &  & \mr{3}{$-\frac{1}{2}$}\\
      \cline{3-7}
      & & $\ePent_{\infty}$ & $\lan$ & $2, 3$, or $4$ & $X_{4,0}$ & $1-1=0$ &\\
      \cline{1-7}
      \mr{2}{$2$} & \mr{2}{$-\frac{1}{2}$} & $\eTri_{\infty}$ & $\nan$ & $2$ & &  &\\
      \cline{3-8}
      & & $\ePent_{\infty}$ & \multicolumn{5}{c|}{Does not exist}\\
      \hline
      \mr{2}{$3$} & \mr{2}{$-\frac{1}{2}$} & $\eTri_{\infty}$ & $\nan$ & $3$ & &  & \mr{4}{$-\frac{1}{2}$}\\
      \cline{3-7}
      & & $\ePent_{\infty}$ & $\ran$ & $5, 6, 1$, or $2$ & $X_{3,0}$ & $1-1=0$ &\\
      \cline{1-7}
      \mr{2}{$4$} & \mr{2}{$-\frac{1}{2}$} & $\eTri_{\infty}$ & $\nan$ & $4$ & & &\\
      \cline{3-7}
      & & $\ePent_{\infty}$ & $\ran$ & $5, 6, 1$, or $2$ & $X_{3,0}$ & $1-1=0$ &\\
      \hline
      \mr{2}{$5$} & \mr{2}{$\frac{1}{2}$} & $\eTri_{\infty}$ & \multicolumn{5}{c|}{Does not exist}\\
      \cline{3-8}
      & & $\ePent_{\infty}$ & $\lan$ & $2, 3$, or $4$ & $\emptyset$ & $0-1=-1$ & \mr{3}{$-\frac{1}{2}$}\\
      \cline{1-7}
      \mr{2}{$6$} & \mr{2}{$-\frac{1}{2}$} & $\eTri_{\infty}$ & $\nan$ & $6$ & &  &\\
      \cline{3-7}
      & & $\ePent_{\infty}$ & $\lan$ & $2, 3$, or $4$ & $X_{4,0}$ & $1-1=0$ &\\
      \hline
    \end{tabular}
    \caption{The computation of $\deg_{\delta}(f_{\infty})$, based on the relative height of the initial generator. We get $-1/2$ in all cases.}
    \label{tab:delta}
  \end{table}

  By Table~\ref{tab:delta}, we see that $f_\infty$ is homogeneous with respect to the $\delta$-grading, and shifts it by $-1/2$. 

  Calculations for $f_0, f_1$, and also for $\phi_k$ and $\psi_k$, are completely analogous. We see that $\deg_{\delta} (f_k) = -1/2$, $\deg_{\delta} (\psi_k) = 1/2$, and $\deg_{\delta} (\psi_k) = 1$.
\end{proof}

For each $k \in \set{\infty, 0, 1}$, we now define $\ee_k$ by
\[
  \ee_k = e_k - e'_k - e_k''.
\]
This definition is consistent with our definitions of $e_k$, $e'_k$, and $e_k''$. Note that if $\eT_{\infty}$ contains a positive crossing, then $(\ee_\infty, \ee_0,\ee_1)=(0,0,0)$, and if $\eT_{\infty}$ contains a negative crossing, then $(\ee_\infty, \ee_0,\ee_1)=(-1,0,1)$. In the following, we say that $\eT_\infty$ is \emph{positive} if it contains a positive crossing, and \emph{negative} otherwise.

\begin{prop}
  \label{prop:smallgradings}
  The morphisms $f_k$, $\phi_k$, and $\psi_k$ are homogeneous with respect to the $\delta$-grading, and their $\delta$-degrees are as follows:
  \begin{align*}
    \deg_{\delta} (f_\infty) & = \frac{\ee_\infty}{2}, \quad & \deg_{\delta} (f_0) & = -\frac{1}{2}, \quad & \deg_{\delta} (f_1) & = \frac{\ee_1-1}{2};\\
    \deg_{\delta} (\phi_\infty) & = \frac{- \ee_1 + 1}{2}, \quad & \deg_{\delta} (\phi_0) & = \frac{- \ee_\infty}{2}, \quad & \deg_{\delta} (\phi_1) & = \frac{1}{2};\\
    \deg_{\delta} (\psi_\infty) & = 1, \quad & \deg_{\delta} (\psi_0) & = 1, \quad & \deg_{\delta} (\psi_1) & = 1.
  \end{align*}
\end{prop}

\begin{proof}
  Consider $f_k \colon \CDTDd (\eT_k, n) \to \CDTDd (\eT_{k+1}, n)$. Let $\eT_{k, f_k}$ and $\eT_{k+1, f_k}$ be oriented tangles such that
  \begin{enumerate}
    \item $\eT_k$ and $\eT_{k, f_k}$ are the same tangle after forgetting orientations;
    \item $\eT_{k+1}$ and $\eT_{k+1, f_k}$ are the same tangle after forgetting orientations; and
    \item $\eT_{k, f_k}$ and $\eT_{k+1, f_k}$ have compatible orientations.
  \end{enumerate}
  Then it is evident that we can factorize $f_k$ into $f_k = \eiota_{k+1, f_k} \circ \orf_k \circ \eiota_{k, f_k}$, where
  \begin{enumerate}
    \item $\eiota_{k, f_k} \colon \CDTDd (\eT_k, n) \to \CDTDd (\eT_{k, f_k}, n)$ is the map induced by the natural correspondence between generators and domains;
    \item $\orf_k \colon \CDTDd (\eT_{k, f_k}, n) \to \CDTDd (\eT_{k+1, f_k}, n)$ is the map $f_k$ described in Section~\ref{sec:proof}; and
    \item $\eiota_{k+1, f_k} \colon \CDTDd (\eT_{k+1, f_k}, n) \to \CDTDd (\eT_{k+1}, n)$ is the map induced by the natural correspondence between generators and domains.
  \end{enumerate}
  Since $\eT_{k, f_k}$ and $\eT_{k+1, f_k}$ by definition have compatible orientations, by Lemma~\ref{lem:compat}, we immediately see that $\degd (\orf_k) = -1/2$.

  Observe now that $\eT_k$, $\eT_{k, f_k}$, and $\eT_{k, f_{k-1}}$ all have no crossings when $k \in \set{0, 1}$. Therefore, by Lemma~\ref{lem:iota_gr}, we also know the $\delta$-degrees of $\eiota_{k, f_k}$ and $\eiota_{k, f_{k-1}}$:
  \[
    \degd (\eiota_{k, f_k}) = \degd (\eiota_{k, f_{k-1}}) = 0 \qquad \text{if } k \in \set{0, 1}.
  \]

  Consider now the case $k = \infty$. Observe that $\eT_{\infty, f_\infty}$ must be negative, while $\eT_{\infty, f_1}$ must be positive. Thus, Lemma~\ref{lem:iota_gr} now implies that
  \[
    \degd (\eiota_{\infty, f_\infty}) =
    \begin{cases}
      1/2 & \text{if } \eT_\infty \text{ is positive;}\\
      0 & \text{if } \eT_\infty \text{ is negative,}
    \end{cases}
    \qquad \degd (\eiota_{\infty, f_1}) =
    \begin{cases}
      0 & \text{if } \eT_\infty \text{ is positive;}\\
      1/2 & \text{if } \eT_\infty \text{ is negative.}
    \end{cases}
  \]

  Adding the $\delta$-degrees of $\eiota_{k, f_k}$, $\orf_k$, and $\eiota_{k+1, f_k}$ together, we obtain
  \[
    \brac{\degd (f_\infty), \degd (f_0), \degd (f_1)} =
    \begin{cases}
      (0, -1/2, -1/2) & \text{if } \eT_\infty \text{ is positive;}\\
      (-1/2, -1/2, 0) & \text{if } \eT_\infty \text{ is negative.}
    \end{cases}
  \]
  Comparing this with the values of $\ee_k$ indicated above the current proposition confirms the $\delta$-degrees of $f_k$.

  A similar analysis ascertains the $\delta$-degrees of $\phi_k$ and $\psi_k$.
\end{proof}

\begin{prop}
  \label{prop:biggradings}
  The morphisms $F_k$, $\Phi_k$, and $\Psi_k$ are homogeneous with respect to the $\delta$-grading, and their $\delta$-degrees are as follows:
  \begin{align*}
    \deg_{\delta} (F_\infty) & = \frac{e_\infty}{2}, \quad & \deg_{\delta} (F_0) & = \frac{e_0 - 1}{2}, \quad & \deg_{\delta} (F_1) & = \frac{e_1 - 1}{2};\\
    \deg_{\delta} (\Phi_\infty) & = \frac{- e_1 + 1}{2}, \quad & \deg_{\delta} (\Phi_0) & = \frac{- e_\infty}{2}, \quad & \deg_{\delta} (\Phi_1) & = \frac{- e_0 + 1}{2};\\
    \deg_{\delta} (\Psi_\infty) & = 1, \quad & \deg_{\delta} (\Psi_0) & = 1, \quad & \deg_{\delta} (\Psi_1) & = 1.
  \end{align*}
\end{prop}

\begin{proof}
  We have that
  \[
    \degd (F_k) = \degd (\iota_k') + \degd (f_k) + \degd (\iota_k'') = \frac{e_k'}{2} + \degd (f_k) + \frac{e_k''}{2}.
  \]
  Proposition~\ref{prop:smallgradings}, together with the equations $\ee_0 = 0$ and $e_k' + \ee_k + e_k'' = e_k$, yield $\degd (F_k)$.

  Recall that $e_\infty + e_0 + e_1 = 0$, and similar equations hold for $e_k'$, $e_k''$, and $\ee_k$. The computation of $\degd (\Phi_k)$ follows similarly, using these equations. The $\delta$-degree of $\Psi_k$ is obvious.
\end{proof}

\begin{proof}[Proof of Theorem~\ref{thm:gradings}]
  In the proof of Theorem~\ref{thm:ourtheorem}, we applied Lemma~\ref{lem:hom_alg} to the morphisms $F_k$, $\Phi_k$, and $\Psi_k$. In this section, we modified the definitions of $F_k$, $\Phi_k$, and $\Psi_k$, to adapt to the $\delta$-graded picture. The proof will be complete if we can compute the $\delta$-degrees of the relevant maps in Lemma~\ref{lem:hom_alg}, when applied to our modified $F_k$, $\Phi_k$, and $\Psi_k$.

  With $\DDm{M}_k = \CDTDd (T_k, n)$, recall that the desired homotopy equivalence is given by the type~$\DD$ homomorphisms $G_k \colon \DDm{M}_k \to \Cone (F_{k+1})$ and $G_k' \colon \Cone (F_{k+1}) \to \DDm{M}_k$, where
  \begin{align*}
    G_k (m_k) & = (F_k (m_k), \Phi_k (m_k)),\\
    G_k' (m_{k+1}, m_{k+2}) & = \Phi_{k+1} (m_{k+1}) + F_{k+2} (m_{k+2}),
  \end{align*}
  and the homotopy morphisms $H_k \colon \DDm{M}_k \to \DDm{M}_k$ and $H_k' \colon \Cone (F_{k+1}) \to \Cone (F_{k+1})$, where
  \begin{align*}
    H_k (m_k) & = \Psi_k (m_k),\\*
    H_k' (m_{k+1}, m_{k+2}) & = (\Psi_{k+1} (m_{k+1}) + \Phi_{k+2} (m_{k+2}), \Psi_{k+2} (m_{k+2})).
  \end{align*}
  We will now use Proposition~\ref{prop:biggradings} to compute the $\delta$-degrees of $G_\infty$, $G_\infty'$, $H_\infty$, and $H_\infty'$.

  First, observe that since $\degd (F_0) = (e_0 - 1)/2$, $\Cone (F_0)$ has underlying module
  \[
    M_0 \sqbrac{\frac{e_0 + 1}{2}} \oplus M_1,
  \]
  and so, in $\Cone (F_0)$,
  \begin{align*}
    \delta (F_\infty (m_\infty), 0) & = \brac{\delta (m_\infty) + \frac{e_\infty}{2}} + \frac{e_0+1}{2} = \delta (m_\infty) + \frac{-e_1 + 1}{2},\\
    \delta (0, F_\infty (m_\infty)) & = \delta (m_\infty) + \frac{-e_1 + 1}{2}.
  \end{align*}
  This shows that $G_\infty$ is homogeneous and $\degd (G_\infty) = (- e_1 + 1)/2$. Next, if $\delta (m_0, m_1) = d$, then
  \[
    \delta (\Phi_0 (m_0)) = \brac{d - \frac{e_0 + 1}{2}} + \frac{-e_\infty}{2} = d + \frac{e_1 - 1}{2}, \qquad \delta (F_1 (m_1)) = d + \frac{e_1 - 1}{2},
  \]
  and so $\degd (G_\infty') = (e_1 - 1)/2$.

  Now clearly $\degd (H_0) = 1$. Finally, if $\delta (m_0, m_1) = d$, then
  \begin{align*}
    \delta (\Psi_0 (m_0), 0) & = d + 1,\\
    \delta (\Phi_1 (m_1), 0) & = \brac{d + \frac{-e_0 + 1}{2}} + \frac{e_0 + 1}{2} = d + 1,\\
    \delta (0, \Psi_1 (m_1)) & = d + 1,
  \end{align*}
  and so $\degd (H_0') = 1$.

  Since $G_\infty$ and $G_\infty'$ are homotopy equivalences of homogeneous degree $(-e_1+1)/2$ and $(e_1-1)/2$ respectively, and $H_\infty$ and $H_\infty'$ are homotopy morphisms of homogeneous degree $1$, our proof is complete.
\end{proof}

\begin{proof}[Proof of Corollary~\ref{cor:hfk}]
  Theorem~\ref{thm:gradings} and the Gluing Theorem for tangle Floer homology~\cite[Theorem~12.4]{pv} together imply that there exists a chain map $F_0 \colon \CFKh (L_0) \otimes V^{m-\ell_0} \otimes W \to \CFKh (L_1) \otimes V^{m-\ell_1} \otimes W$ of $\delta$-degree $(e_0 - 1)/2$ such that
  \[
    \CFKh (L_\infty; \F{2}) \otimes V^{m-\ell_\infty} \otimes W \simeq \Cone (F_0) \sqbrac{\frac{e_1-1}{2}}.
  \]
  The claimed exact triangle now follows from the exact triangle on homology associated to a mapping cone of chain complexes.
\end{proof}

%%%%%%%%%%%%%%%%%%%%%%%%%%%%%%%%%%%%%%%%%%%%%%%%%%%%%%%

%%%%%%%%%%%%%%%%%%%%%%%%%%%%%%%%%%%%%%%%%%%%%%%%%%%%%%%
% !TEX root = ../skein.tex
%%%%%%%%%%%%%%%%%%%%%%%%%%%%%%%%%%%%%%%%%%%%%%%%%%%%%%%

\section{The oriented skein relation} % (fold)
\label{sec:oriented}

%%%%%%%%%%%%%%%%%%%%%%%%%%%%%%%%%%%%%%%%%%%%%%%%%%%%%%%

%%%%%%%%%%%%%%%%%%%%%%%%%%%%%%%%%%%%%%%%%%%%%%%%%%%%%%%
% section oriented 
%%%%%%%%%%%%%%%%%%%%%%%%%%%%%%%%%%%%%%%%%%%%%%%%%%%%%%%

In this section, we give the proof of Theorem~\ref{thm:oriented}, which is similar to that of \cite[Theorem~9.2.1]{OSSbook}.

Let $(\eT_+, \eT_-, \eT_0)$ be an oriented skein triple of elementary $(n,n)$-tangles, with the strands at which the tangles differ oriented from right to left, as in Figure~\ref{fig:Tel_ori_1}.

\begin{figure}[h]
  \centering
  \includegraphics[scale=1.05]{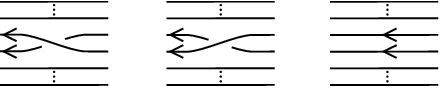}
  \caption{From left to right, the elementary tangles $\eT_+$, $\eT_-$, $\eT_0$.}
  \label{fig:Tel_ori_1}
\end{figure}

We will now describe a common diagram from which we can obtain corresponding Heegaard diagrams for these tangles. Consider Figure~\ref{fig:hd_skein_ori}. Letting
\begin{align*}
  \betas_+ & = \set{\beta_1, \dotsc, \beta_{i-1}, \beta_{i,+}, \beta_{i+1}, \dotsc, \beta_n},\\
  \betas_- & = \set{\beta_1, \dotsc, \beta_{i-1}, \beta_{i,-}, \beta_{i+1}, \dotsc, \beta_n},
\end{align*}
and
\[
  \XX = \set{X_1, \dotsc, X_n}, \quad \XX' = \set{X_1', X_2, \dotsc, X_n}, \quad \YY = \set{Y_1, Y_2, X_3, \dotsc, X_n},
\]
we define
\begin{align*}
  \HD_+ & = (\Sigma, \alphas, \betas_+, \XX, \OO), & \HD_- & = (\Sigma, \alphas, \betas_-, \XX', \OO),\\
  \HD_0 & = (\Sigma, \alphas, \betas_+, \YY, \OO), & \HD_0' & = (\Sigma, \alphas, \betas_-, \YY, \OO).
\end{align*}
Then $\HD_+$ is a Heegaard diagram for $\eT_+$, $\HD_-$ is a diagram for $\eT_-$, and both $\HD_0$ and $\HD_0'$, which are related by a commutation move~\cite[Section~5.3.1]{pv}, are diagrams for $\eT_0$.

\begin{figure}[h]
  \centering
  \labellist
  \pinlabel  $X_1$ at 101 67
  \pinlabel  $Y_1$ at 112 69
  \pinlabel  $X_1'$ at 121 70
  \pinlabel  $X_2$ at 112 60
  \pinlabel  $Y_2$ at 111 51
  \pinlabel  $O_3$ at 45 68
  \pinlabel  $O_4$ at 46 54
  \pinlabel  $c$ at 103 60
  \pinlabel  $c'$ at 120 62
  \pinlabel  \textcolor{red}{$c^{\mathit{FR}}_0$} at 148 32
  \pinlabel  \textcolor{red}{$c^{\mathit{FR}}_{n}$} at 147 89
  \pinlabel  \textcolor{red}{$c^{\mathit{FL}}_0$} at 10 35
  \pinlabel  \textcolor{red}{$c^{\mathit{FL}}_{n}$} at 9 86
  \pinlabel  \textcolor{red}{$c^{\mathit{BL}}_0$} at -5 37
  \pinlabel  \textcolor{red}{$c^{\mathit{BL}}_{n}$} at -5 85
  \pinlabel \rotatebox{90}{\textcolor{red}{$\dots$}} at -5 61
  \pinlabel \rotatebox{90}{\textcolor{red}{$\dots$}} at 12 61
  \pinlabel \rotatebox{90}{\textcolor{red}{$\dots$}} at 148 62
  \pinlabel \textcolor{blue!60!}{$\beta_{n}$} at 73 47
  \pinlabel \textcolor{blue}{$\beta_{i, +}$} at 66 21
  \pinlabel \textcolor{OliveGreen}{$\beta_{i, -}$} at 69 31
  \pinlabel \textcolor{blue!60!}{$\beta_{0}$} at 60 5
  \pinlabel \rotatebox{74}{\textcolor{blue!60!}{$\dots$}} at 61 13
  \endlabellist
  \includegraphics[scale=2]{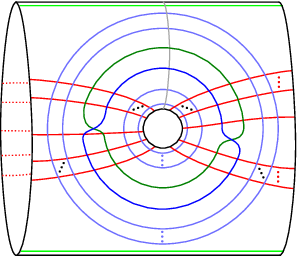}
  \caption{The diagram that combines $\HD_+, \HD_-, \HD_0, \HD_0'$.}
  \label{fig:hd_skein_ori}
\end{figure}

We denote the variables in $\CDTDm$ corresponding to $O_3$ and $O_4$ by $U_3$ and $U_4$, respectively. Observe that $\eT_+, \eT_-$, and $\eT_0$ all have the same left and the same right oriented boundaries, which we may denote by $\bdyL \T$ and $\bdyR \T$ respectively. Denote the variables in $\am{-\bdy^R\T}$ corresponding to the $i^{\text{th}}$ and $(i+1)^{\text{st}}$ point in $-\bdy^R\T$ by $U_2$ and $U_1$, respectively.

As before, we cut open the Heegaard diagram along the indicated grey circle in 
Figure~\ref{fig:hd_skein_ori} and also delete the non-combinatorial regions, to 
obtain Figure~\ref{fig:hd_cut_ori}.
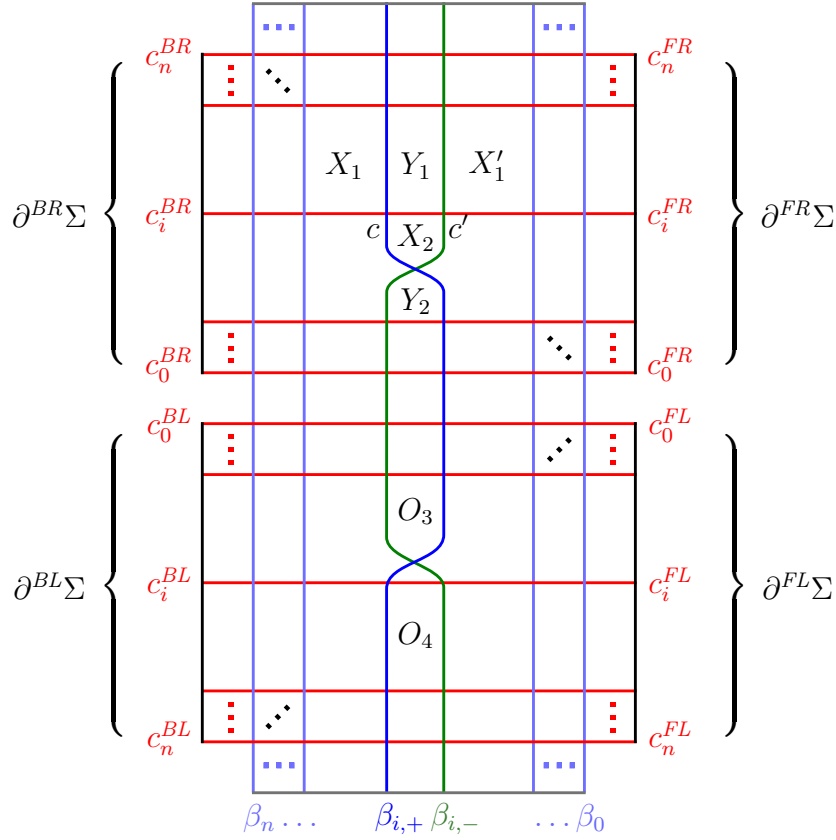
\begin{figure}[h]
  \centering
  \labellist
  \pinlabel  $X_1$ at 90 395
  \pinlabel  $Y_1$ at 135 395
  \pinlabel  $X_1'$ at 180 395
  \pinlabel  $X_2$ at 135 350
  \pinlabel  $Y_2$ at 135 310
  \pinlabel  $O_3$ at 135 178
  \pinlabel  $O_4$ at 135 100
  \pinlabel  $c$ at 108 354
  \pinlabel  $c'$ at 162 357
  \pinlabel  \textcolor{red}{$c^{\mathit{FR}}_0$} at 295 270
  \pinlabel  \textcolor{red}{$c^{\mathit{FR}}_i$} at 295 364
  \pinlabel  \textcolor{red}{$c^{\mathit{FR}}_{n}$} at 295 464
  \pinlabel  \textcolor{red}{$c^{\mathit{FL}}_0$} at 295 232
  \pinlabel  \textcolor{red}{$c^{\mathit{FL}}_i$} at 295 132
  \pinlabel  \textcolor{red}{$c^{\mathit{FL}}_{n}$} at 295 35
  \pinlabel  \textcolor{red}{$c^{\mathit{BR}}_0$} at -20 270
  \pinlabel  \textcolor{red}{$c^{\mathit{BR}}_i$} at -20 364
  \pinlabel  \textcolor{red}{$c^{\mathit{BR}}_{n}$} at -20 464
  \pinlabel  \textcolor{red}{$c^{\mathit{BL}}_0$} at -20 232
  \pinlabel  \textcolor{red}{$c^{\mathit{BL}}_i$} at -20 132
  \pinlabel  \textcolor{red}{$c^{\mathit{BL}}_{n}$} at -20 35
  \pinlabel \rotatebox{90}{$\underbrace{\hspace{4cm}}$} at 335 131
  \pinlabel $\bdy^{\mathit{FL}}\Sigma$ at 375 131
  \pinlabel \rotatebox{90}{$\underbrace{\hspace{4cm}}$} at 335 365
  \pinlabel $\bdy^{\mathit{FR}}\Sigma$ at 375 365
  \pinlabel \rotatebox{90}{$\overbrace{\hspace{4cm}}$} at -55 131
  \pinlabel $\bdy^{\mathit{BL}}\Sigma$ at -95 131
  \pinlabel \rotatebox{90}{$\overbrace{\hspace{4cm}}$} at -55 365
  \pinlabel $\bdy^{\mathit{BR}}\Sigma$ at -95 365
  \pinlabel \textcolor{blue!60!}{$\beta_{n}\ldots$} at 50 -15
  \pinlabel \textcolor{blue}{$\beta_{i, +}$} at 125 -15
  \pinlabel \textcolor{OliveGreen}{$\beta_{i, -}$} at 160 -15
  \pinlabel \textcolor{blue!60!}{$\ldots\beta_{0}$} at 232 -15
  \endlabellist
  \includegraphics[scale = .6]{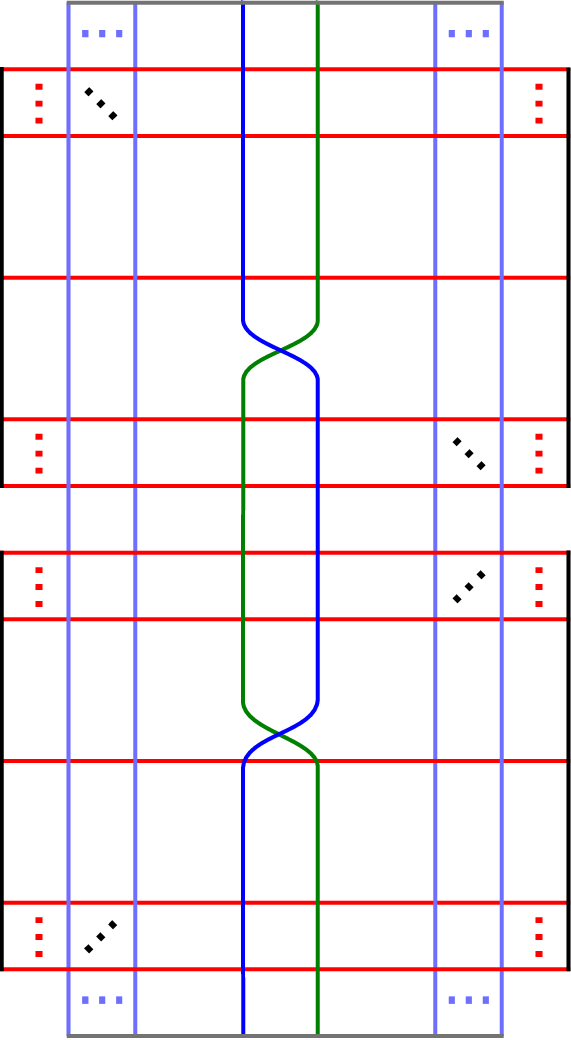}
  \vskip .5 cm
  \caption{The combined diagram for the four elementary tangles, obtained by cutting open the diagram in Figure~\ref{fig:hd_skein_ori} along the indicated grey circle and deleting the non-combinatorial regions.}
  \label{fig:hd_cut_ori}
\end{figure}

We denote by $c$ the intersection point in $\alpha_i^R \cap \beta_{i,+}$; note, then, that the two squares containing $X_1$ and $X_2$ in $\HD_+$ meet at $c$, and so do the squares containing $Y_1$ and $Y_2$ in $\HD_0$. Similarly, we denote by $c'$ the intersection point in $\alpha_i^R \cap \beta_{i,-}$; then the two squares containing $X_1'$ and $X_2$ in $\HD_0'$ meet at $c'$, and so do the squares containing $Y_1$ and $Y_2$ in $\HD_0'$.

We now partition the set of generators
\[
  \SS (\HD_+) = \II (\HD_+) \cup \NN (\HD_+), \quad \SS (\HD_0) = \II (\HD_0) \cup \NN (\HD_0)
\]
according to whether or not a given generator contains the point $c$, and similarly partition
\[
  \SS (\HD_-) = \II' (\HD_-) \cup \NN' (\HD_-), \quad \SS (\HD_0') = \II' (\HD_0') \cup \NN' (\HD_0')
\]
according to whether or not a generator contains the point $c'$.

There is a natural identification between $\II (\HD_+)$ and $\II (\HD_0)$, and one between $\NN (\HD_+)$ and $\NN (\HD_0)$, as sets. Similarly, there is a natural identification between $\II' (\HD_-)$ and $\II' (\HD_0')$, and one between $\NN' (\HD_-)$ and $\NN' (\HD_0')$. However, note that, for example, $\NN (\HD_+)$ and $\NN (\HD_0)$ have different $(M, A)$-bigradings that depend on $\HD_+$ and $\HD_0$. We will think of the Maslov grading $M$ as given by functions $M_+$, $M_-$, $M_0$, and $M_0'$, and similarly for the Alexander grading $A$.

Further, let $\TT \colon \II' (\HD_0') \to \II (\HD_+)$ be the unique one-to-one correspondence for which $\x \cap \beta_{i,-}$ and $T(\x)\cap \beta_{i,+}$ lie on the same $\alpha$ curve, and $\x \setminus \beta_{i,-} = T(\x)\setminus \beta_{i,+}$. Note that $\TT$ can also be written as a function $\TT \colon \II' (\HD_-) \to \II (\HD_0)$. We can think of $\TT$ as a map counting the small triangle that contains $X_2$.

\begin{lem}
  \label{lem:gr-gens}
  For $\x\in \II(\HD_+) = \II(\HD_0)$,
  \[
    M_+(\x) = M_0(\x), \qquad A_+(\x) = A_0(\x) - \frac 1 2;
  \]
  for $\x\in \NN(\HD_+) = \NN(\HD_0)$,
  \[
    M_+(\x) = M_0(\x), \qquad A_+(\x) = A_0(\x) + \frac 1 2;
  \]
  for $\x\in \II'(\HD_-) = \II'(\HD_0')$,
  \[
    M_-(\x) = M_0'(\x) \qquad A_-(\x) = A_0'(\x)+\frac 1 2;
  \]
  for $\x\in \NN'(\HD_-) = \NN'(\HD_0')$,
  \[
    M_-(\x) = M_0'(\x) \qquad A_-(\x) = A_0'(\x)-\frac 1 2.
  \]
  Furthermore, for $\x\in \II'(\HD_0')$ and $\TT(\x)\in \II(\HD_+)$,
  \[
    M_0'(\x) = M_+(\TT(\x))+1 \qquad A_0'(\x) = A_+(\TT(\x))+\frac 1 2,
  \]
  and for $\x\in \II'(\HD_-)$ and $\TT(\x)\in \II(\HD_0)$,
  \[
    M_-(\x) = M_0(\TT(\x))+1 \qquad A_-(\x) = A_0(\TT(\x))+\frac 1 2.
  \]
\end{lem}

\begin{proof}
  Let $\x\in \II(\HD_+) = \II(\HD_0)$. Both $\HD_+$ and $\HD_0$ use $\betas_+$, with respect to which we have $\inv(\x^R, \XX^R) = \inv(\x^R, \YY^R)-2$, and $\inv (\XX^R) = \inv(\YY^R)-1$, whereas all other counts in the definition of the bigrading agree for the two diagrams. Thus,  $M_+(\x) = M_0(\x)$ and $A_+(\x) = A_0(\x) - \frac 1 2$. One obtains the second, third, and fourth statements of the lemma similarly. 

  Let $\x\in \II'(\HD_0')$. We will denote inversions counted with respect to $\betas_+$ or $\betas_-$ by $\inv_{\betas_+}$ or $\inv_{\betas_-}$, respectively. Note that $\inv_{\betas_-}(\YY^L) = \inv_{\betas_+}(\XX^L)-1$. Since $\beta_{i,-}\cap \alpha_i^L\notin \x$ and  $\beta_{i,+}\cap \alpha_i^L\notin \TT(\x)$, we have $\inv_{\betas_-}(\x^L, \YY^L) = \inv(\x^L, \XX^L)$. All other terms in the definition of the bigrading agree for the two diagrams. Thus,  $M_0'(\x) = M_+(\TT(\x))+1$ and $A_0'(\x) = A_+(\TT(\x))+\frac 1 2$. One obtains the sixth statement of the lemma similarly.  \end{proof}

We define bigraded type~$\DD$ structures $(\II, \delta_{\II, \II}^1)$ and $(\NN, \delta_{\NN, \NN}^1)$ over $(\alg{A}^- (-\bdyL \T), \alg{A}^- (-\bdyR \T))$ as follows.  As a module, $\II$ is freely generated over $\F{2} [U_1, \dotsc, U_n]$ by the set $\II (\HD_+) = \II (\HD_0)$, and $\NN$ by $\NN (\HD_+) = \NN (\HD_0)$.  The structure map $\delta_{\II, \II}^1$ counts rectangles not crossing $Y_1$ or $Y_2$, and $\delta_{\NN, \NN}^1$ counts rectangles not crossing $X_1$ or $X_2$. The bigradings on $\II$ and $\NN$ are given by
\begin{align*}
  M_{\II} (\x) & = M_+ (\x) + 1 = M_0 (\x) + 1, & A_{\II} (\x) & = A_+ (\x) + 1 = A_0 (\x) + \frac{1}{2},\\
  M_{\NN} (\x) & = M_+ (\x) + 1 = M_0 (\x) + 1, & A_{\NN} (\x) & = A_+ (\x) = A_0 (\x) + \frac{1}{2},
\end{align*}
which are well defined in light of Lemma~\ref{lem:gr-gens}. We may think of $\CDTDm (\HD_+)$ as the mapping cone of a type~$\DD$ homomorphism $\delta_{\II, \NN}^1 \colon (\II, \delta_{\II, \II}^1) \to (\NN, \delta_{\NN, \NN}^1)$ that counts rectangles crossing exactly one of $Y_1$ and $Y_2$, and think of $\CDTDm (\HD_0)$ as the mapping cone of $\delta_{\NN, \II}^1 \colon (\NN, \delta_{\NN, \NN}^1) \to (\II, \delta_{\II, \II}^1)$ that counts rectangles crossing exactly one of $X_1$ and $X_2$.
In other words,
\begin{align*}
  \CDTDm (\HD_+) & = \Cone (\delta_{\II, \NN}^1) [-1] \cbrac{0}, & \deg (\delta_{\II, \NN}^1) & = (-1, -1),\\
  \CDTDm (\HD_0) & = \Cone (\delta_{\NN, \II}^1) [-1] \cbrac{-\frac{1}{2}}, & \deg (\delta_{\NN, \II}^1) & = (-1, 0).
\end{align*}
Similarly, we may define type~$\DD$ structures $(\II', \delta_{\II', \II'}^1)$ and $(\NN', \delta_{\NN', \NN'}^1)$ over the same algebras: As a module, $\II'$ is freely generated by the set $\II' (\HD_-) = \II' (\HD_0')$, and $\NN'$ by $\NN' (\HD_-) = \NN' (\HD_0')$. The structure map $\delta_{\II', \II'}^1$ counts rectangles not crossing $X_1'$ or $X_2$, and $\delta_{\NN', \NN'}^1$ counts rectangles not crossing $Y_1$ or $Y_2$. The bigradings on $\II'$ and $\NN'$ are given by
\begin{align*}
  M_{\II'} (\x) & = M_- (\x) = M_0' (\x), & A_{\II'} (\x) & = A_- (\x) - 1 = A_0' (\x) - \frac{1}{2},\\
  M_{\NN'} (\x) & = M_- (\x) = M_0' (\x), & A_{\NN'} (\x) & = A_- (\x) = A_0' (\x) - \frac{1}{2},
\end{align*}
which are well defined by Lemma~\ref{lem:gr-gens}. We may think of $\CDTDm (\HD_-)$ as the mapping cone of a type~$\DD$ homomorphism $\delta_{\NN', \II'}^1 \colon (\NN', \delta_{\NN', \NN'}^1) \to (\II', \delta_{\II', \II'}^1)$ that counts rectangles crossing exactly one of $Y_1$ and $Y_2$, and think of $\CDTDm (\HD_0')$ as the mapping cone of $\delta_{\II', \NN'}^1 \colon (\II', \delta_{\II', \II'}^1) \to (\NN', \delta_{\NN', \NN'}^1)$ that counts rectangles crossing exactly one of $X_1'$ and $X_2$. In other words,
\begin{align*}
  \CDTDm (\HD_-) & = \Cone (\delta_{\NN', \II'}^1) [0] \cbrac{1}, & \deg (\delta_{\NN', \II'}^1) & = (-1, -1),\\
  \CDTDm (\HD_0') & = \Cone (\delta_{\II', \NN'}^1) [0] \cbrac{\frac{1}{2}}, & \deg (\delta_{\II', \NN'}^1) & = (-1, 0).
\end{align*}

\begin{rmk}
  Our notation for grading shifts (for modules) differs from that in \cite{OSSbook} by a negative sign.
\end{rmk}

\begin{lem}
  \label{lem:T_isom}
  The correspondence $\TT \colon \II' (\HD_0') \to \II (\HD_+)$ extends to a type~$\DD$ isomorphism $\TT \colon (\II', \delta_{\II', \II'}^1) \to (\II, \delta_{\II, \II}^1)$ of $(M, A)$-degree $(0, 1)$.
\end{lem}

\begin{proof}
  Refer to Figure~\ref{fig:hd_cut_ori}. Recall that the structure map $\delta_{\II, \II}^1$ counts rectangles (of which there are seven types), introduced in Section~\ref{ssec:tf}, that do not cross $Y_1$ or $Y_2$. Since generators in $\II$ all have a component on $c$, this means that the rectangles counted in $\delta_{\II, \II}^1$ do not have an edge on $\beta_{i,+}$. Note that if such a rectangle were to cross $X_1$, then it would necessarily cross $Y_1$ (and hence $X_1'$) also; and if it were to cross $X_2$, then it would cross $Y_2$ also. This implies that the rectangles counted in $\delta_{\II, \II}^1$ do not intersect $\set{X_1, X_1', X_2, Y_1, Y_2}$. Similarly, rectangles counted in $\delta_{\II', \II'}^1$ do not intersect this set. In other words, $\delta_{\II, \II}^1$ and $\delta_{\II', \II'}^1$ both count empty rectangles that do not intersect the $2$-chain whose boundary is formed by an arc in $\alpha_{i-1}^R$, an arc in $\beta_{i-1}$, an arc in $\alpha_{i+1}^R$, and an arc in $\beta_{i+1}$, in the induced orientation. We can now conclude that for $\x', \y' \in \II'$, there is a one-to-one correspondence between rectangles connecting $\x'$ to $\y'$ and rectangles connecting $\TT (\x')$ to $\TT (\y')$.

  The fact that these rectangles do not have an edge on $\beta_{i,+}$ or $\beta_{i,-}$ also implies that, for $\x', \y' \in \II'$, an empty rectangle $r'$ connecting $\x'$ to $\y'$ crosses $O_3$ (resp.\ $O_4$) if and only if the corresponding rectangle $r$ connecting $\TT (\x')$ to $\TT (\y')$ crosses $O_3$ (resp.\ $O_4$). This shows that the algebra elements $U^{r'}$ and $U^r$ are equal. Since obviously $\algl{\x',r'} = \algl{\TT (\x'),r}$ and $\algr{\x',r'} = \algr{\TT (\x'),r}$, the term
  \[
    \algl{\x',r'} \otimes U^{r'} \y' \otimes \algr{\x',r}
  \]
  appears in $\delta_{\II',\II'}^1 (\x')$ if and only if
  \[
    \algl{\TT (\x'),r} \otimes U^r \TT (\y') \otimes \algr{\TT (\x'),r} = \algl{\x',r'} \otimes U^{r'} \TT (\y') \otimes \algr{\x',r'}
  \]
  appears in $\delta_{\II,\II}^1 (\TT (\x'))$. This shows that $\TT$ is a type~$\DD$ isomorphism.

  The $(M, A)$-degree of $\TT$ follows from Lemma~\ref{lem:gr-gens} and from the definitions of $\II$ and $\II'$.
\end{proof}

We now consider the following diagram of type~$\DD$ homomorphisms:
\begin{equation}
  \label{eq:comm_diag}
  \renewcommand{\labelstyle}{\textstyle}
  \mathcenter{
    \xymatrix{
      (\II', \delta_{\II', \II'}^1) \ar[rr]^{\delta_{\II', \NN'}^1} \ar[dd]_{\delta_{\II, \NN} \circ \TT} & & (\NN', \delta_{\NN', \NN'}^1) \ar[dd]^{\TT \circ \delta_{\NN', \II'}}\\
      & & \\
      (\NN, \delta_{\NN, \NN}^1) \ar[rr]^{\delta_{\NN, \II}^1} & & (\II, \delta_{\II, \II}^1)
    }
  }
\end{equation}
Lemma~\ref{lem:T_isom} and the discussion immediately preceding it imply that:
\begin{itemize}
  \item Each edge map in \eqref{eq:comm_diag} has $(M, A)$-degree $(-1, 0)$.
  \item The left column in \eqref{eq:comm_diag} is identified with $\CDTDm (\HD_+) [1] \cbrac{0}$.
  \item The right column in \eqref{eq:comm_diag} is identified with $\CDTDm (\HD_-) [0] \cbrac{0}$.
  \item The top row in \eqref{eq:comm_diag} is identified with $\CDTDm (\HD_0') [0] \cbrac{-1/2}$.
  \item The bottom row in \eqref{eq:comm_diag} is identified with $\CDTDm (\HD_0) [1] \cbrac{1/2}$.
\end{itemize}

\begin{lem}
  \label{lem:comm_diag}
  The type~$\DD$ homomorphisms
  \[
    \delta_{\NN, \II}^1 \circ \delta_{\II, \NN}^1 \colon (\II, \delta_{\II, \II}^1) \to (\II, \delta_{\II, \II}^1), \qquad
    \delta_{\NN', \II'}^1 \circ \delta_{\II', \NN'}^1 \colon (\II', \delta_{\II', \II'}^1) \to (\II', \delta_{\II', \II'}^1)
  \]
  are given by
  \begin{align*}
    \delta_{\NN, \II}^1 \circ \delta_{\II, \NN}^1 & = \Id_{\II} \otimes (U_1 + U_2 - U_3 - U_4),\\
    \delta_{\NN', \II'}^1 \circ \delta_{\II', \NN'}^1 & = \Id_{\II'} \otimes (U_1 + U_2 - U_3 - U_4).
  \end{align*}
  As a consequence, \eqref{eq:comm_diag} commutes.
\end{lem}

\begin{proof}
  Consider the homomorphism $\delta_{\NN, \II}^1 \circ \delta_{\II, \NN}^1$. This counts domains that can be written as a juxtaposition of two empty rectangles, the first of which connects some $\x \in \II$ to some $\y \in \NN$, while the second connects $\y$ to some $\z \in \II$. Inspecting Figure~\ref{fig:hd_cut_ori}, we see that for a given $\x \in \II$, there are exactly four such domains: the vertical annulus bounded by $\beta_{i-1}$ and $\beta_{i,+}$, the vertical annulus bounded by $\beta_{i,+}$ and $\beta_{i+1}$, the horizontal strip bounded by $\alpha_{i-1}$ and $\alpha_i$, and the horizontal strip bounded by $\alpha_i$ and $\alpha_{i+1}$. These four domains contribute the algebra elements $U_4$, $U_3$, $U_2$, and $U_1$ respectively.

  The homomorphism $\delta_{\NN', \II'}^1 \circ \delta_{\II', \NN'}^1$ can be computed in the same way.
\end{proof}

The commutativity of \eqref{eq:comm_diag} implies the existence of a type~$\DD$ homomorphism
\[
  (\delta_{\II, \NN}^1 \circ \TT, \TT \circ \delta_{\NN', \II'}^1) \colon \Cone (\delta_{\II', \NN'}^1) \to \Cone (\delta_{\NN, \II}^1),
\]
which we may think of as
\[
  (\delta_{\II, \NN}^1 \circ \TT, \TT \circ \delta_{\NN', \II'}^1) \colon \CDTDm (\HD_0') [0] \cbrac{-\frac{1}{2}} \to \CDTDm (\HD_0) [1] \cbrac{\frac{1}{2}},
\]
with $(M, A)$-degree $(-1, 0)$. We now work to compute this type~$\DD$ homomorphism.

Let $h_{X_2} \colon \CDTDm (\HD_0') \to \CDTDm (\HD_0')$ be the morphism defined by
\[
  h_{X_2} (\x) = \sum_{y \in \genset (\HD_0')} \sum_{\substack{r \in \eRect (\x, \y) \\ r \cap (\XX \cup \YY) = \set{X_2}}} \algl{r} \otimes U^r \y \otimes \algr{r},
\]
and for $i = 1, 2$, let $h_{Y_i} \colon \CDTDm (\HD_0') \to \CDTDm (\HD_0')$ be the morphism defined by
\[
  h_{Y_i} (\x) = \sum_{y \in \genset (\HD_0')} \sum_{\substack{r \in \eRect (\x, \y) \\ r \cap \YY = \set{Y_i}}} \algl{r} \otimes U^r \y \otimes \algr{r}.
\]
Let $h_Y = h_{Y_1} + h_{Y_2}$.

\begin{lem}
  \label{lem:htpy_deg}
  The morphism $h_{X_2}$ is homogeneous of $(M, A)$-degree $(-1, 0)$, while the morphisms $h_{Y_1}$ and $h_{Y_2}$ are both homogeneous of $(M, A)$-degree $(-1, -1)$. Moreover, identifying $\CDTDm (\HD_0')$ with $\Cone (\delta_{\II', \NN'}^1) [0] \cbrac{1/2}$ using \eqref{eq:comm_diag},
  $h_{X_2}$ sends $\II'[0]\cbrac{1/2}$ to $\NN'[0]\cbrac{1/2}$ and vanishes on $\NN'[0]\cbrac{1/2}$, while $h_{Y_i}$ vanishes on $\II'[0]\cbrac{1/2}$.
\end{lem}

\begin{proof}
  The $(M, A)$-degree of $h_{X_2}$ follows from the fact that the rectangles counted in $h_{X_2}$ all contribute to $\delta_{\HD_0'}^1$, since $\HD_0' = (\Sigma, \alphas, \betas_{-}, \YY, \OO)$. The $(M, A)$-degree of $h_{Y_i}$ can be computed by adapting the proof of Lemma~\ref{lem:DDgr} to the case when the rectangle $r$ contains exactly one basepoint of type $X$ (in the present context, $Y_i$). The images of the morphisms (when restricted to $\II'[0]\cbrac{1/2}$ or $\NN'[0]\cbrac{1/2}$) are clear from the local picture near $c'$.
\end{proof}

Lemma~\ref{lem:htpy_deg} implies that $h_{X_2} \circ h_Y + h_Y \circ h_{X_2} \colon \CDTDm (\HD_0') \to \CDTDm (\HD_0')$ has $(M, A)$-degree $(-2, -1)$. To align with the degree shifts of other morphisms defined above, and to simplify notation, we define the morphism
\[
  h \colon \CDTDm (\HD_0') [0] \cbrac{-\frac{1}{2}} \to \CDTDm (\HD_0') [1] \cbrac{\frac{1}{2}}
\]
by
\[
  h_{X_2} \circ h_Y + h_Y \circ h_{X_2},
\]
with degrees appropriately shifted, so that $h$ has $(M, A)$-degree $(-1, 0)$.

\begin{lem}
  \label{lem:vert_is_htpy}
  There is a homotopy equivalence $\PP \colon \CDTDm (\HD_0) \to \CDTDm  (\HD_0')$ such that
  \begin{equation}
    \label{eq:vert_is_htpy}
    \PP [1] \cbrac{\frac{1}{2}} \circ (\delta_{\II, \NN}^1 \circ \TT, \TT \circ \delta_{\NN', \II'}^1) = h,
  \end{equation}
  so $h$, and consequently $h_{X_2} \circ h_Y + h_Y \circ h_{X_2}$, are type~$\DD$ homomorphisms. Furthermore, the homotopy equivalence $\PP [1] \cbrac{1/2}$ induces a homotopy equivalence
  \[
    \Cone (\delta_{\II, \NN}^1 \circ \TT, \TT \circ \delta_{\NN', \II'}^1) \to \Cone (h) = \Cone (h_{X_2} \circ h_Y + h_Y \circ h_{X_2}) [1] \cbrac{\frac{1}{2}}.
  \]
\end{lem}

\begin{proof}
  This proof is analogous to the proof of \cite[Lemma~9.2.7]{OSSbook}.

  The homotopy equivalence $\PP$ is obtained in a manner analogous to that in \cite[Lemma~5.7]{pv}. It is defined by counting pentagons, with the only difference being that here we are modifying $\beta$-circles rather than $\alpha$-circles, and there is interaction with the algebra, as  in the proof of Equation~\ref{eqn:pent_homo}. (However, note that, contrary to the homomorphisms $\PP_k$ in Section~\ref{sec:proof}, $\PP$ here counts pentagons that possibly contain two components, as in rectangles of type~\eqref{case:delta_6} in the definition of $\delta^1_{\CDTDm (\HD)}$. All pentagons counted in $\PP$ are either interior or right-bordered.)

  It can easily be checked that the $(M, A)$-degrees of the maps involved are correct. Thus, to streamline our discussion, we will ignore all gradings for the rest of this proof; doing so, in order to prove Equation~\ref{eq:vert_is_htpy}, it will be sufficient to verify the identities
  \begin{gather}
    \PP \circ \TT \circ \delta_{\NN', \II'}^1 = h_{X_2} \circ h_Y, \label{eq:htpy1}\\
    \PP \circ \delta_{\II, \NN}^1 \circ \TT = h_Y \circ h_{X_2} \label{eq:htpy2},
  \end{gather}
  considered as morphisms of \emph{ungraded} type~$\DD$ modules, by exhibiting a one-to-one correspondence between domains contributing to both sides. All cases of the correspondences are analogous to those in the proof of \cite[Lemma~9.2.7]{OSSbook}; in our context, however, each case contains a number of subcases, as the domain under consideration may be bordered and interact non-trivially with the algebras.

  We provide a sample proof of Equation~\ref{eq:htpy1}. Consider a domain contributing to $\PP \circ \TT \circ \delta_{\NN', \II'}^1$; let $r_1, p_1$, and $p_2$ respectively be the rectangle contributing to $\delta_{\NN', \II'}^1$, the triangle contributing to $\TT$, and the pentagon contributing to $\PP$. Then $p_1 * p_2$ is a rectangle $r_2$ that contains $X_2$ but not $Y_1$ or $Y_2$. The juxtaposition $r_1 * r_2$ then represents a term in $h_{X_2} \circ h_Y$. Conversely, if $r_1 * r_2$ contributes to $h_{X_2} \circ h_Y$, we may remove the small triangle $p_1$ containing $X_2$ from $r_2$ to obtain a pentagon $p_2$. Since $r_1 * r_2 = r_1 * p_1 * p_2$, we see that $\algr{r_1 * r_2} = \algr{r_1 * p_2 * p_2}$. All subcases are exhibited in Figures~\ref{fig:hx2-hy-1}, \ref{fig:hx2-hy-2}, \ref{fig:hx2-hy-3}, and \ref{fig:hx2-hy-4}, each corresponding to a row in the left of \cite[Figure~9.4]{OSSbook}. In these four figures, each pair of domains is organized as follows: The left represents a domain contributing to $h_{X_2} \circ h_Y$, and the right the corresponding domain contributing to $\PP \circ \TT \circ \delta_{\NN', \II'}^1$; the shading of the polygons, from dark to light, indicates the order of composition.

  \begin{figure}
    \centering
    \includegraphics[scale = .2]{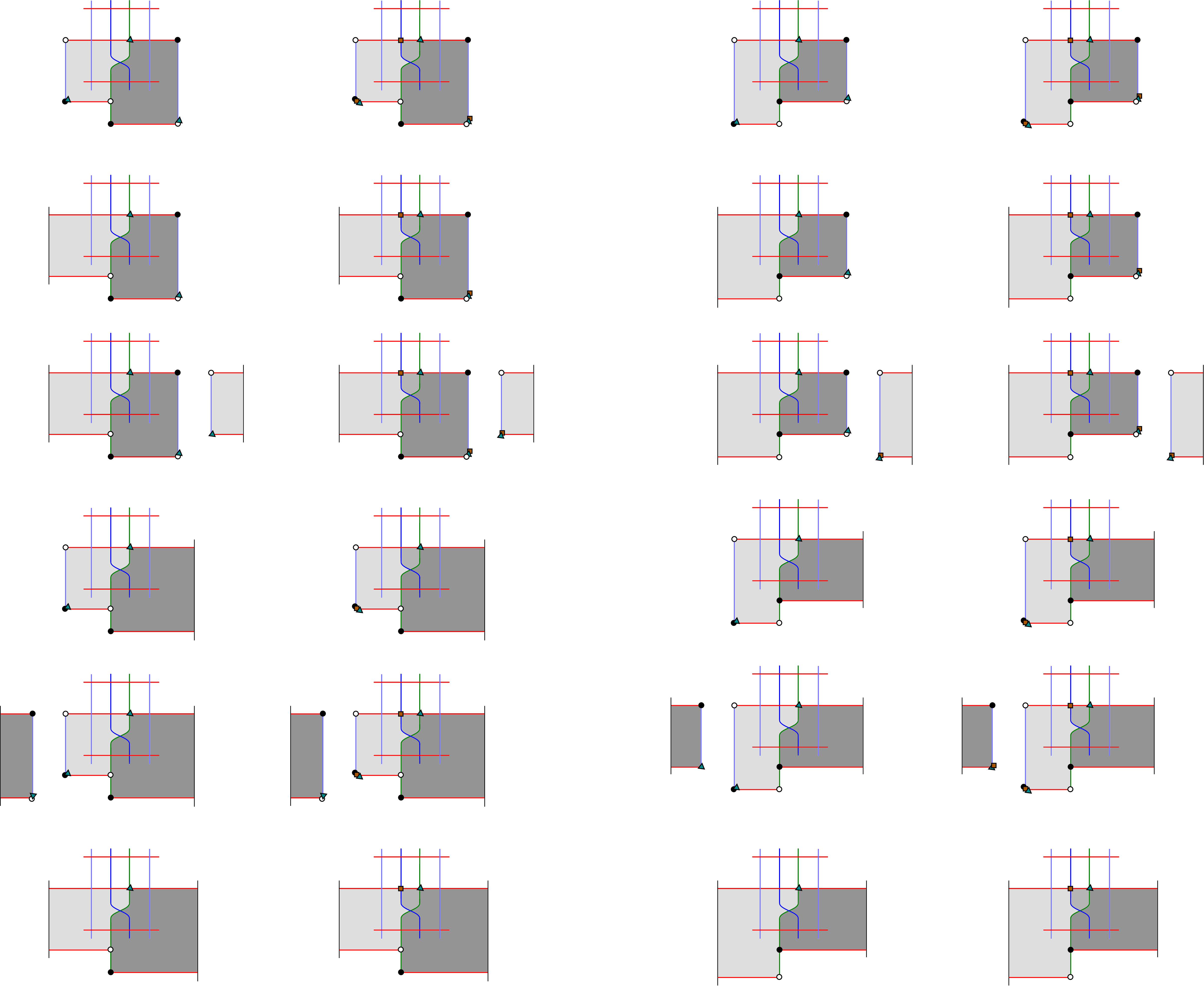}
    \caption{Corresponding domains that contribute to the two sides of Equation~\ref{eq:htpy1}, analogous to the first row, left, of \cite[Figure~9.4]{OSSbook}.}
    \label{fig:hx2-hy-1}
  \end{figure}

  \begin{figure}
    \centering
    \includegraphics[scale = .2]{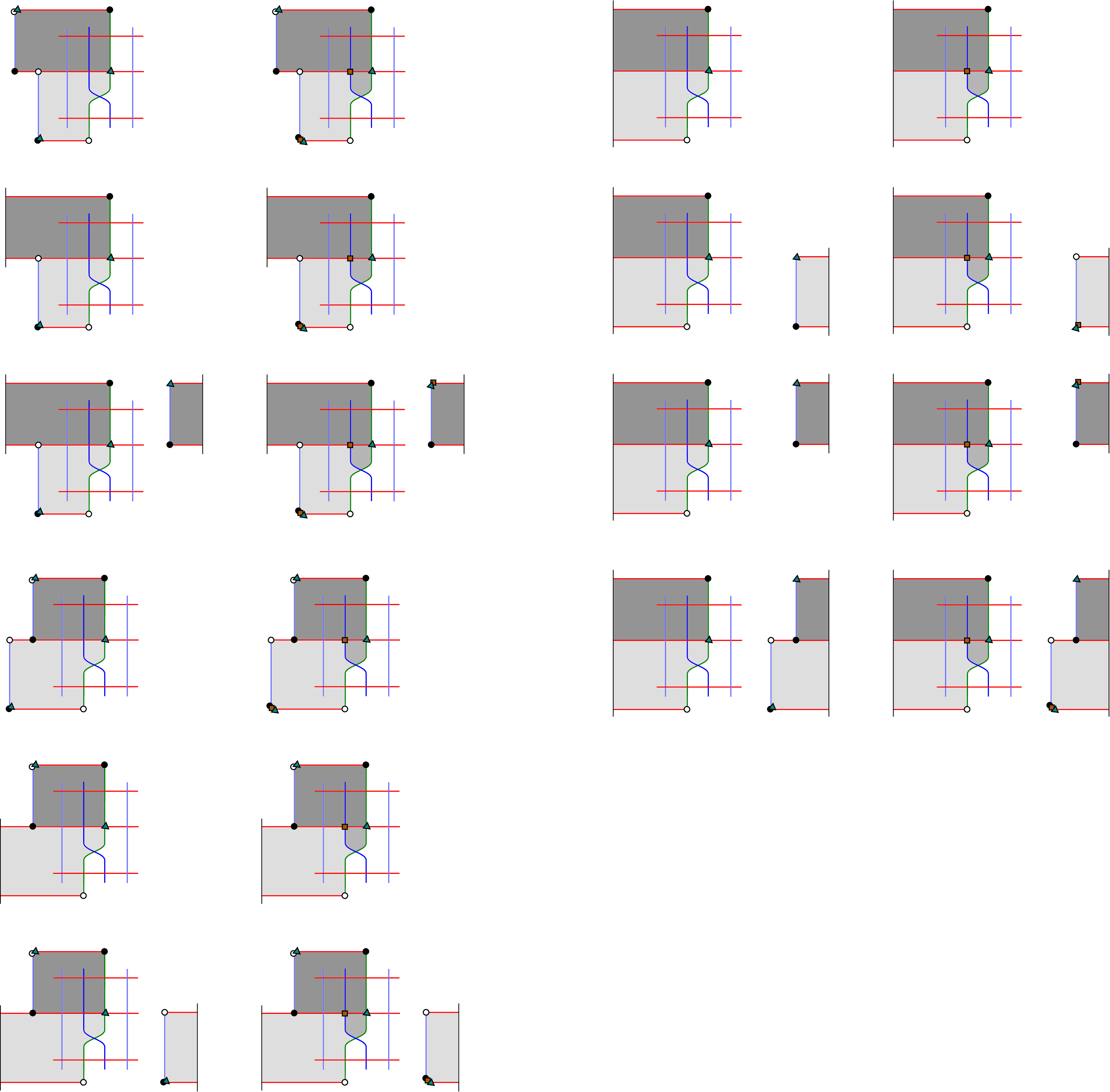}
    \caption{Corresponding domains that contribute to the two sides of Equation~\ref{eq:htpy1}, analogous to the second row, left, of \cite[Figure~9.4]{OSSbook}.}
    \label{fig:hx2-hy-2}
  \end{figure}

  \begin{figure}
    \centering
    \includegraphics[scale = .2]{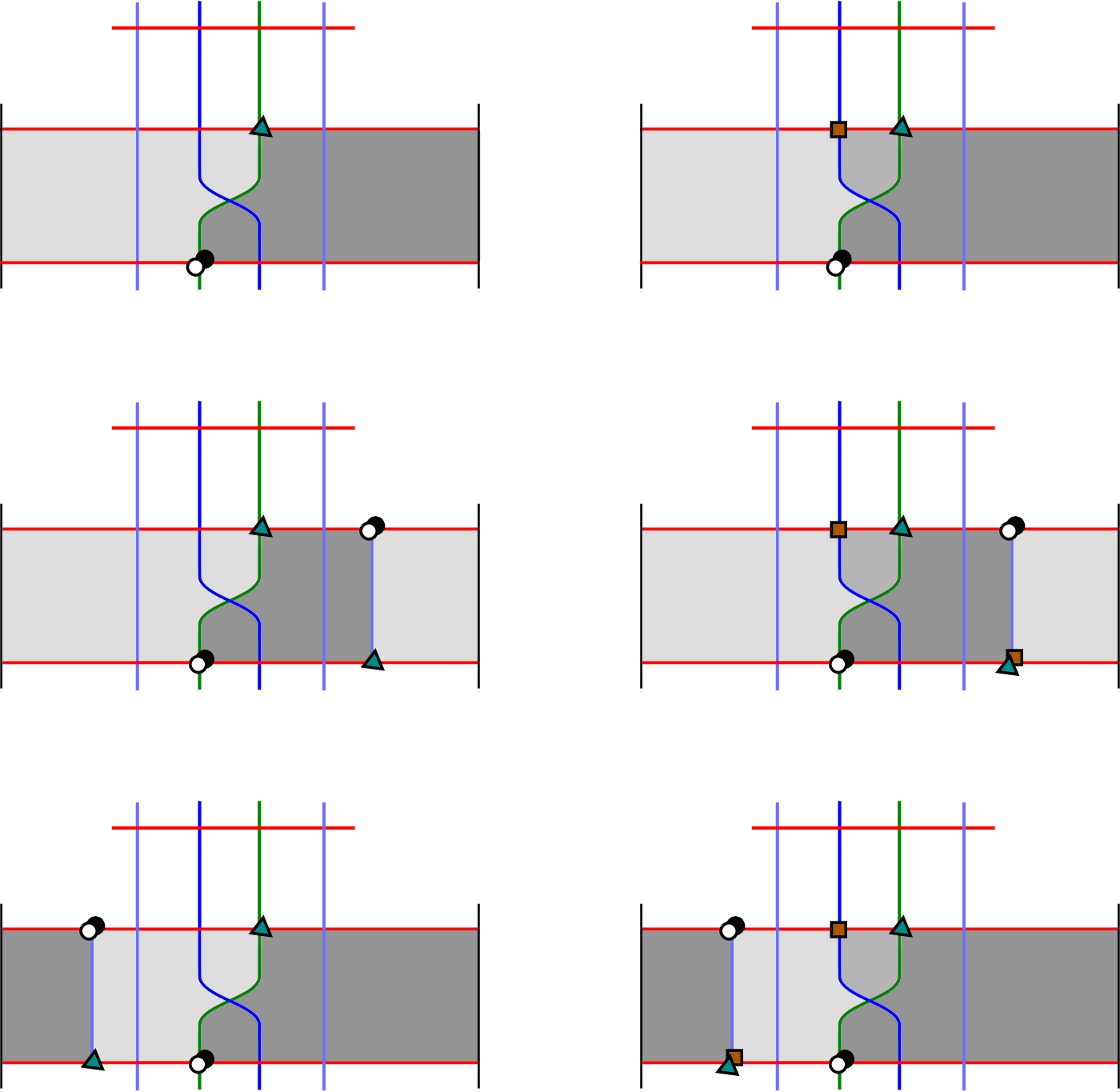}
    \caption{Corresponding domains that contribute to the two sides of Equation~\ref{eq:htpy1}, analogous to the third row, left, of \cite[Figure~9.4]{OSSbook}.}
    \label{fig:hx2-hy-3}
  \end{figure}

  \begin{figure}
    \centering
    \includegraphics[scale = .2]{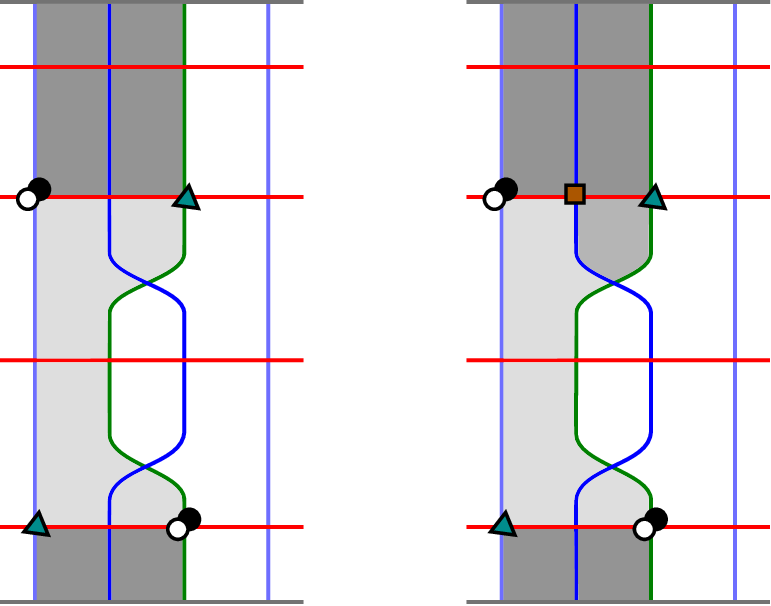}
    \caption{Corresponding domains that contribute to the two sides of Equation~\ref{eq:htpy1}, analogous to the last row, left, of \cite[Figure~9.4]{OSSbook}.}
    \label{fig:hx2-hy-4}
  \end{figure}

  The proof of Equation~\ref{eq:htpy2} follows similarly, as in the proof of \cite[Lemma~9.2.7]{OSSbook}. All subcases are exhibited in Figures~\ref{fig:hy-hx2-1}, \ref{fig:hy-hx2-2}, \ref{fig:hy-hx2-3}, and \ref{fig:hy-hx2-4}, each corresponding to a row in the right of \cite[Figure~9.4]{OSSbook}. In these four figures, each pair of domains is organized as follows: The left represents a domain contributing to $h_Y \circ h_{X_2}$, and the right the corresponding domain contributing to $\PP \circ \delta_{\II, \NN}^1 \circ \TT$; the shading is as before.

  \begin{figure}
    \centering
    \includegraphics[scale = .2]{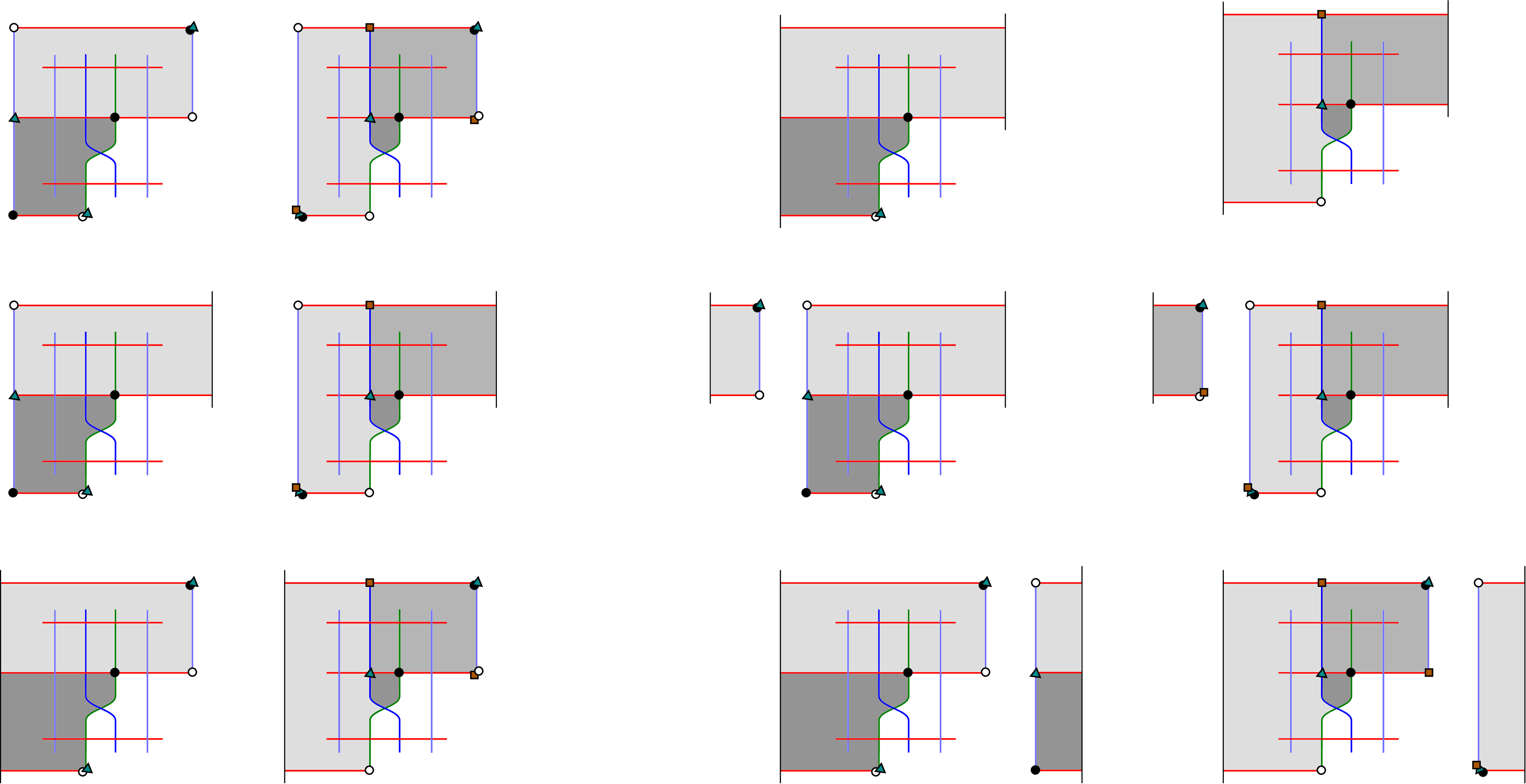}
    \caption{Corresponding domains that contribute to the two sides of Equation~\ref{eq:htpy2}, analogous to the first row, right, of \cite[Figure~9.4]{OSSbook}.}
    \label{fig:hy-hx2-1}
  \end{figure}

  \begin{figure}
    \centering
    \includegraphics[scale = .2]{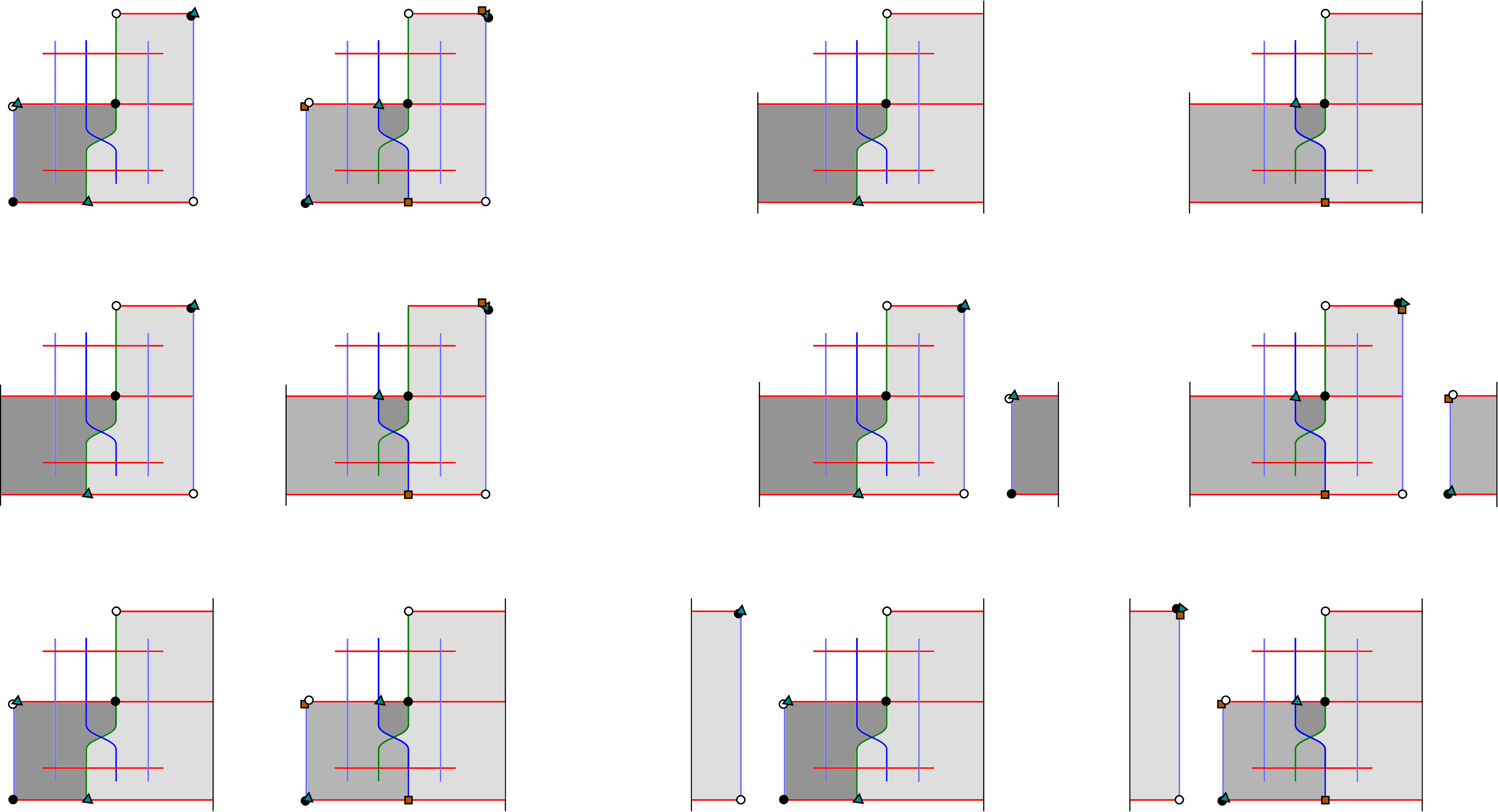}
    \caption{Corresponding domains that contribute to the two sides of Equation~\ref{eq:htpy2}, analogous to the second row, right, of \cite[Figure~9.4]{OSSbook}.}
    \label{fig:hy-hx2-2}
  \end{figure}

  \begin{figure}
    \centering
    \includegraphics[scale = .2]{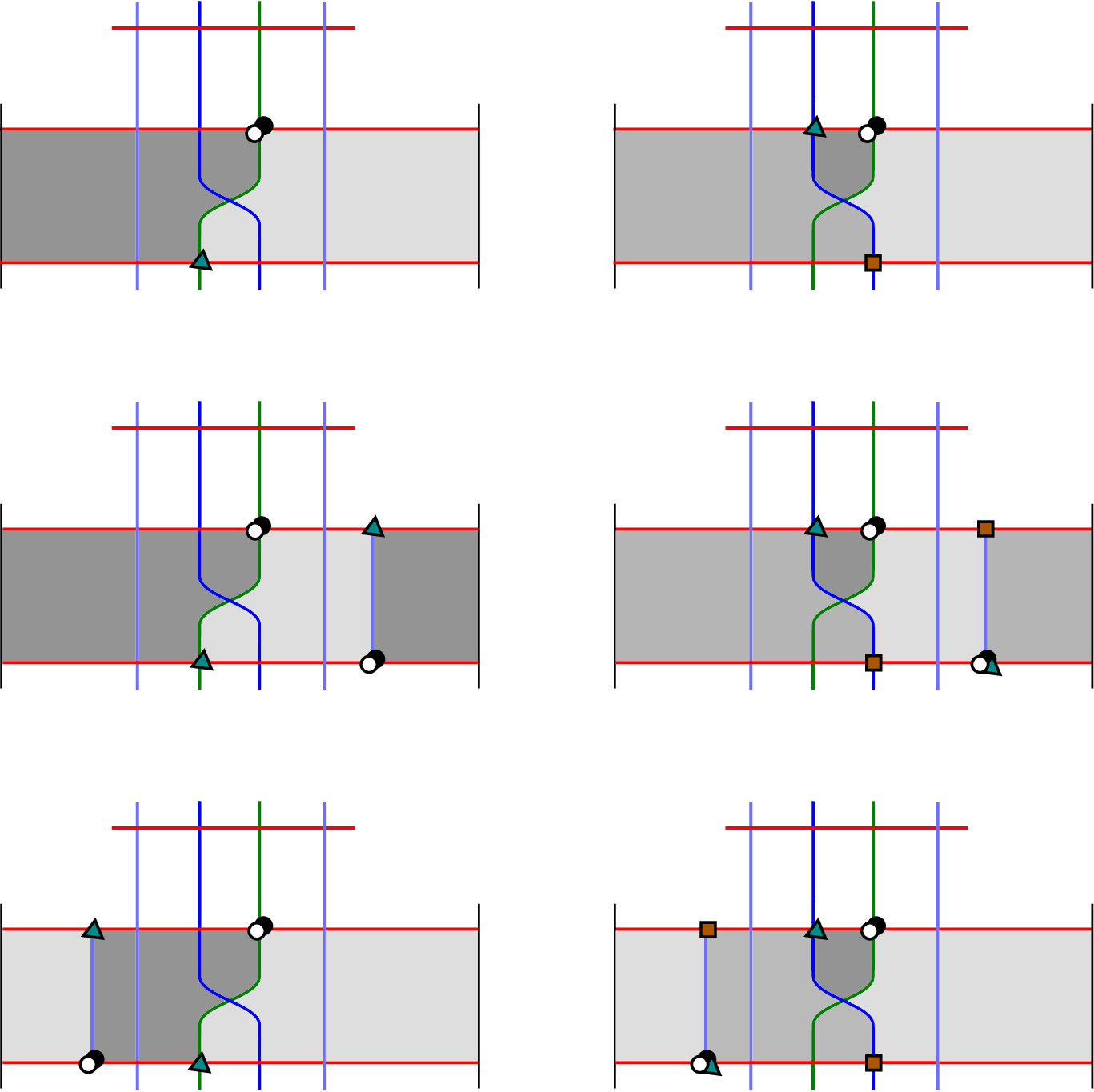}
    \caption{Corresponding domains that contribute to the two sides of Equation~\ref{eq:htpy2}, analogous to the third row, right, of \cite[Figure~9.4]{OSSbook}.}
    \label{fig:hy-hx2-3}
  \end{figure}

  \begin{figure}
    \centering
    \includegraphics[scale = .2]{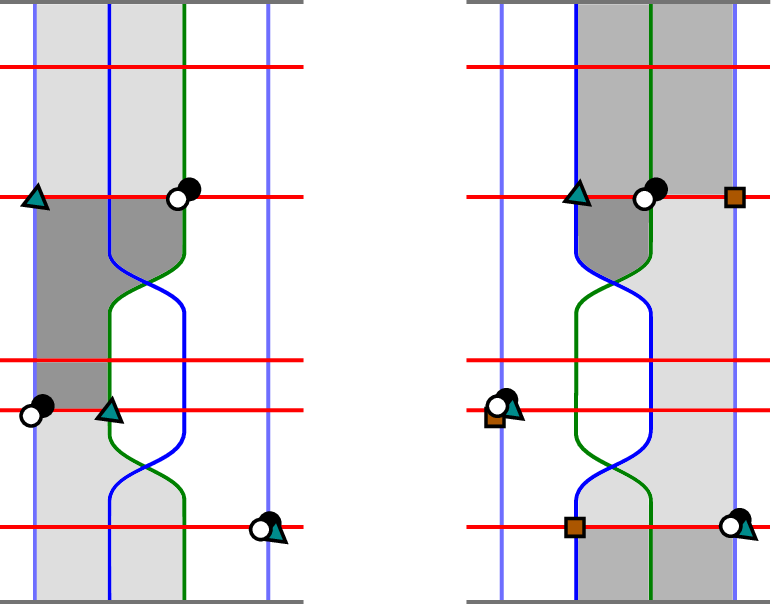}
    \caption{Corresponding domains that contribute to the two sides of Equation~\ref{eq:htpy2}, analogous to the last row, right, of \cite[Figure~9.4]{OSSbook}.}
    \label{fig:hy-hx2-4}
  \end{figure}

  In all subcases except one, the domain contributing to the left of Equation~\ref{eq:htpy1} or Equation~\ref{eq:htpy2} is the same one contributing to the right. The only exception is shown in Figure~\ref{fig:hy-hx2-4}, where the domains contributing to the two sides are both annuli of width one, but one appears to the left and one to the right of the local picture; the intersection of either domain with $\OO$ is $\set{O_4}$, and so their contributions are the same.

  Lemma~\ref{lem:hom_alg_2} applied to Equation~\ref{eq:vert_is_htpy} implies that $\Cone (f_1) \simeq \Cone (f_2)$.
\end{proof}

\begin{lem}
  \label{lem:h_is_u}
  The type $\DD$ homomorphisms $h_{X_2} \circ h_Y + h_Y \circ h_{X_2}$ and $\Id_{ \CDTDm (\HD_0')} \otimes(U_2 - U_1)$ are homotopic, and so there is a homotopy equivalence
  \[
    \Cone (h_{X_2} \circ h_Y + h_Y \circ h_{X_2})\simeq \Cone(\Id_{ \CDTDm (\HD_0')} \otimes(U_2 - U_1))
  \]
\end{lem}

\begin{proof}
  For $i = 1, 2$, let $h_{X_2,Y_i} \colon \CDTDm (\HD_0') \to \CDTDm (\HD_0')$ be the morphism defined by
  \[
    h_{X_2,Y_i} (\x) = \sum_{y \in \genset (\HD_0')} \sum_{\substack{r \in \eRect (\x, \y) \\ r \cap \YY = \set{Y_i} \\ X_2\in r}} \algl{r} \otimes U^r \y \otimes \algr{r}.
  \]
  Let $h_{X_2,Y} = h_{X_2, Y_1} + h_{X_2, Y_2}$. One can check that $h_{X_2,Y}$ has degree $(-1,-1)$; further, $h_{X_2,Y}$ sends $\NN'$ to $\NN'$ and vanishes on $\II'$. 

  Observe that on $\NN'$ the moprhism $h_{X_2} \circ h_Y$ decomposes as
  \[
    h_{X_2} \circ h_Y =  \delta_{\NN', \NN'}^1\circ h_{X_2,Y} + h_{X_2,Y}\circ \delta_{\NN', \NN'}^1 + d  h_{X_2,Y}  + \Id_{\NN'} \otimes U_2 + \Id_{\NN'} \otimes U_4.
  \]
  The map $\Id_{\NN'} \otimes U_2$ (resp.\ $\Id_{\NN'} \otimes  U_4$) comes from juxtapositions of rectangles such that the support of the resulting domain is the component of $\Sigma \setminus \alphas$ containing $X_2$ (resp.\ the annular component of $\Sigma \setminus \betas_{-}$ containing $X_2$). 

  Similarly, on $\II'$ the morphism $h_Y \circ h_{X_2}$ decomposes as
  \[
    h_Y \circ h_{X_2} = h_{X_2,Y}\circ \delta_{\II', \NN'}^1+ \Id_{\II'} \otimes U_2 + \Id_{\II'} \otimes U_4.
  \]

  Adding up the two identities, we obtain
  \[
    h_{X_2} \circ h_Y + h_Y \circ h_{X_2} = \partial h_{X_2,Y} + \Id_{ \CDTDm (\HD_0')} \otimes(U_2+U_4),
  \]
  i.e.\ $h_{X_2,Y}$ is a homotopy between $h_{X_2} \circ h_Y + h_Y \circ h_{X_2}$ and $\Id_{ \CDTDm (\HD_0')} \otimes(U_2+U_4)$. By \cite[Lemma~3.35]{pv}, the actions of $U_1$ and $U_4$ are homotopic, and so the statement of the lemma follows.
\end{proof}

\begin{proof}[Proof of Theorem~\ref{thm:oriented}]
  Define the bigraded type~$\DD$ homomorphism $\Ppm \colon \CDTDm (\HD_+) \to \CDTDm (\HD_-)$ by the two horizontal maps in \eqref{eq:comm_diag}; then by definition, $\Cone (\Ppm)$ is identified with the cone of the commutative diagram \eqref{eq:comm_diag} itself, without degree shifts. On the other hand, the cone of \eqref{eq:comm_diag} is also identified with $\Cone (\delta_{\II, \NN}^1 \circ \TT, \TT \circ \delta_{\NN', \II'}^1)$. Combining this fact with Lemmas~\ref{lem:vert_is_htpy} and \ref{lem:h_is_u}, we have
  \[
    \Cone (\Ppm) = \Cone (\delta_{\II, \NN}^1 \circ \TT, \TT \circ \delta_{\NN', \II'}^1) \simeq \Cone (\Id_{\CDTDm (\HD_0')} \otimes (U_2 - U_1)) [1] \cbrac{\frac{1}{2}}.
  \]
  Finally, there is an obvious homotopy equivalence
  \[
    \Cone (\Id_{\CDTDm (\HD_0')} \otimes (U_2 - U_1)) \simeq \Cone (\Id_{\CDTDm (\HD_0)} \otimes (U_2 - U_1))
  \]
  induced by $\PP$ and $\PP'$, which completes the proof.
\end{proof}

\begin{proof}[Proof of Theorem~\ref{thm:oriented-tilde}]
  Note that $\Id_{\CDTDm (\HD_0)} \otimes (U_2 - U_1)$ has $(M, A)$-degree $(-2, -1)$, and so its mapping cone has underlying module $\CDTDm (\HD_0) [-1] \cbrac{-1} \oplus \CDTDm (\HD_0) [0] \cbrac{0}$. Setting all $U_i$ to zero in Theorem~\ref{thm:oriented}, we obtain a homomorphism $\Ppmt$ such that
  \[
    \Cone (\Ppmt) \simeq \Cone (\Id_{\CDTDt (\HD_0)} \otimes 0) [1] \cbrac{\frac{1}{2}},
  \]
  but the latter mapping cone is simply
  \[
    \CDTDt (\HD_0) [0] \cbrac{-\frac{1}{2}} \oplus \CDTDt (\HD_0) [1] \cbrac{\frac{1}{2}}.
  \]
  A gluing argument as in the proof of Theorem~\ref{thm:ourtheorem} extends this result to any oriented skein triple.
\end{proof}

One immediately recovers the closed case. We outline the proof below. 

\begin{proof}[Proof of Corollary~\ref{cor:hfk-oriented}]
  The proof is analogous to that of Corollary~\ref{cor:hfk}. Theorems~\ref{thm:oriented} and \ref{thm:oriented-tilde} together with the gluing property of tangle Floer homology (see \cite[Theorem~6.1]{pv} for ``minus'') imply the existence of a chain map $F_{+,-} \colon \CFKm (L_+) \to \CFKm (L_-)$ whose mapping cone corresponds to $\CFKm(L_0)$, so that the exact triangles on homology follow.
\end{proof}

Finally, we prove Corollary~\ref{cor:gl11}.

\begin{proof}[Proof of Corollary~\ref{cor:gl11}]
  By the type~$\DA$ version of Theorem~\ref{thm:oriented-tilde}, we have a type~$\DA$ homomorphism $\Ppmt  \colon \CDTAt (\eT_+) \to \CDTAt (\eT_-)$ of $(M, A)$-degree $(0,0)$ and a homotopy equivalence
  \[
    \Cone (\Ppmt) \simeq \CDTAt (\eT_0) \sqbrac{0} \cbrac{-\frac{1}{2}} \oplus \CDTAt (\eT_0)\sqbrac{1} \cbrac{\frac{1}{2}}.
  \]
  By \cite{epv}, we have homomorphisms on Grothendieck groups
  \[  \sqbrac{\Cone ( \Ppmt )} = \sqbrac{\CDTAt (\eT_+)\sqbrac{1}\oplus \CDTAt (\eT_-)} = -Q(\eT_+) + Q(\eT_-)
  \]
  and
  \begin{align*}
    \sqbrac{ \CDTAt (\eT_0) \sqbrac{0} \cbrac{-\frac{1}{2}} \oplus \CDTAt (\eT_0)\sqbrac{1} \cbrac{\frac{1}{2}}}
    & = q^{-1} \sqbrac{\CDTAt (\eT_0)} - q \sqbrac{ \CDTAt (\eT_0)}\\
    & = q^{-1} Q(\eT_0) - q Q(\eT_0),
  \end{align*}
  where we have used the fact that the Maslov grading descends to powers of $-1$, and twice the Alexander grading descends to powers of $q$ in the Grothendieck group. This completes the proof.
\end{proof}

%%%%%%%%%%%%%%%%%%%%%%%%%%%%%%%%%%%%%%%%%%%%%%%%%%%%%%%

\bibliographystyle{hamsplain2}
\bibliography{master}

\end{document}